\documentclass[11pt]{article}
\usepackage{fullpage}
\usepackage{math-nate}
\usepackage{multirow}
\usepackage[numbers]{natbib}
\usepackage[all]{xy}
\usepackage{hyperref}
\usepackage{url}

\numberwithin{equation}{section}

\newcommand{\nm}{\norm{\cdot}}
\newcommand{\inner}[2]{\langle #1 , #2 \rangle}
\newcommand{\tr}{\operatorname{tr}}
\newcommand{\dom}{\operatorname{dom}}
\newcommand{\grad}{\nabla}
\newcommand{\mathquote}[1]{\text{`` $ #1 $ ''}}

\begin{document}
\title{Analysis and Probability on Infinite-Dimensional Spaces}
\author{Nathaniel Eldredge}
\maketitle

\section*{Preface}

I wrote these lecture notes for a graduate topics course I taught at
Cornell University in Fall 2011 (Math 7770).  The ostensible primary
goal of the course was for the students to learn some of the
fundamental results and techniques in the study of probability on
infinite-dimensional spaces, particularly Gaussian measures on Banach
spaces (also known as abstract Wiener spaces).  As others who have
taught such courses will understand, a nontrivial secondary goal of
the course was for the instructor (i.e., me) to do the same.  These
notes only scratch the very surface of the subject, but I tried to use them
to work through some of the basics and see how they fit together into
a bigger picture.  In addition to theorems and proofs, I've left in
some more informal discussions that attempt to develop intuition.

Most of the material here comes from the books
\cite{kuo-gaussian-book, nualart, bogachev-gaussian-book}, and the
lecture notes prepared by Bruce Driver for the 2010 Cornell
Probability Summer School \cite{driver-cpss-notes,
  driver-probability}.  If you are looking to learn more, these are
great places to look.\footnote{In reading these notes in conjunction
  with \cite{nualart}, you should identify Nualart's abstract
  probability space $(\Omega, \mathcal{F}, P)$ with our Banach space
  $(W, \mathcal{B}, \mu)$.  His ``Gaussian process'' $h \mapsto W(h)$
  should be viewed as corresponding to our map $T$ defined in Section
  \ref{sec-cameron-martin}; his indexing Hilbert space $H$ may be
  identified with the Cameron--Martin space, and his $W(h)$ is the
  random variable, defined on the Banach space, that we have denoted
  by $Th$ or $\langle h, \cdot \rangle$. There's a general principle
  in this area that all the ``action'' takes place on the
  Cameron--Martin space, so one doesn't really lose much by dropping
  the Banach space structure on the space $W$ and replacing it with an
  generic $\Omega$ (and moreover generality is gained). Nonetheless, I
  found it helpful in building intuition to work on a concrete space
  $W$; this also gives one the opportunity to explore how the
  topologies of $W$ and $H$ interact.}

Any text marked \textbf{Question N} is something that I found myself
wondering while writing this, but didn't ever resolve.  I'm not
proposing them as open problems; the answers could be
well-known, just not by me.  If you know the answer to any of them,
I'd be happy to hear about it!   There are also
still a few places where proofs are rough or have some gaps that I
never got around to filling in.

On the other hand, something marked
\textbf{Exercise N} is really meant as an exercise.  

I would like to take this opportunity to thank the graduate students
who attended the course.  These notes were much improved by their
questions and contributions.  I'd also like to thank several
colleagues who sat in on the course or otherwise contributed to these
notes, particularly Clinton Conley, Bruce Driver, Leonard Gross, Ambar
Sengupta, and Benjamin Steinhurst.  Obviously, the many deficiencies
in these notes are my responsibility and not theirs. 

Questions and comments on these notes are most welcome.  I am now at
the University of Northern Colorado, and you can email me at
\verb|neldredge@unco.edu|.

\section{Introduction}

\subsection{Why analysis and probability on $\mathbb{R}^n$ is nice}

Classically, real analysis is usually based on the study of real-valued
functions on finite-dimensional Euclidean space $\R^n$, and operations
on those functions involving limits, differentiation, and
integration.  Why is $\R^n$ such a nice space for this theory?
\begin{itemize}
\item $\R^n$ is a nice topological space, so limits behave well.
  Specifically, it is a complete separable metric space, and it's
  locally compact.
  
\item $\R^n$  has a nice algebraic structure: it's a vector space, so
  translation and scaling make sense.  This is where differentiation
  comes in: the derivative of a function just measures how it changes
  under infinitesimal translation.

\item $\R^n$ has a natural measure space structure; namely, Lebesgue
  measure $m$ on the Borel $\sigma$-algebra.  The most important property
  of Lebesgue measure is that it is invariant under translation.  This
  leads to nice interactions between differentiation and integration,
  such as integration by parts, and it gives nice functional-analytic
  properties to differentiation operators: for instance, the Laplacian
  $\Delta$ is a self-adjoint operator on the Hilbert space $L^2(\R^n, m)$.
\end{itemize}

Of course, a lot of analysis only involves local properties, and so it
can be done on spaces that are locally like $\R^n$: e.g. manifolds.
Let's set this idea aside for now.

\subsection{Why infinite-dimensional spaces might be less nice}

The fundamental idea in this course will be: how can we do analysis
when we replace $\R^n$ by an infinite dimensional space?  First we
should ask: what sort of space should we use?  Separable Banach spaces seem to
be nice.  They have a nice topology (complete separable metric spaces)
and are vector spaces.  But what's missing is Lebesgue measure.
Specifically:

\begin{theorem}
  ``There is no infinite-dimensional Lebesgue measure.''  Let $W$ be
  an infinite-dimensional separable Banach space.  There does not exist a
  translation-invariant Borel measure on $W$ which assigns
  positive finite measure to open balls.  In fact, any
  translation-invariant Borel measure $m$ on $W$ is either the zero measure or
  assigns infinite measure to every open set.
\end{theorem}

\begin{proof}
  Essentially, the problem is that inside any ball $B(x,r)$, one can
  find infinitely many disjoint balls $B(x_i, s)$ of some fixed smaller
  radius $s$.  By translation invariance, all the $B(x_i, s)$ have the
  same measure.  If that measure is positive, then $m(B(x,r)) =
  \infty$.  If that measure is zero, then we observe that $W$ can be
  covered by countably many balls of radius $s$ (by separability) and
  so $m$ is the zero measure.

  The first sentence is essentially Riesz's lemma: given any proper
  closed subspace $E$, one can find a point $x$ with $||x|| \le 1$ and
  $d(x,E) > 1/2$.  (Start by picking any $y \notin E$, so that $d(y,E)
  > 0$; then by
  definition there is an $z \in E$ with $d(y,z) < 2d(y,E)$.  Now look
  at $y-z$ and rescale as needed.)  Now let's look at $B(0,2)$ for
  concreteness.  Construct $x_1, x_2, \dots$ inductively by letting
  $E_n = \spanop\{x_1, \dots, x_n\}$ (which is closed) and choosing
  $x_{n+1}$ as in Riesz's lemma with $||x_{n+1}|| \le 1$ and
  $d(x_{n+1}, E_n) > 1/2$.  In particular, $d(x_{n+1}, x_i) > 1/2$ for
  $i \le n$.  Since our space is infinite dimensional, the
  finite-dimensional subspaces $E_n$ are always proper and the
  induction can continue, producing a sequence $\{x_i\}$ with $d(x_i,
  x_j) > 1/2$ for $i \ne j$, and thus the balls $B(x_i, 1/4)$ are
  pairwise disjoint.
\end{proof}

\begin{exercise}
  Prove the above theorem for $W$ an infinite-dimensional Hausdorff
  topological vector space.  (Do we need separability?)
\end{exercise}

A more intuitive idea why infinite-dimensional Lebesgue measure can't
exist comes from considering the effect of scaling.  In $\R^n$, the
measure of $B(0,2)$ is $2^n$ times larger than $B(0,1)$.  When $n =
\infty$ this suggests that one cannot get sensible numbers for the
measures of balls.

There are nontrivial translation-invariant Borel measures on
infinite-dimensional spaces: for instance, counting measure.  But
these measures are useless for analysis since they cannot say anything
helpful about open sets.

So we are going to have to give up on translation invariance, at least
for now.  Later, as it turns out, we will study some measures that
recover a little bit of this: they are \emph{quasi}-invariant under
\emph{some} translations.  This will be explained in due course.

\subsection{Probability measures in infinite dimensions}

If we just wanted to think about Borel measures on
infinite-dimensional topological vector spaces, we actually have lots
of examples from probability, that we deal with every day.

\begin{example}\label{product-example}
  Let $(\Omega, \mathcal{F}, \mathbb{P})$ be a probability space and $X_1, X_2,
  \dots$ a sequence of random variables.  Consider the infinite
  product space $\R^\infty$, thought of as the space of all sequences
  $\{x(i)\}_{i=1}^\infty$ of real numbers.  This is a topological
  vector space when equipped with its product topology.  We can equip
  $\R^\infty$ with its Borel $\sigma$-algebra, which is the same as
  the product Borel $\sigma$-algebra (verify).  Then the map from
  $\Omega$ to $\R^\infty$ which sends $\omega$ to the sequence $x(i) =
  X_i(\omega)$ is measurable.  The pushforward of
  $\mathbb{P}$ under this map gives a Borel probability measure $\mu$ on
  $\R^\infty$.

  The Kolmogorov extension theorem guarantees lots of choices for the
  joint distribution of the $X_i$, and hence lots of probability
  measures $\mu$.  Perhaps the simplest interesting case is when the
  $X_i$ are iid with distribution $\nu$, in which case $\mu$ is the
  infinite product measure $\mu = \prod_{i=1}^\infty \nu$.  Note that
  in general one can only take the infinite product of
  \emph{probability} measures (essentially because the only number $a$
  with $0 < \prod_{i=1}^\infty a < \infty$ is $a=1$).
\end{example}

\begin{example}
  Let $(\Omega, \mathcal{F}, \mathbb{P})$ be a probability space, and
  $\{X_t : 0 \le t \le 1\}$ be any stochastic process.  We could play
  the same game as before, getting a probability measure on
  $\R^{[0,1]}$ (with its product $\sigma$-algebra).  This case is not
  as pleasant because nothing is countable.  In particular, the Borel
  $\sigma$-algebra generated by the product topology is not the same
  as the product $\sigma$-algebra (exercise: verify this, perhaps by
  showing that the latter does not contain singleton sets.)  Also, the
  product topology on $\R^{[0,1]}$ is rather nasty; for example it is
  not first countable.  (In contrast, $\R^\infty$ with its product
  topology is actually a Polish space.)  So we will avoid examples like
  this one.
\end{example}

\begin{example}
  As before, but now assume $\{X_t : 0 \le t \le 1\}$ is a
  \emph{continuous} stochastic process.  We can then map $\Omega$ into
  the Banach space $C([0,1])$ in the natural way, by sending $\omega$
  to the continuous function $X_\cdot(\omega)$.  One can check that
  this map is measurable when $C([0,1])$ is equipped with its Borel
  $\sigma$-algebra.  (Hint: $||x|| \le 1$ if and only if $|x(t)| \le
  1$ for all $t$ in a countable dense subset of $[0,1]$.)  So by
  pushing forward $\mathbb{P}$ we get a Borel probability measure on
  $C([0,1])$.  For example, if $X_t$ is Brownian motion, this is the classical
  Wiener measure.
\end{example}

So probability measures seem more promising.  We are going to
concentrate on Gaussian probability measures.  Let's start by looking
at them in finite dimensions.

\subsection{Gaussian measures in finite dimensions}

In one dimension everyone knows what Gaussian means.  We are going to
require our measures / random variables to be centered (mean zero) to
have fewer letters floating around.  However we are going to include
the degenerate case of zero variance.

\begin{definition}
  A Borel probability measure $\mu$ on $\R$ is \textbf{Gaussian} with variance $\sigma^2$ iff
  \begin{equation*}
    \mu(B) = \int_B \frac{1}{\sqrt{2 \pi} \sigma} e^{-x^2/2\sigma^2}\,dx
  \end{equation*}
  for all Borel sets $B \subset \R$.  We also want to allow the case
  $\sigma = 0$, which corresponds to $\mu = \delta_0$ being a Dirac
  mass at 0.

  We could also specify $\mu$ in terms of its Fourier
  transform (or characteristic function).  The above condition is
  equivalent to having
  \begin{equation*}
    \int_\R e^{i \lambda x} \mu(dx) = e^{-\sigma^2 \lambda^2/2}
  \end{equation*}
  for all $\lambda \in \R$.  (Note $\sigma = 0$ is naturally included
  in this formulation.)

  A random variable $X$ on some probability space $(\Omega,
  \mathcal{F}, \mathbb{P})$ is Gaussian with variance $\sigma^2$ if its
  distribution measure is Gaussian, i.e.
  \begin{align*}
    \mathbb{P}(X \in B) = \int_B \frac{1}{\sqrt{2 \pi} \sigma} e^{-x^2/2\sigma^2}\,dx.
  \end{align*}
  for all Borel sets $B$.  For $\sigma = 0$ we have the constant
  random variable $X = 0$.  Equivalently
  \begin{equation*}
    \mathbb{E}[e^{i \lambda X}] = e^{- \sigma^2 \lambda^2/2}
  \end{equation*}
  for all $\lambda \in \R$.
\end{definition}

Let's make a trivial observation: $\mu$ is not translation invariant.
However, translation doesn't mess it up completely.

\begin{notation}
  Let $\mu$ be a measure on a vector space $W$.  For $y \in W$, we
  denote by $\mu_y$ the translated measure defined by $\mu_y(A) =
  \mu(A-y)$.  In other words, $\int_W f(x) \mu_y(dx) = \int_W f(x+y)$.
\end{notation}

\begin{exercise}
  Check that I didn't screw up the signs in the previous paragraph.
\end{exercise}

\begin{definition}
  A measure $\mu$ on a vector space $W$ is said to be
  \textbf{quasi-invariant} under translation by $y \in W$ if the
  measures $\mu, \mu_y$ are mutually absolutely continuous (or
  \textbf{equivalent}); that is, if $\mu(A)=0 \Leftrightarrow \mu_y(A)
  = 0$ for measurable sets $A \subset W$.
\end{definition}

Intuitively, quasi-invariance means that translation can change the
measure of a set, but it doesn't change whether or not the measure is
zero.

One way I like to think about equivalent measures is with the
following baby example.  Suppose I have two dice which look identical
on the surface, but one of them is fair, and the other produces
numbers according to the distribution $(0.1,0.1,0.1,0.1,0.1,0.5)$
(i.e. it comes up $6$ half the time).  (Note that they induce
equivalent measures on $\{1,2,3,4,5,6\}$: in both cases the only set
of measure zero is the empty set.)  I pick one of the dice and ask you
to determine which one it is.  If you roll a lot of 6s, you will have
a strong suspicion that it's the unfair die, but you can't absolutely
rule out the possibility that it's the fair die and you are just
unlucky.

On the other hand, suppose one of my dice always comes up even, and
the other always comes up odd.  In this case the induced measures are
mutually singular: there is a set (namely $\{1,3,5\}$) to which one
gives measure 0 and the other gives measure 1.  If I give you one of
these dice, then all you have to do is roll it once and see whether
the number is even or odd, and you can be (almost) sure which die you
have.

For Gaussian measures on $\R$, note that if $\sigma \ne 0$, then $\mu$
is quasi-invariant under translation by any $y \in \R$.  This is a
trivial fact: both $\mu$ and $\mu_y$ have positive densities with
respect to Lebesgue measure $m$, so $\mu(A) = 0$ iff $\mu_y(A) = 0$
iff $m(A) = 0$.  We'll also note that, as absolutely continuous
measures, they have a Radon-Nikodym derivative, which we can compute
just by dividing the densities:
\begin{align*}
  \frac{d \mu_y}{d \mu}(x) &= \frac{\frac{1}{\sqrt{2 \pi} \sigma} e^{-(x-y)^2/2\sigma}}
{\frac{1}{\sqrt{2 \pi} \sigma} e^{-x^2/2\sigma^2}} \\
&= e^{-\frac{y^2}{2\sigma^2} + \frac{xy}{\sigma^2}}.
\end{align*}
Just pay attention to the form of this expression, as we will see it
again later.

On the other hand, in the degenerate case $\sigma = 0$, then $\mu =
\delta_0$ and $\mu_y = \delta_y$ are mutually singular.

Now let's look at the $n$-dimensional case.

\begin{definition}
  An $n$-dimensional random vector $\mathbf{X} = (X_1, \dots, X_n)$ is
  \textbf{Gaussian} if and only if $\mathbf{\lambda} \cdot \mathbf{X} := \sum
  \lambda_i X_i$ is a Gaussian random variable for all
  $\mathbf{\lambda} = (\lambda_1, \dots, \lambda_n) \in \R^n$.
\end{definition}

Or in terms of measures:

\begin{definition}
  Let $\mu$ be a Borel probability measure on $\R^n$.  For each
  $\mathbf{\lambda} \in \R^n$, we can think of the map $\R^n \ni \mathbf{x}
  \mapsto \mathbf{\lambda} \cdot \mathbf{x} \in \R$ as a random
  variable on the probability space $(\R^n, \mathcal{B}_{\R^n},
  \mu)$.  $\mu$ is \textbf{Gaussian} if and only if this random
  variable is Gaussian for each $\mathbf{\lambda}$.
\end{definition}

Of course, we know that the distribution of $\mathbf{X}$ is uniquely
determined by its $n \times n$ covariance matrix $\Sigma$, where
$\Sigma_{ij} = \Cov(X_i, X_j)$.  Note that $\Sigma$ is clearly
symmetric and positive semidefinite.  Furthermore, any symmetric,
positive semidefinite matrix $\Sigma$ can arise as a covariance
matrix: let $\mathbf{Z} = (Z_1, \dots, Z_n)$ where $Z_i$ are iid
Gaussian with variance $1$, and set $\mathbf{X} = \Sigma^{1/2} \mathbf{Z}$. 

A consequence of this is that if $(X_1, \dots, X_n)$ has a joint
Gaussian distribution, then the $X_i$ are independent if and only if
they are uncorrelated (i.e. $\Cov(X_i, X_j) = 0$ for $i \ne j$,
i.e. $\Sigma$ is diagonal).  Note that this \emph{fails} if all we
know is that each $X_i$ has a Gaussian distribution.

\begin{proposition}
  $\mathbf{X}$ is Gaussian if and only if it has characteristic
  function
  \begin{equation*}
    \mathbb{E}[e^{i \mathbf{\lambda} \cdot \mathbf{X}}] = e^{- \frac{1}{2}    \mathbf{\lambda} \cdot \Sigma \mathbf{\lambda}}
  \end{equation*}
  where $\Sigma$ is the covariance matrix of $\mathbf{X}$.
\end{proposition}

Or, in terms of measures:

\begin{proposition}
  A probability measure $\mu$ on $\R^n$ is Gaussian if and only if
  \begin{equation*}
    \int_{\R^n} e^{i \mathbf{\lambda} \cdot \mathbf{x}} \mu(d\mathbf{x}) = e^{-
    \frac{1}{2}\mathbf{\lambda} \cdot \Sigma \mathbf{\lambda}}
  \end{equation*}
  for some $n \times n$ matrix $\Sigma$, which is necessarily positive
  semidefinite and can be chosen symmetric.
\end{proposition}

We would like to work more abstractly and basis-free, in preparation
for the move to infinite dimensions.  The map $\mathbf{x} \mapsto
\mathbf{\lambda} \cdot \mathbf{x}$ is really just a linear functional
on $\R^n$.  So let's write:

\begin{definition}\label{finite-gaussian-good}
  Let $\mu$ be a Borel probability measure on a finite-dimensional
  topological vector space $W$.  Then each $f \in W^*$ can be seen as a random
  variable on the probability space $(W, \mathcal{B}_W, \mu)$.  $\mu$
  is \textbf{Gaussian} if and only if, for each $f \in W^*$, this
  random variable is Gaussian.  

  Equivalently, $\mu$ is Gaussian iff
  the pushforward $\mu \circ f^{-1}$ is a Gaussian measure on $\R$ for
  each $f \in W^*$.
\end{definition}

Of course we have not done anything here because a finite-dimensional
topological vector space $W$ is just some $\R^n$ with
its usual topology.  Every linear functional $f \in W^*$ is of the
form $f(\mathbf{x}) = \mathbf{\lambda} \cdot \mathbf{x}$, and all such
linear functionals are continuous, hence measurable.

If $f, g \in W^*$ are thought of as Gaussian random variables on $(W,
\mu)$, then $q(f,g) = \Cov(f,g)$ is a symmetric, positive semidefinite,
bilinear form on $W^*$.  We'll also write $q(f) = q(f,f) = \Var(f)$.
$q$ is the basis-free analogue of the covariance matrix; we could call
it the covariance form.  As we argued above, in this
finite-dimensional case, any such bilinear form can arise as a
covariance form.

Another way to think about this is that since each $f \in W^*$ is
Gaussian, it is certainly square-integrable, i.e. $\int_W |f(x)|^2
\mu(dx) = E|f|^2 = \Var(f) < \infty$.  So $V^*$ can be thought of as a
subspace of $L^2(V, \mu)$.  Then $q$ is nothing but the restriction of
the $L^2$ inner product to the subspace $W^*$.  

(Technically, $q$ may be degenerate, in which case it is actually the
quotient of $W^*$ by the kernel of $q$ that we identify as a subspace
of $L^2(W, \mu)$.)

\begin{exercise}
  The support of the measure $\mu$ is given by
  \begin{equation*}
    \supp \mu = \bigcap_{q(f,f) = 0} \ker f.
  \end{equation*}
  One could write $\supp \mu = (\ker q)^\perp$.  In particular, if $q$
  is positive definite, then $\mu$ has full support.  (Recall that the
  support of a measure $\mu$ is defined as the smallest closed set
  with full measure.)
\end{exercise}

\begin{exercise}
  The restriction of $\mu$ to its support is a nondegenerate Gaussian
  measure (i.e. the covariance form is positive definite).
\end{exercise}

\begin{exercise}
  $\mu$ is quasi-invariant under translation by $y$ if and only if $y
  \in \supp \mu$.
  If $y \notin \supp \mu$, then $\mu$ and $\mu_y$ are
  mutually singular.  (We'll see that in infinite dimensions, the
  situation is more complex.)
\end{exercise}

In terms of characteristic functions, then, we have

\begin{proposition}
  A Borel probability measure $\mu$ on a finite-dimensional
  topological vector space $W$ is Gaussian if and only if, for each $f
  \in W^*$, we have
  \begin{equation*}
    \int_W e^{i f(x)} \mu(dx) = e^{-\frac{1}{2}q(f,f)}
  \end{equation*}
  where $q$ is some positive semidefinite symmetric bilinear form on $W^*$.
\end{proposition}

\section{Infinite-dimensional Gaussian measures}

Definition \ref{finite-gaussian-good} will generalize pretty
immediately to infinite-dimensional topological vector spaces.  There
is just one problem.  An arbitrary linear functional on a topological
vector space can be nasty; in particular, it need not be Borel
measurable, in which case it doesn't represent a random variable.  But
continuous linear functionals are much nicer, and are Borel measurable
for sure, so we'll restrict our attention to them.\footnote{In the
case of separable Banach spaces, or more generally Polish topological
vector spaces, this sufficient condition is also necessary: a linear
functional is Borel measurable if and only if it is continuous.  Even
the weaker assumption of so-called Baire measurability is sufficient,
in fact.  See 9.C of \cite{kechris}.  So we are not giving up anything by requiring continuity.
Thanks to Clinton Conley for explaining this to me and providing a
reference.  This sort of goes to show that a linear functional on a
separable Banach space is either continuous or really really nasty.}

As usual, $W^*$ will denote the continuous dual of $W$.

\begin{definition}
  Let $W$ be a topological vector space, and $\mu$ a Borel probability
  measure on $W$.  $\mu$ is \textbf{Gaussian} iff, for each continuous
  linear functional $f \in W^*$, the pushforward $\mu \circ f^{-1}$ is a
  Gaussian measure on $\R$, i.e. $f$ is a Gaussian random variable on
  $(W, \mathcal{B}_W, \mu)$.
\end{definition}

As before, we get a covariance form $q$ on $W^*$ where $q(f,g) =
\Cov(f,g)$.  Again, $W^*$ can be identified as a subspace of $L^2(W,
\mu)$, and $q$ is the restriction of the $L^2$ inner product.

\begin{proposition}
  A Borel probability measure $\mu$ on a topological vector space $W$
  is Gaussian if and only if, for each $f \in W^*$, we have
  \begin{equation*}
    \int_W e^{i f(x)} \mu(dx) = e^{-\frac{1}{2}q(f,f)}
  \end{equation*}
  where $q$ is some positive semidefinite symmetric bilinear form on $W^*$.
\end{proposition}

\begin{exercise}
  If $f_1, \dots, f_n$ in $W^*$, then $(f_1, \dots, f_n)$ has a joint
  Gaussian distribution.
\end{exercise}

\begin{exercise}
  Any $q$-orthogonal subset of $W^*$ is an independent set of random
  variables on $(W, \mu)$.
\end{exercise}

For the most part, we will concentrate on the case that $W$ is a
separable Banach space.  But as motivation, we first want to look at a
single case where it isn't.  If you want more detail on the general
theory for topological vector spaces, see Bogachev \cite{bogachev-gaussian-book}.

\section{Motivating example: $\R^\infty$ with product Gaussian
  measure}

As in Example \ref{product-example}, let's take $W = \R^\infty$ with
its product topology.  Let's record some basic facts about this
topological vector space.

\begin{exercise}
  $W$ is a Fr\'echet space (its topology is generated by a countable
  family of seminorms).
\end{exercise}

\begin{exercise}
  $W$ is a Polish space.  (So we are justified in doing all our
  topological arguments with sequences.)
\end{exercise}

\begin{exercise}
  The topology of $W$ does not come from any norm.
\end{exercise}

\begin{exercise}
  The Borel $\sigma$-algebra of $W$ is the same as the product
  $\sigma$-algebra.
\end{exercise}

\begin{exercise}
  Every continuous linear functional $f \in W^*$ is of the form
  \begin{equation*}
    f(x) = \sum_{i=1}^n a_i x(i)
  \end{equation*}
  for some $a_1, \dots, a_n \in \R$.  Thus $W^*$ can be identified
  with $c_{00}$, the set of all real sequences which are eventually zero.
\end{exercise}

Let's write $e_i$ for the element of $W$ with $e_i(j) = \delta_{ij}$,
and $\pi_i$ for the projection onto the $i$'th coordinate $\pi_i(x) =
x(i)$. (Note $\pi_i \in W^*$; indeed they form a basis.)  

As in Example \ref{product-example}, we choose $\mu$ to be an infinite
product of Gaussian measures with variance $1$.  Equivalently, $\mu$
is the distribution of an iid sequence of standard Gaussian random
variables.  So the random variables $\pi_i$ are iid standard Gaussian.

\begin{exercise}
  $\mu$ is a Gaussian measure.  The covariance form of $\mu$ is given
  by
  \begin{equation*}
    q(f,g) = \sum_{i=1}^\infty f(e_i) g(e_i).
  \end{equation*}
  (Note that the sum is actually finite.)
\end{exercise}

\begin{exercise}
  $\mu$ has full support.
\end{exercise}

$q$ is actually positive definite: the only $f \in W^*$ with $q(f,f) =
0$ is $f=0$.  So $W^*$ is an honest subspace of $L^2(W, \mu)$.  It is
not a closed subspace, though; that is, $W^*$ is not complete in the
$q$ inner product.  Let $K$ denote the $L^2(W,\mu)$-closure of $W^*$.

\begin{exercise}\label{Kexercise}
  Show that $K$ consists of all functions $f : W \to \R$ of the form
  \begin{equation}\label{Kdef}
    f(x) = \sum_{i=1}^\infty a_i x(i)
  \end{equation}
  where $\sum_{i=1}^\infty |a_i|^2 < \infty$.  This formula requires
  some explanation.  For an arbitrary $x \in W$, the sum in
  (\ref{Kdef}) may not converge. However, show that it does converge
  for $\mu$-a.e. $x \in W$.  (Hint: Sums of independent random
  variables converge a.s. as soon as they converge in $L^2$; see
  Theorem 2.5.3 of Durrett \cite{durrett}.)  Note well that the measure-1 set on
  which (\ref{Kdef}) converges depends on $f$, and there will not be a
  single measure-1 set where convergence holds for every $f$.
  Moreover, show each $f \in K$ is a Gaussian random variable.

($K$ is isomorphic to $\ell^2$; this should make sense,
  since it is the completion of $W^* = c_{00}$ in the $q$ inner
  product, which is really the $\ell^2$ inner product.)
\end{exercise}

Now let's think about how $\mu$ behaves under translation.  A first
guess, by analogy with the case of product Gaussian measure on $\R^n$,
is that it is quasi-invariant under all translations.  But let's look
closer at the finite-dimensional case.  If $\nu$ is standard Gaussian
measure on $\R^n$, i.e.
\begin{equation*}
  d\nu = \frac{1}{(2 \pi)^{n/2}} e^{-|x|^2/2} dx
\end{equation*}
then a simple calculation shows
\begin{equation}\label{radon-finite}
  \frac{d \nu_y}{d \nu}(x) = e^{-\frac{1}{2} |y|^2 + x \cdot y}.
\end{equation}
Note that the Euclidean norm of $y$ appears.  Sending $n \to \infty$,
the Euclidean norm becomes the $\ell^2$ norm.  This suggests that
$\ell^2$ should play a special role.  In particular, translation by
$y$ is not going to produce a reasonable positive Radon-Nikodym
derivative if $\norm{y}_{\ell^2} = \infty$.

Let's denote $\ell^2$, considered as a subset of $W$, by $H$.  $H$
has a Hilbert space structure coming from the $\ell^2$ norm, which
we'll denote by $\nm_H$.  We note that, as shown in Exercise
\ref{Kexercise}, that for fixed $h \in H$, $(h,x)_H$ makes sense not
only for $x \in H$ but for $\mu$-a.e. $x \in W$, and $(h,\cdot)_H$ is
a Gaussian random variable on $(W,\mu)$ with variance $\norm{h}_H^2$.

\begin{theorem}
  [Special case of the Cameron-Martin theorem]  If $h \in H$, then
  $\mu$ is quasi-invariant under translation by $h$, and
  \begin{equation}\label{radon-infinite}
    \frac{d \mu_h}{d \mu}(x) = e^{-\frac{1}{2} \norm{h}_H^2 + (h,x)_H}.
  \end{equation}
  Conversely, if $y \notin H$, then $\mu, \mu_y$ are mutually singular.
\end{theorem}

\begin{proof}
  We are trying to show that
  \begin{equation}\label{radon2}
    \mu_h(B) = \int_B e^{-\frac{1}{2} \norm{h}_H^2 + (h,x)_H} \mu(dx)
  \end{equation}
  for all Borel sets $B \subset W$.  It is sufficient to consider the
  case where $B$ is a ``cylinder set'' of the form $B = B_1 \times
  \dots \times B_n \times \R \times \dots$, since the collection of
  all cylinder sets is a $\pi$-system which generates the Borel
  $\sigma$-algebra.  But this effectively takes us back to the
  $n$-dimensional setting, and unwinding notation will show that in
  this case (\ref{radon-infinite}) is the same as
  (\ref{radon-finite}).

  This is a bit messy to write out; here is an attempt.  If you don't
  like it I encourage you to try to just work it out yourself.

  Since $W$ is a product space, let us
  decompose it as $W = \R^n \times R^\infty$, writing $x \in W$ as
  $(x_n, x_\infty)$ with $x_n = (x(1), \dots, x(n))$ and $x_\infty =
  (x(n+1), \dots)$.  Then $\mu$ factors as $\mu^n \times \mu^\infty$,
  where $\mu^n$ is standard Gaussian measure on $\R^n$ and
  $\mu^\infty$ is again product Gaussian measure on $\R^\infty$.
  $\mu_h$ factors as $\mu^n_{h_n} \times \mu^\infty_{h_\infty}$.
  Also, the integrand in (\ref{radon2}) factors as
  \begin{equation*}
    e^{-\frac{1}{2} |h_n|^2 + h_n \cdot x_n} e^{-\frac{1}{2}
    ||h_\infty||_{\ell_\infty}^2 + (h_\infty, x_\infty)_{\ell^2}}.
  \end{equation*}
  So by Tonelli's theorem the right side of (\ref{radon2}) equals
  \begin{equation*}
    \int_{B_1 \times \dots \times B_n} e^{-\frac{1}{2} |h_n|^2 + h_n
    \cdot x_n} \mu^n(dx_n) \int_{\R^\infty} e^{-\frac{1}{2}
    ||h_\infty||_{\ell_\infty}^2 + (h_\infty, x_\infty)_{\ell^2}} \mu^\infty(dx_\infty).
  \end{equation*}
  The first factor is equal to $\mu^n_{h_n}(B_1 \times \dots \times
  B_n)$ as shown in (\ref{radon-finite}).  Since $(h_\infty,
  \cdot)_{\ell^2}$ is a Gaussian random variable on $(\R^\infty,
  \mu^\infty)$ with variance $||h_\infty||^2_{\ell^2}$, the second
  factor is of the form $E[e^{X - \sigma^2/2}]$ for $X \sim N(0,
  \sigma^2)$, which is easily computed to be $1$.  Since
  $\mu^\infty_{h_\infty}(\R^\infty) = 1$ also (it is a probability
  measure), we are done with the forward direction.

  For the converse direction, suppose $h \notin H$.  Then, by the
  contrapositive of Lemma \ref{ell2}, there exists $g \in \ell^2$ such
  that $\sum h(i) g(i)$ diverges.  Let $A = \{ x \in W : \sum x(i)
  g(i) \text{ converges}\}$; this set is clearly Borel, and we know
  $\mu(A) = 1$ by Exercise \ref{Kexercise}.  But if $\sum x(i) g(i)$
  converges, then $\sum (x-h)(i)g(i)$ diverges, so $A-h$ is disjoint
  from $A$ and $\mu_h(A) = \mu(A-h) = 0$.
\end{proof}

We call $H$ the \textbf{Cameron--Martin space} associated to $(W, \mu)$.

\begin{exercise}
  $H$ is dense in $W$, and the inclusion map $H \hookrightarrow W$ is
  continuous (with respect to the $\ell^2$ topology on $H$ and the
  product topology on $W$).
\end{exercise}

Although $H$ is dense, there are several senses in which it is small.

\begin{proposition}
  $\mu(H) = 0$.
\end{proposition}

\begin{proof}
  For $x \in W$, $x \in H$ iff $\sum_i |\pi_i(x)|^2 < \infty$.  Note
  that the $\pi_i$ are iid $N(0,1)$ random variables on $(W, \mu)$.
  So by the strong law of large numbers, for $\mu$-a.e. $x \in W$ we
  have $\frac{1}{n} \sum_{i=1}^n
  |\pi_i(x)|^2 \to 1$; in particular $\sum_i |\pi_i(x)|^2 = \infty$.
\end{proof}

\begin{exercise}
  Any bounded subset of $H$ is precompact and nowhere dense in $W$.
  In particular, $H$ is meager in $W$.
\end{exercise}

So $\mu$ is quasi-invariant only under translation by elements of the
small subset $H$.

\section{Abstract Wiener space}

Much of this section comes from Bruce Driver's notes \cite{driver-probability} and from Kuo's
book \cite{kuo-gaussian-book}.

\begin{definition}
  An \textbf{abstract Wiener space} is a pair $(W,\mu)$ consisting of
  a separable Banach space $W$ and a Gaussian measure $\mu$ on $W$.
\end{definition}

Later we will write an abstract Wiener space as $(W,H,\mu)$ where $H$
is the Cameron--Martin space.  Technically this is redundant because
$H$ will be completely determined by $W$ and $\mu$.  Len Gross's
original development
\cite{gross-measurable-hilbert-1962,gross-abstract-wiener} went the other way, starting with $H$ and
choosing a $(W,\mu)$ to match it, and this choice is not unique.
We'll discuss this more later.

\begin{definition}
  $(W,\mu)$ is \textbf{non-degenerate} if the covariance form $q$ on
  $W^*$ is positive definite.
\end{definition}

\begin{exercise}\label{full-support-implies-nondegenerate}
  If $\mu$ has full support (i.e. $\mu(U) > 0$ for every nonempty open
  $U$) then $(W,\mu)$ is non-degenerate.  (For the converse, see
  Exercise \ref{nondegenerate-implies-full-support} below.)
\end{exercise}

From now on, we will assume $(W,\mu)$ is non-degenerate unless
otherwise specified.  (This assumption is really harmless, as will be
justified in Remark
\ref{rk-non-degenerate} below.)  So $W^*$ is honestly (injectively) embedded into
$L^2(\mu)$, and $q$ is the restriction to $W^*$ of the $L^2(\mu)$
inner product.  As before, we let $K$ denote the closure of $W^*$ in
$L^2(\mu)$.

Note that we now have two different topologies on $W^*$: the operator norm
topology (under which it is complete), and the topology induced by the
$q$ or $L^2$ inner product (under which, as we shall see, it is not
complete).  The interplay between them will be a big part of what we
do here.

\subsection{Measure-theoretic technicalities}

The main point of this subsection is that the continuous linear
functionals $f \in W^*$, and other functions you can easily build from
them, are the only functions on $W$ that you really have to care
about.

Let $\mathcal{B}$ denote the Borel $\sigma$-algebra on $W$.

\begin{lemma}\label{weak-sigma-field}
  Let $\sigma(W^*)$ be the $\sigma$-field on $W$ generated by $W^*$,
  i.e. the smallest $\sigma$-field that makes every $f \in W^*$
  measurable.  Then $\sigma(W^*) = \mathcal{B}$.
\end{lemma}

Note that the \emph{topology} generated by $W^*$ is not the same as
the original topology on $W$; instead it's the weak topology.

\begin{proof}
  Since each $f \in B^*$ is Borel measurable, $\sigma(W^*) \subset
  \mathcal{B}$ is automatic.

  Let $B$ be the closed unit ball of $W$; we will show $B \in
  \sigma(W^*)$.  Let $\{x_n\}$ be a countable dense subset of $W$.  By
  the Hahn-Banach theorem, for each $x_n$ there exists $f_n \in W^*$
  with $||f_n||_{W^*} =1$ and $f_n(x_n) = ||x_n||$.  I claim that
  \begin{equation}\label{B-cap}
    B = \bigcap_{n = 1}^\infty \{x : |f_n(x)| \le 1\}.
  \end{equation}
  The $\subset$ direction is clear because for $x \in B$, $|f_n(x)|
  \le ||f_n||\cdot ||x|| = ||x|| \le 1$.  For the reverse direction, suppose
  $|f_n(x)| \le 1$ for all $n$.  Choose a sequence $x_{n_k} \to x$; in
  particular $f_{n_k}(x_{n_k}) = ||x_{n_k}|| \to ||x||$.  But since
  $||f_{n_k}|| = 1$, we have 
  $||f_{n_k}(x_{n_k}) - f_{n_k}(x)|| \le ||x_{n_k} - x|| \to 0$, so
  $||x|| = \lim f_{n_k}(x_{n_k}) = \lim f_{n_k}(x) \le 1$.  We have thus
  shown $B \in \sigma(W^*)$, since the right side of (\ref{B-cap}) is
  a countable intersection of sets from $W^*$.

  If you want to show $B(y,r) \in \sigma(W^*)$, we have
  \begin{equation*}
    B(y,r) = \bigcap_{n = 1}^\infty \{x : |f_n(x) - f_n(y))| < r\}.
  \end{equation*}

  Now note that any open subset $U$ of $W$ is a \emph{countable} union of
  closed balls (by separability) so $U \in \sigma(W^*)$ also.  Thus
  $\mathcal{B} \subset \sigma(W^*)$ and we are done.
\end{proof}

Note that we used the separability of $W$, but we did not assume that
$W^*$ is separable.  

\begin{exercise}
If $W$ is not separable, Lemma \ref{weak-sigma-field} may be false.  For a
counterexample, consider $W = \ell^2(E)$ for some uncountable set $E$.
One can show that $\sigma(W^*)$ consists only of sets that depend on
countably many coordinates.  More precisely, for $A \subset E$ let
$\pi_A : \ell^2(E) \to \ell^2(A)$ be the restriction map.  Show that
$\sigma(W^*)$ is exactly the set of all $\pi_A^{-1}(B)$, where $A$ is
countable and $B \subset \ell^2(A)$ is Borel.  In particular,
$\sigma(W^*)$ doesn't contain any singletons (in fact, every nonempty
subset of $\sigma(W^*)$ is non-separable).
\end{exercise}

\begin{question}
  Is Lemma \ref{weak-sigma-field} \emph{always} false for
  non-separable $W$?
\end{question}

Functionals $f \in W^*$ are good for approximation in several senses.
We are just going to quote the following results.  The proofs can be
found in \cite{driver-probability}, and are mostly along the same lines that you prove
approximation theorems in basic measure theory.  For this subsection,
assume $\mu$ is a Borel probability measure on $W$ (not necessarily
Gaussian).

\begin{notation}
  Let $\mathcal{F} C_c^\infty(W)$ denote the ``smooth cylinder
  functions'' on $W$: those functions $F : W \to \R$ of the form $F(x)
  = \varphi(f_1(x), \dots, f_n(x))$ for some $f_1, \dots, f_n \in W^*$
  and $\varphi \in C_c^\infty(\R)$.  (Note despite the notation that
  $F$ is not compactly supported; in fact there are no nontrivial
  continuous functions on $W$ with compact support.)
\end{notation}

\begin{notation}
  Let $\mathcal{T}$ be the ``trigonometric polynomials'' on $W$: those
  functions $F : W \to \R$ of the form $F(x) = a_1 e^{i f_1(x)} +
  \dots + a_n e^{i f_n(x)}$ for $a_1, \dots, a_n \in \R$, $f_1, \dots,
  f_n \in W^*$.
\end{notation}

\begin{theorem}
  $\mathcal{F} C^\infty_c$ and $\mathcal{T}$ are each dense in $L^p(W,
  \mu)$ for any $1 \le p < \infty$.
\end{theorem}

A nice way to prove this is via Dynkin's multiplicative system theorem
(a functional version of the $\pi$-$\lambda$ theorem).

\begin{theorem}[Uniqueness of the Fourier transform]\label{fourier-unique}
  Let $\mu$, $\nu$ be two Borel probability measures on $W$.  If $\int
  e^{i f(x)} \mu(dx) = \int e^{i f(x)} \nu(dx)$ for all $f \in W^*$,
  then $\mu = \nu$.
\end{theorem}

We could think of the Fourier transform of $\mu$ as the map $\hat{\mu}
: W^* \to \R$ defined by $\hat{\mu}(f) = \int e^{i f(x)} \mu(dx)$.
The previous theorem says that $\hat{\mu}$ completely determines
$\mu$.

\subsection{Fernique's theorem}

The first result we want to prove is Fernique's theorem
\cite{fernique70}, which in some sense says that a Gaussian measure
has Gaussian tails: the probability of a randomly sampled point being
at least a distance $t$ from the origin decays like $e^{-t^2}$.  In
one dimension this is easy to prove: if $\mu$ is a Gaussian measure on
$\R$ with, say, variance 1, we have
\begin{equation}\label{gaussian-tail-1d}
\begin{split}
  \mu(\{x : |x| > t\}) &= 2 \int_t^\infty \frac{1}{\sqrt{2 \pi}}
  e^{-x^2/2}\,dx \\ 
  &\le \frac{2}{\sqrt{2\pi}} \int_t^\infty \frac{x}{t} e^{-x^2/2}\,dx
  \\
  &= \frac{2}{t \sqrt{2 \pi}} e^{-t^2/2}
\end{split}
\end{equation}
where the second line uses the fact that $\frac{x}{t} \ge 1$ for $x
\ge t$, and the third line computes the integral directly.

This is sort of like a heat kernel estimate.

\begin{theorem}[Fernique \cite{fernique70}]\label{fernique-theorem}
  Let $(W,\mu)$ be an abstract Wiener space.  There exist $\epsilon >
  0$, $C > 0$ such that for all $t$,
  \begin{equation*}
    \mu(\{x : ||x||_W \ge t\}) \le C e^{-\epsilon t^2}.
  \end{equation*}
\end{theorem}

The proof is surprisingly elementary and quite ingenious.

Let's prove Fernique's theorem.  We follow Driver's proof
\cite[Section 43.1]{driver-probability}.  Some of
the details will be sketched; refer to Driver to see them filled in.

The key idea is that products of Gaussian measures are
``rotation-invariant.''

\begin{lemma}
  Let $(W,\mu)$ be an abstract Wiener space.  Then the product
  measure $\mu^2 = \mu \times \mu$ is a Gaussian measure on $W^2$.  
\end{lemma}

If you're worried about technicalities, you can check the following: $W^2$ is a Banach space
under the norm $||(x,y)||_{W^2} := ||x||_W + ||y||_W$; the norm
topology is the same as the product topology; the Borel $\sigma$-field
on $W^2$ is the same as the product of the Borel $\sigma$-fields on
$W$.

\begin{proof}
  Let $F \in (W^2)^*$.  If we set $f(x) = F(x,0)$, $g(y) = F(0,y)$, we
  see that $f,g \in W^*$ and $F(x,y) =f(x) + g(y)$.  Now when we
  compute the Fourier transform of $F$, we find
  \begin{align*}
    \int_{W^2} e^{i \lambda F(x,y)} \mu^2(dx,dy) &= \int_W e^{i \lambda
    f(x)} \mu(dx)
    \int_W e^{i \lambda g(y)} \mu(dy) \\
    &= e^{-\frac{1}{2}\lambda^2 (q(f,f) + q(g,g))}.
  \end{align*}
\end{proof}

\begin{proposition}
  For $\theta \in \R$, define the ``rotation'' $R_\theta$ on $W^2$ by
  \begin{equation*}
    R_\theta(x,y) = (x \cos\theta - y \sin\theta, x \sin \theta + y
    \cos \theta).
  \end{equation*}
  (We are actually only going to use $R_{\pi/4}(x,y) =
  \frac{1}{\sqrt{2}} (x-y, x+y)$.)  If $\mu$ is Gaussian, $\mu^2$ is
  invariant under $R_\theta$.
\end{proposition}

Invariance of $\mu^2$ under $R_{\pi/4}$ is the only hypothesis we need in order
to prove Fernique's theorem.  You might think this is a very weak
hypothesis, and hence Fernique's theorem should apply to many other
classes of measures.  However, it can actually be shown that any
measure $\mu$ satisfying this condition must in fact be Gaussian, so
no generality is really gained.

\begin{proof}
  Let $\nu = \mu^2 \circ R_\theta^{-1}$; we must show that
  $\nu = \mu^2$.  It is enough to compare their Fourier transforms.
  Let $F \in (W^2)^*$, so $W(x,y) = f(x) + g(y)$, and then
  \begin{align*}
    \int_{W^2} e^{i (f(x) + g(y))} \nu(dx,dy) &= \int_{W^2} e^{i (f(Tx) +
    g(Ty))} \mu^2(dx,dy) \\
    &= \int_{W^2} e^{i (\cos \theta f(x) - \sin \theta f(y) +
    \sin\theta g(x) + \cos\theta g(y))}\, \mu^2(dx,dy) \\
    &= \int_{W} e^{i (\cos \theta f(x) +\sin\theta g(x))}\, \mu(dx)
    \int_{W} e^{i (- \sin \theta f(y) + \cos\theta g(y))}\, \mu(dy)
    \\
    &= e^{-\frac{1}{2}(\cancel{(\sin^2 \theta + \cos^2 \theta)}q(f,f)
    + \cancel{(\sin^2 \theta + \cos^2 \theta)}q(g,g))} \\
    &= \int_{W^2} e^{i (f(x) + g(y))} \mu^2(dx,dy).
\end{align*}
\end{proof}

We can now really prove Fernique's theorem.

\begin{proof}
  In this proof we shall write $\mu(\norm{x} \le t)$ as shorthand for
  $\mu(\{x : \norm{x} \le t\})$, etc.  

  Let $0 \le s \le t$, and consider
  \begin{align*}
    \mu(\norm{x} \le s) \mu(\norm{x} \ge t) &= \mu^2(\{(x,y) : \norm{x} \le s,
    \norm{y} \ge t\}) \\
    &=  \mu^2\left(\norm{\frac{1}{\sqrt{2}}(x-y)} \le s,
    \norm{\frac{1}{\sqrt{2}}(x+y)} \ge t\right)
  \end{align*}
  by $R_{\pi/4}$ invariance of $\mu^2$.  Now some gymnastics with the
  triangle inequality shows that if we have $\norm{\frac{1}{\sqrt{2}}(x-y)}
  \le s$ and $\norm{\frac{1}{\sqrt{2}}(x+y)} \ge t$, then $\norm{x},
  \norm{y} \ge \frac{t-s}{\sqrt{2}}$.  So we have
  \begin{align*}
    \mu(\norm{x} \le s) \mu(\norm{x} \ge t) &\le \mu^2\left(\norm{x} \ge
    \frac{t-s}{\sqrt{2}}, \norm{y} \ge \frac{t-s}{\sqrt{2}}\right)\\ 
    &= \left(\mu\left(\norm{x} \ge \frac{t-s}{\sqrt{2}}\right)\right)^2.
  \end{align*}
  If we rearrange and let $a(t) = \frac{\mu(\norm{x} \ge
  t)}{\mu(\norm{x} \le s)}$, this gives
  \begin{equation}\label{a-relation}
    a(t) \le a\left(\frac{t-s}{\sqrt{2}}\right)^2.
  \end{equation}
  Once and for all, fix an $s$ large enough that $\mu(\norm{x} \ge s)
  < \mu(\norm{x} \le s)$ (so that $a(s) < 1$).  Now we'll iterate
  (\ref{a-relation}).  Set $t_0 = s$, $t_{n+1} = \sqrt{2}(t_n + s)$,
  so that (\ref{a-relation}) reads $a(t_{n+1}) \le a(t_n)^2$, which by
  iteration implies $a(t_n) \le a(s)^{2^n}$.

  Since $t_n \uparrow \infty$, for any $r \ge s$ we have $t_n \le r
  \le t_{n+1}$ for some $n$.  Note that
  \begin{equation*}
    t_{n+1} = s \sum_{k=0}^{n+1} 2^{k/2} \le C 2^{n/2}
  \end{equation*}
  since the largest term dominates.  $a$ is decreasing so we have
  \begin{align*}
    a(r) \le a(t_n) \le a(s)^{2^n} \le a(s)^{r^2/C^2}
  \end{align*}
  so that $a(r) \le e^{-\epsilon r^2}$, taking $\epsilon = -\log(a(s))
  /C^2$.  Since $a(r) = \mu(\norm{x} \le s) \mu(\norm{x} \ge r)$ we
  are done. 
\end{proof}

\begin{corollary}
  If $\epsilon$ is as provided by Fernique's theorem, for $\epsilon' <
  \epsilon$ we have  $\int_W
  e^{\epsilon' ||x||^2} \mu(dx) < \infty$.
\end{corollary}
\begin{proof}
  Standard trick: for a nonnegative random variable $X$, $EX =
  \int_0^\infty P(X > t)\,dt$.  So
  \begin{align*}
    \int_W e^{\epsilon' ||x||^2}\mu(dx) &= \int_0^\infty \mu(\{x :
    e^{\epsilon' ||x||^2} > t\})dt \\
    &= \int_0^\infty \mu\left(\left\{x :
    ||x|| > \sqrt{\frac{\log t}{\epsilon'}}\right\}\right)dt  \\
    &\le \int_0^\infty t^{-\epsilon/\epsilon'}\,dt < \infty.
  \end{align*}
\end{proof}

The following corollary is very convenient for dominated convergence
arguments.

\begin{corollary}
  For any $p > 0$, $\int_W ||x||_W^p \mu(dx) <\infty$.
\end{corollary}

\begin{proof}
  $t^p$ grows more slowly than $e^{\epsilon t^2}$.
\end{proof}

\begin{corollary}
  The inclusion $W^* \hookrightarrow L^2(\mu)$ is bounded.  In
  particular, the $L^2$ norm on $W^*$ is weaker than the operator norm.
\end{corollary}

\begin{proof}
  For $f \in W^*$, $||f||_{L^2}^2 = \int_W |f(x)|^2 \mu(dx) \le
  ||f||_{W^*}^2 \int_W ||x||^2 \mu(dx) \le C ||f||_{W^*}^2$ by the
  previous corollary.
\end{proof}

(This would be a good time to look at Exercises
\ref{topologies-first}--\ref{topologies-last} to get some practice
working with different topologies on a set.)

Actually we can say more than the previous corollary.  Recall that an
operator $T : X \to Y$ on normed spaces is said to be \textbf{compact}
if it maps bounded sets to precompact sets, or equivalently if for
every bounded sequence $\{x_n\} \subset X$, the sequence $\{T x_n\}
\subset Y$ has a convergent subsequence.

\begin{proposition}
  The inclusion $W^* \hookrightarrow L^2(\mu)$ is compact.
\end{proposition}

\begin{proof}
  Suppose $\{f_n\}$ is a bounded sequence in $W^*$; say
  $\norm{f_n}_{W^*} \le 1$ for all $n$.  By Alaoglu's theorem there is
  a weak-* convergent subsequence $f_{n_k}$, which is to say that
  $f_{n_k}$ converges pointwise to some $f \in W^*$.  Note also that
  $|f_{n_k}(x)| \le \norm{x}_W$ for all $k$, and $\int_W \norm{x}_W^2
  \mu(dx) < \infty$ as we showed.  So by dominated convergence,
  $f_{n_k} \to f$ in $L^2(W,\mu)$, and we found an $L^2$-convergent
  subsequence.
\end{proof}

This fact is rather significant: since compact maps on
infinite-dimensional spaces can't have continuous inverses, this shows
that the $W^*$ and $L^2$ topologies on $W^*$ must be quite different.
In particular:

\begin{corollary}
  $W^*$ is not complete in the $q$ inner product (i.e. in the
  $L^2(\mu)$ inner product), except in the trivial case that $W$ is
  finite dimensional.
\end{corollary}

\begin{proof}
  We've shown the
  identity map $(W^*, \norm{\cdot}_{W^*}) \to (W^*, q)$ is continuous
  and bijective.  If $(W^*, q)$ is complete, then by the open mapping
  theorem, this identity map is a homeomorphism, i.e. the $W^*$ and
  $q$ norms are equivalent.  But the identity map is also compact,
  which means that bounded sets, such as the unit ball, are precompact
  (under either topology).  This means that $W^*$ is locally compact
  and hence finite dimensional.
\end{proof}

So the closure $K$ of $W^*$ in $L^2(W,\mu)$ is a \emph{proper}
superset of $W^*$.

\subsection{The Cameron--Martin space}\label{sec-cameron-martin}

Our goal is to find a Hilbert space $H \subset W$ which will play the
same role that $\ell^2$ played for $\R^\infty$.  The key is that, for
$h \in H$, the map $W^* \ni f \mapsto f(h)$ should be continuous with
respect to the $q$ inner product on $W^*$.

As before, let $K$ be the closure of $W^*$ in $L^2(W,\mu)$.  We'll
continue to denote the covariance form on $K$ (and on $W^*$) by $q$.
We'll also use $m$ to denote the inclusion map $m : W^* \to K$.
Recall that we previously argued that $m$ is compact.

\begin{lemma}
  Every $k \in K$ is a Gaussian random variable on $(W,\mu)$.
\end{lemma}

\begin{proof}
  Since $W^*$ is dense in $K$, there is a sequence $f_n \in W^*$
  converging to $k$ in $L^2(W,\mu)$.  In particular, they converge in
  distribution.  By Lemma \ref{limit-of-gaussian}, $k$ is Gaussian.
\end{proof}

\begin{definition}
  The \textbf{Cameron--Martin space} $H$ of $(W,\mu)$ consists of
  those $h \in W$ such that the evaluation functional $f \mapsto f(h)$
  on $W^*$ is continuous with respect to the $q$ inner product.
\end{definition}

$H$ is obviously a vector space.

For $h \in H$, the map $W^* \ni f \mapsto f(h)$ extends by continuity
to a continuous linear functional on $K$.  Since $K$ is a Hilbert
space this may be identified with an element of $K$ itself.  Thus we
have a mapping $T : H \to K$ such that for $f \in W^*$,
\begin{equation*}
  q(Th, f) = f(h).
\end{equation*}
A natural norm on $H$ is defined by
\begin{equation*}
  \norm{h}_H = \sup\left\{\frac{|f(h)|}{\sqrt{q(f,f)}} : f \in W^*, f \ne 0\right\}.
\end{equation*}
This makes $T$ into an isometry, so $\norm{\cdot}_H$ is in fact
induced by an inner product $\inner{\cdot}{\cdot}_H$ on $H$.

Next, we note that $H$ is continuously embedded into $W$.  We have
previously shown (using Fernique) that the embedding of $W^*$ into $K$
is continuous, i.e. $q(f,f) \le C^2 \norm{f}_{W^*}^2$.  So for $h \in
H$ and $f \in W^*$, we have
\begin{equation*}
  \frac{|f(h)|}{\norm{f}_{W^*}} \le C \frac{|f(h)|}{\sqrt{q(f,f)}}.
\end{equation*}
When we take the supremum over all nonzero $f \in W^*$, the left side
becomes $\norm{h}_W$ (by Hahn--Banach) and the right side becomes $C
\norm{h}_H$.  So we have $\norm{h}_W \le C \norm{h}_H$ and the
inclusion $i : H \hookrightarrow W$ is continuous.

(Redundant given the next paragraph.)  Next, we check that $(H, \norm{\cdot}_H)$ is complete.  Suppose $h_n$
is Cauchy in $H$-norm.  In particular, it is bounded in $H$ norm, so
say $\norm{h_n}_H \le M$ for all $n$.  Since the inclusion of $H$ into
$W$ is bounded, $h_n$ is also Cauchy in $W$-norm, hence converges in
$W$-norm to some $x \in W$.  Now fix $\epsilon > 0$, and choose $n$ so
large that $||h_n - h_m||_H < \epsilon$ for all $m \ge n$.  Given a
nonzero $f \in W^*$, we can choose $m \ge n$ so large that $|f(h_m -
x)| \le \epsilon \sqrt{q(f,f)}$.  Then
\begin{equation*}
  \frac{f(h_n - x)}{\sqrt{q(f,f)}} \le \frac{|f(h_n -
  h_m)}{\sqrt{q(f,f)}} + \frac{|f(h_m - x)|}{\sqrt{q(f,f)}} <
  \norm{h_n - h_m}_H + \epsilon < 2\epsilon.
\end{equation*}
We can then take the supremum over $f$ to find that $\norm{h_n - x}_H
< 2 \epsilon$, so $h_n \to x$ in $H$-norm.

Next, we claim the inverse of $T$ is given by
\begin{equation*}
  J k = \int_W x k(x) \mu(dx)
\end{equation*}
where the integral is in the sense of Bochner.  (To see that the
integral exists, note that by Fernique $\norm{\cdot} \in L^2(W,
\mu)$.)  For $f \in W^*, k \in K$, we have
\begin{equation*}
  \abs{f(Jk)} = \abs{\int_W f(x)
  k(x)\mu(dx)} = \abs{q(f,k)} \le \sqrt{q(f,f) q(k,k)}
\end{equation*}
whence $\norm{\int_W x k(x) \mu(dx)}_H \le \sqrt{q(k,k)}$.  So $J$ is a
continuous operator from $K$ to $H$.  Next, for $f \in W^*$ we have
\begin{equation*}
  q(TJk, f) = f(Jk) = q(k,f)
\end{equation*}
as we just argued.  Since $W^*$ is dense in $K$, we have $TJk = k$.
In particular, $T$ is surjective, and hence unitary.

\begin{question}
  Could we have done this without the Bochner integral?
\end{question}

We previously showed that the inclusion map $i : H \to W$ is
continuous, and it's clearly 1-1.  It has an adjoint operator $i^* :
W^* \to H$.  We note that for $f \in W^*$ and $h \in H$, we have
\begin{align*}
  q(f, Th) = f(h) = \inner{i^* f}{h}_H = q(Ti^*f, Th).
\end{align*}
Since $T$ is surjective we have $q(f,k) = q(T i^* f,k)$ for all $k \in
K$; thus $T i^*$ is precisely the inclusion map $m : W^* \to K$.
Since $m$ is compact and 1-1 and $T$ is unitary, it follows that $i^*$ is
compact and 1-1.  

Since $i^*$ is 1-1, it follows that $H$ is dense in $W$: if $f \in
W^*$ vanishes on $H$, it means that for all $h \in H$, $0 = f(h) =
\inner{i^* f}{h}_H$, so $i^* f = 0$ and $f = 0$.  The Hahn--Banach
theorem then implies $H$ is dense in $W$.  Moreover, Schauder's
theorem from functional analysis (see for example \cite[Theorem
VI.3.4]{conway}) states that an operator between Banach spaces is
compact iff its adjoint is compact, so $i$ is compact as well.  In
particular, $H$ is not equal to $W$, and is not complete in the $W$
norm.

We can sum up all these results with a diagram.

\begin{theorem}
The following diagram commutes.
\begin{equation}
\xymatrix{
 & & W^* \ar@{.>}[ddll]_{i^*} \ar@{.>}[rr]^m & & K \ar@{->}@/^/[ddllll]^J  \\ \\
  H  \ar@{.>}[rr]^i \ar@{->}@/^/[uurrrr]^T & & W
}
\end{equation}
All spaces are complete in their own norms.  All dotted arrows are
compact, 1-1, and have dense image.  All solid arrows are unitary.
\end{theorem}

Sometimes it's convenient to work things out with a basis.

\begin{proposition}\label{ek-basis}
  There exists a sequence $\{e_k\}_{k=1}^\infty \subset W^*$ which is
  an orthonormal basis for $K$.  $e_k$ are iid $N(0,1)$ random
  variables under $\mu$.  For $h \in H$, we have $\norm{h}_H^2
  = \sum_{k=1}^\infty |e_k(h)|^2$, and the sum is infinite for $h \in
  W \backslash H$.  
\end{proposition}

\begin{proof}
  The existence of $\{e_k\}$ is proved in Lemma
  \ref{dense-subspace-basis}.  They are jointly Gaussian random
  variables since $\mu$ is a Gaussian measure.  Orthonormality means
  they each have variance 1 and are uncorrelated, so are iid.

  If $h \in H$, then $\sum_k |e_k(h)|^2 = \sum_k |q(e_k, Th)|^2 =
  \norm{Th}_K^2 = \norm{h}_H^2$ since $T$ is an isometry.  Conversely,
  suppose $x \in W$ and $M := \sum_k |e_k(x)|^2 < \infty$.  Let $E
  \subset X^*$ be the linear span of $\{e_k\}$, i.e. the set of all $f \in
  W^*$ of the form $f = \sum_{k=1}^n a_k e_k$.  For such $f$ we have
  \begin{align*}
    |f(x)|^2 &= \abs{\sum_{k=1}^n a_k e_k(x)}^2 \\ 
    &\le \left(\sum_{k=1}^n |a_k|^2\right) \left(\sum_{k=1}^n
    |e_k(x)|^2\right) && \text{(Cauchy--Schwarz)} \\
    \le M q(f,f)
  \end{align*}
  Thus $x \mapsto f(x)$ is a bounded linear functional on $(E,
  q)$. $(E,q)$ is dense in $(W^*, q)$ so the same bound holds for all
  $f \in X^*$.  Thus by definition we have $x \in H$.
\end{proof}

\begin{proposition}
  $\mu(H) = 0$.
\end{proposition}

\begin{proof}
  For $h \in H$ we have $\sum |e_k(h)|^2 < \infty$.  But since $e_k$
  are iid, by the strong law of large numbers we have that $\sum
  |e_k(x)|^2 = +\infty$ for $\mu$-a.e. $x$. 
\end{proof}

\begin{notation}
  Fix $h \in H$.  Then $\inner{h}{x}_H$ is unambiguous for all $x \in
  H$.  If we interpret $\inner{h}{x}_H$ as $(Th)(x)$, it is also
  well-defined for almost every $x$, and so $\inner{h}{\cdot}_H$ is a
  Gaussian random variable on $(W,\mu)$ with variance $\norm{h}_H^2$.  
\end{notation}

\begin{theorem}[Cameron--Martin]
  For $h \in H$, $\mu_h$ is absolutely continuous with respect to
  $\mu$, and 
  \begin{equation*}
    \frac{d \mu_h}{d \mu}(x) = e^{- \frac{1}{2} \norm{h}_H^2 +
    \inner{h}{x}_H}.
  \end{equation*}
  For $x \in W \backslash H$, $\mu_x$ and $\mu$ are singular.
\end{theorem}

\begin{proof}
  Suppose $h \in H$.  We have to show $\mu_h(dx) = e^{- \frac{1}{2}
    \norm{h}_H^2 + \inner{h}{x}_H} \mu(dx)$.  It is enough to show
    their Fourier transforms are the same (Theorem
    \ref{fourier-unique}).  For $f \in W^*$ we have
    \begin{equation*}
      \int_W e^{i f(x)} \mu_h(dx) = \int_W e^{i f(x+h)} \mu(dx) = e^{i
      f(h) - \frac{1}{2} q(f,f)}.
    \end{equation*}
    On the other hand,
    \begin{align*}
      \int_W e^{i f(x)} e^{- \frac{1}{2}
    \norm{h}_H^2 + \inner{h}{x}_H} \mu(dx) &= e^{- \frac{1}{2}
    \norm{h}_H^2} \int_W e^{i (f-iTh)(x)} \mu(dx) \\
      &= e^{- \frac{1}{2}
    \norm{h}_H^2} e^{-\frac{1}{2} q(f-iTh, f-iTh)}
    \end{align*}
    since $f - iTh$ is a complex Gaussian random variable (we will let
    the reader check that everything works fine with complex numbers
    here).  But we have
    \begin{align*}
      q(f-iTh, f-iTh) = q(f,f) - 2 i q(f, Th) - q(Th, Th) = q(f,f) -
      2i f(h) - \norm{h}_H^2
    \end{align*}
    by properties of $T$, and so in fact the Fourier transforms are
    equal.

    Conversely, if $x \in W \backslash H$, by Lemma \ref{ek-basis} we
    have $\sum_k |e_k(x)|^2 = \infty$.  By Lemma \ref{ell2} there
    exists $a \in \ell^2$ such that $\sum a_k e_k(x)$ diverges.  Set
    $A = \{y \in W : \sum a_k e_k(y) \text{ converges}\}$.  We know
    that $\sum_k a_k e_k$ converges in $L^2(W,\mu)$, and is a sum of
    independent random variables (under $\mu$), hence it converges
    $\mu$-a.s.  Thus $\mu(A) = 1$.  However, if $y \in A$, then $\sum
    a_k e_k(y-x)$ diverges, so $A-x$ is disjoint from $A$, and thus
    $\mu_x(A) = \mu(A-x) = 0$.      
\end{proof}

\begin{exercise}\label{nondegenerate-implies-full-support}
  $\mu$ has full support, i.e. for any nonempty open set $U$, $\mu(U)
  > 0$.  This is the converse of Exercise
  \ref{full-support-implies-nondegenerate}.  (Hint: First show this
  for $U \ni 0$.  Then note any nonempty open $U$ contains a
  neighborhood of some $h \in H$.  Translate.)  (Question: Can we
  prove this without needing the Cameron--Martin hammer?  I think yes,
  look for references.)
\end{exercise}

\begin{remark}\label{rk-non-degenerate}
  There really isn't any generality lost by assuming that $(W,\mu)$ is
  non-degenerate.  If you want to study the degenerate case, let $F =
  \{f \in W^* : q(f,f) = 0\}$ be the kernel of $q$, and 
  consider the closed subspace
  \begin{equation*}
    W_0 := \bigcap_{f \in F} \ker f \subset W.
  \end{equation*}
  We claim that $\mu(W_0) = 1$.  For each $f \in F$, the condition
  $q(f,f) = \int f^2\,d\mu = 0$ implies that $f = 0$ $\mu$-almost
  everywhere, so $\mu(\ker f) = 1$, but as written, $W_0$ is an
  uncountable intersection of such sets.  To fix that, note that since
  $W$ is separable, the unit ball $B^*$ of $W^*$ is weak-* compact
  metrizable, hence weak-* separable metrizable, hence so is its
  subset $F \cap B^*$.  So we can choose a countable weak-* sequence
  $\{f_n\} \subset F \cap B^*$.  Then I claim
  \begin{equation*}
    W_0 = \bigcap_n \ker f_n.
  \end{equation*}
  The $\subset$ inclusion is obvious.  To see the other direction,
  suppose $x \in \bigcap_n \ker f_n$ and $f \in F$; we will show $f(x)
  = 0$.  By rescaling, we can assume without loss of generality that
  $f \in B^*$.  Now choose a subsequence $f_{n_k}$ converging weak-*
  to $f$; since $f_{n_k}(x) = 0$ by assumption, we have $f(x) = 0$
  also.  Now $W_0$ is written as a \emph{countable} intersection of
  measure-$1$ subsets, so $\mu(W_0) = 1$.

  We can now work on the abstract Wiener space $(W_0, \mu|_{W_0})$.
  Note that the covariance form $q_0$ defined on $W_0^*$ by $q_0(f_0,
  f_0) = \int_{W_0} f_0^2\,d\mu$ agrees with $q$, since given any
  extension $f \in W^*$ of $f_0$ will satisfy
  \begin{equation*} q(f,f) = \int_W
  f^2\,d\mu = \int_{W_0} f^2\,d\mu = \int_{W_0} f_0^2\,d\mu = q_0(f_0,
  f_0).
  \end{equation*}

  This makes it easy to see that that $q_0$ is positive definite on
  $W_0^*$.  Suppose $q_0(f_0, f_0) = 0$ and use Hahn--Banach to choose
  an extension $f \in W^*$ of $f_0$.  Then $q(f,f) = 0$, so by
  definition of $W_0$, we have $W_0 \subset \ker f$; that is, $f$
  vanishes on $W_0$, so the restriction $f_0 = f|_{W_0}$ is the zero
  functional.

  It now follows, from the previous exercise, that the support of
  $\mu$ is precisely $W_0$.  So $(W_0, \mu|_{W_0})$ is a
  non-degenerate abstract Wiener space, and we can do all our work on
  this smaller space.
  
  I'd like to thank Philipp Wacker for suggesting this remark and sorting out
  some of the details.
\end{remark}

\subsection{Example: Gaussian processes}

Recall that a one-dimensional stochastic process $X_t, 0 \le t \le 1$
is said to be \textbf{Gaussian} if, for any $t_1, \dots, t_n \ge 0$,
the random vector $(X_{t_1}, \dots, X_{t_n})$ has a joint Gaussian
distribution.  If the process is continuous, its distribution gives a
probability measure $\mu$ on $W = C([0,1])$.  If there is any good in
the world, this ought to be an example of a Gaussian measure.

By the Riesz representation theorem, we know exactly what $W^*$ is:
it's the set of all finite signed Borel measures $\nu$ on $[0,1]$.  We
don't yet know that all of these measures represent Gaussian random
variables, but we know that some of them do.  Let $\delta_t$ denote
the measure putting unit mass at $t$, so $\delta_t(\omega) = \int_0^1
\omega(t)\,d\delta_t = \omega(t)$.  We know that $\{\delta_t\}_{t \in
  [0,1]}$ are jointly Gaussian.  If we let $E \subset W^*$ be their linear span,
i.e. the set of all finitely supported signed measures, i.e. the set
of measures $\nu = \sum_{i=1}^n a_i \delta_{t_i}$, then all measures
in $E$ are Gaussian random variables.

\begin{lemma}
  $E$ is weak-* dense in $W^*$, and dense in $K$.
\end{lemma}

\begin{proof}
  Suppose $\nu \in W^*$.  Given a partition $\mathcal{P} = \{0 = t_0 < \dots < t_n
  = 1\}$ of $[0,1]$, set $\nu_{\mathcal{P}} = \sum_{j=1}^n \int
  1_{(t_{j-1}, t_j]} d\nu \delta_{t_j}$.  Then for each $\omega \in
  C([0,1])$, $\int \omega \,d\nu_{\mathcal{P}} = \int
  \omega_{\mathcal{P}} \,d\nu$, where
  \begin{equation*}\omega_{\mathcal{P}} =
    \sum_{j=1}^n \omega(t_j) 1_{(t_{j-1}, t_j]}.
  \end{equation*}
  But by uniform
  continuity, as the mesh size of $\mathcal{P}$ goes to 0, we have
  $\omega_{\mathcal{P}} \to \omega$ uniformly, and so $\int
  \omega_{\mathcal{P}} d\nu \to \int
  \omega \, d\nu$.  Thus $\nu_{\mathcal{P}} \to \nu$ weakly-*.
\end{proof}

\begin{corollary}
  $\mu$ is a Gaussian measure.
\end{corollary}

\begin{proof}
  Every $\nu \in W^*$ is a pointwise limit of a sequence of Gaussian
  random variables, hence Gaussian.
\end{proof}

\begin{lemma}
  $E$ is dense in $K$.
\end{lemma}

\begin{proof}
  $\{\nu_{\mathcal{P}}\}$ is bounded in total variation (in fact
  $\norm{\nu_{\mathcal{P}}} \le \norm{\nu}$).  So by Fernique's
  theorem and dominated convergence, $\nu_{\mathcal{P}} \to \nu$ in
  $L^2(X,\mu)$.  Thus $E$ is $L^2$-dense in $W^*$.  Since $W^*$ is
  $L^2$-dense in $K$, $E$ is dense in $K$.
\end{proof}

Note that in order to get $\mu$ to be non-degenerate, it may be
necessary to replace $W$ by a smaller space.  For example, if $X_t$ is
Brownian motion started at 0, the linear functional $\omega \mapsto
\omega(0)$ is a.s. zero.  So we should take $W = \{ \omega \in
C([0,1]) : \omega(0)=0\}$.  One might write this as $C_0((0,1])$.

Recall that a Gaussian process is determined by its covariance
function $a(s,t) = E[X_s X_t] = q(\delta_s, \delta_t)$.  Some examples:

\begin{enumerate}
  \item Standard Brownian motion $X_t = B_t$ started at 0: $a(s,t) =
  s$ for $s < t$.  Markov, martingale, independent increments,
  stationary increments.

  \item Ornstein--Uhlenbeck process defined by $dX_t = -X_t \,dt +
  \sigma \,dB_t$: $a(s,t) = \frac{\sigma^2}{2} (e^{-(t-s)} -
  e^{-(t+s)})$, $s < t$.  Markov, not a martingale.

  \item Fractional Brownian motion with Hurst parameter $H \in (0,1)$:
  $a(s,t) = \frac{1}{2}(t^{2H} + s^{2H} - (t-s)^{2H})$, $s < t$.  Not
  Markov.

  \item Brownian bridge $X_t = B_t - t B_1$: $a(s,t) = s(1-t)$, $s < t$.  (Here
  $W$ should be taken as $\{\omega \in C([0,1]) : \omega(0) =
  \omega(1) = 0\} = C_0((0,1))$, the so-called pinned loop space.)  
\end{enumerate}

\begin{lemma}
  The covariance form $q$ for a Gaussian process is defined by
  \begin{equation*}
    q(\nu_1, \nu_2) = \int_0^1 \int_0^1 a(s,t) \nu_1(ds) \nu_2(dt)
  \end{equation*}
  for $\nu_1, \nu_2 \in W^*$.
\end{lemma}

\begin{proof}
  Fubini's theorem, justified with the help of Fernique.
\end{proof}

\begin{lemma}
  $J : K \to H$ is defined by $Jk(t) = q(k, \delta_t)$.  For $k = \nu
  \in W^*$ this gives $J\nu(t) = i^* \nu(t) = \int_0^1 a(s,t)
  \nu(ds)$.  In particular $J \delta_s(t) = a(s,t)$.
\end{lemma}

\begin{proof}
  $Jk(t) = \delta_t(Jk) = q(k, \delta_t)$.
\end{proof}

Observe that $a$ plays the role of a reproducing kernel in $H$: we
have
\begin{align*}
  \inner{h}{a(s, \cdot)}_H = \inner{h}{J \delta_s} = \delta_s(h) = h(s).
\end{align*}
This is why $H$ is sometimes called the ``reproducing kernel Hilbert
space'' or ``RKHS'' for $W$.

\subsection{Classical Wiener space}

Let $\mu$ be Wiener measure on classical Wiener space $W$, so $a(s,t)
= s \wedge t$.

\begin{theorem}\label{classical-cameron-martin}
  The Cameron--Martin space $H \subset W$ is given by the set of all
  $h \in W$ which are absolutely continuous and have $\dot{h} \in
  L^2([0,1], m)$.  The Cameron-Martin inner product is given by
  $\inner{h_1}{h_2}_H = \int_0^1 \dot{h_1}(t) \dot{h_2}(t)\,dt$.
\end{theorem}

\begin{proof}
  Let $\tilde{H}$ be the candidate space with the candidate norm
  $\nm_{\tilde{H}}$.  It's easy to see that $\tilde{H}$ is a Hilbert
  space.  

  Note that $J \delta_s(t) = s \wedge t \in \tilde{H}$, so by
  linearity $J$ maps $E$ into $\tilde{H}$.  Note $\dot{J \delta_s} =
  1_{[0,s]}$.  Moreover,
  \begin{align*}
    \inner{J \delta_s}{J \delta_r}_{\tilde{H}} = \int_0^1 1_{[0,s]}
    1_{[0,r]} dm = s \wedge r = q(\delta_s, \delta_r)
  \end{align*}
  so $J$ is an isometry from $(E,q)$ to $\tilde{H}$.  Hence it extends
  to an isometry of $K$ to $\tilde{H}$.  Since $J$ is already an
  isometry from $K$ to $H$ we have $H = \tilde{H}$ isometrically.
\end{proof}

Now what can we say about $T$?  It's a map that takes a continuous
function from $H$ and returns a random variable.  Working informally,
we would say that 
\begin{equation}\label{T-informal}
Th(\omega) = \inner{h}{\omega}_H = \int_0^1
\dot{h}(t) \dot{\omega}(t)\,dt.
\end{equation}  This formula is absurd because
$\dot{\omega}$ is nonexistent for $\mu$-a.e. $\omega$ (Brownian motion
sample paths are nowhere differentiable).  However, it is actually the
right answer if interpreted correctly.

Let's suppose that $h$ is piecewise linear: then its derivative is a
step function $\dot{h} = \sum_{i=1}^n b_i 1_{[c_i, d_i]}$.  Note that
the reproducing kernel $a(s, \cdot)$ has as its derivative the step
function $1_{[0,s]}$.  So by integrating, we see that we can write
\begin{equation*}
  h(t) = \sum_{i=1}^n b_i (a(d_i, t) - a(c_i, t)).
\end{equation*}
Now we know that $T[a(s,\cdot)] = \delta_s$, i.e. the random variable
$B_s$.  So we have
\begin{equation*}
  Th = \sum_{i=1}^n b_i (B_{d_i} - B_{c_i}).
\end{equation*}
We can recognize this as the stochastic integral of the step function
$\dot{h} = \sum_{i=1}^n b_i 1_{[c_i, d_i]}$: 
\begin{equation}\label{T-integral}
  Th = \int_0^1
\dot{h}(t)\,dB_t. 
\end{equation}
Moreover, by the It\^o isometry we know that
\begin{equation*}
  \norm{ \int_0^1
\dot{h}(t)\,dB_t}_{L^2(W,\mu)}^2 = \norm{\dot{h}}_{L^2([0,1])}^2 = \norm{h}_H^2.
\end{equation*}
Thus both sides of (\ref{T-integral}) are isometries on $H$, and they
are equal for all piecewise linear $H$.  Since the step functions are
dense in $L^2([0,1])$, the piecewise linear functions are dense in $H$
(take derivatives), so in fact (\ref{T-integral}) holds for all $h \in
H$.  We have rediscovered the stochastic integral, at least for
deterministic integrands.  This is sometimes called the Wiener
integral.  Of course, the It\^o integral also works for stochastic
integrands, as long as they are adapted to the filtration of the
Brownian motion.  Later we shall use our machinery to produce the
Skorohod integral, which will generalize the It\^o integral to
integrands which need not be adapted, giving us an ``anticipating
stochastic calculus.''

\begin{exercise}
  For the Ornstein--Uhlenbeck process, show that $H$ is again the set
  of absolutely continuous functions $h$ with $\dot{h} \in
  L^2([0,1])$, and the Cameron--Martin inner product is given by
  \begin{equation*}
    \inner{h_1}{h_2}_H = \frac{1}{\sigma^2}\int_0^1 \dot{h_1}(t) \dot{h_2}(t) + h_1(t) h_2(t)\,dt.
  \end{equation*}
\end{exercise}

\begin{exercise}
  For the Brownian bridge, show that $H$ is again the set of
  absolutely continuous functions $h$ with $\dot{h} \in L^2([0,1])$,
  and the Cameron--Martin inner product is given by
  $\inner{h_1}{h_2}_H = \int_0^1 \hat{h_1}(t) \hat{h_2}(t)\,dt$, where
  \begin{equation*}
    \hat{h}(t) = \dot{h}(t) + \frac{h(t)}{1-t}.
  \end{equation*}
\end{exercise}

Perhaps later when we look at some stochastic differential equations,
we will see where these formulas come from.

Note that in this case the Cameron--Martin theorem is a special case
of Girsanov's theorem: it says that a Brownian motion with a ``smooth'' drift
becomes a Brownian motion without drift under an equivalent measure.
Indeed, suppose $h \in H$.  If we write $B_t(\omega) = \omega(t)$, so
that $\{B_t\}$ is a Brownian motion on $(W,\mu)$, then $B_t + h(t)$ is
certainly a Brownian motion (without drift!) on $(W, \mu_h)$.  The Cameron-Martin
theorem says that $\mu_h$ is an equivalent measure to $\mu$.  Anything
that $B_t$ can't do, $B_t + h(t)$ can't do either (since the
$\mu$-null and $\mu_h$-null sets are the same).  This fact has many
useful applications.  For example, in mathematical finance, one might
model the price of an asset by a geometric Brownian motion with a drift
indicating its average rate of return
(as in the Black--Scholes model).  The Cameron--Martin/Girsanov
theorem provides an equivalent measure under which this process is a
martingale, which makes it possible to compute the arbitrage-free
price for options involving the asset.  The equivalence of the
measures is important because it guarantees that changing the measure
didn't allow arbitrage opportunities to creep in.

\subsection{Construction of $(W,\mu)$ from $H$}

This section originates in \cite{gross-abstract-wiener} via Bruce
Driver's notes \cite{driver-probability}.

When $W = \R^n$ is finite-dimensional and $\mu$ is non-degenerate, the
Cameron--Martin space $H$ is all of $W$ (since $H$ is known to be
dense in $W$), and one can check that the Cameron--Martin norm is 
\begin{equation}
  \inner{x}{y}_H = x \cdot \Sigma^{-1} y
\end{equation}
where $\Sigma$ is the covariance matrix.  We also know that $\mu$ has
a density with respect to Lebesgue measure $dx$, which we can write as
\begin{equation}
  \mu(dx) = \frac{1}{Z} e^{-\frac{1}{2} \norm{x}_H^2} dx
\end{equation}
where $Z = \int_{\R^n} e^{-\frac{1}{2} \norm{x}_H^2} dx$ is a
normalizing constant chosen to make $\mu$ a probability measure.
Informally, we can think of $\mu$ as being given by a similar formula
in infinite dimensions:
\begin{equation}\label{nonsense-abstract}
  \mu(dx) \mathquote{=} \frac{1}{\mathcal{Z}} e^{-\frac{1}{2} \norm{x}_H^2} \mathcal{D}x
\end{equation}
where $\mathcal{Z}$ is an appropriate normalizing constant, and
$\mathcal{D}x$ is infinite-dimensional Lebesgue measure.  Of course
this is nonsense in at least three different ways, but that doesn't
stop physicists, for instance.

For classical Wiener measure this reads
\begin{equation}\label{nonsense-classical}
  \mu(dx) \mathquote{=} \frac{1}{\mathcal{Z}} e^{-\frac{1}{2} \int_0^1
  |\dot{\omega}(t)|^2 dt} \mathcal{D}\omega.
\end{equation}

Since the only meaningful object appearing on the right side of
(\ref{nonsense-abstract}) is $\norm{\cdot}_H$, it is reasonable to ask
if we can start with a Hilbert space $H$ and produce an abstract
Wiener space $(W,\mu)$ for which $H$ is the Cameron--Martin space.

\subsubsection{Cylinder sets}

Let $(H, \norm{\cdot}_H)$ be a separable Hilbert space.

\begin{definition}
  A \textbf{cylinder set} is a subset $C \subset H$ of the form
  \begin{equation}\label{cylinder-set}
    C = \{h \in H : (\inner{h}{k_1}_H, \dots, \inner{h}{k_n}_H) \in A\}
  \end{equation}
  for some $n \ge 1$, orthonormal $k_1, \dots, k_n$,  and $A \subset
 \R^n$ Borel.
\end{definition}

\begin{exercise}
    Let $\mathcal{R}$ denote the collection of all cylinder sets in
  $H$.  $\mathcal{R}$ is an algebra: we have $\emptyset \in
  \mathcal{R}$ and $\mathcal{R}$ is closed under complements and
  \emph{finite} unions (and intersections).  However, if $H$ is
  infinite dimensional then $\mathcal{R}$ is not a $\sigma$-algebra.
\end{exercise}

Note by Lemma \ref{weak-sigma-field} that $\sigma(\mathcal{R}) =
\mathcal{B}_H$, the Borel $\sigma$-algebra.

We are going to try to construct a Gaussian measure $\tilde{\mu}$ on $H$ with
covariance form given by $\inner{\cdot}{\cdot}_H$.  Obviously we can
only get so far, since we know of several obstructions to completing
the task.  At some point we will have to do something different.  But
by analogy with finite dimensions, we know what value $\tilde{\mu}$ should
give to a cylinder set of the form (\ref{cylinder-set}): since $k_1,
\dots, k_n$ are orthonormal, they should be iid standard normal with
respect to $\tilde{\mu}$, so we should have
\begin{equation}\label{mu-cylinder-def}
  \tilde{\mu}(C) = \mu_n(A)
\end{equation}
where $d \mu_n = \frac{1}{(2 \pi)^{n/2}} e^{-|x|^2/2}dx$ is standard
Gaussian measure on $\R^n$.

\begin{proposition}
  The expression for $\tilde{\mu}(C)$ in (\ref{mu-cylinder-def}) is
  well-defined, and $\tilde{\mu}$ is a finitely additive probability measure on $\mathcal{R}$.
\end{proposition}

\begin{proof}
  To check that $\tilde{\mu}(C)$ is well-defined, suppose that
  \begin{equation}\label{Ceq}
    C = \{h \in H : (\inner{h}{k_1}_H, \dots, \inner{h}{k_n}_H) \in
    A \subset \R^n\} = \{h \in H : (\inner{h}{k'_1}_H, \dots,
    \inner{h}{k'_{n'}}_H) \in A' \subset \R^{n'}\}.
  \end{equation}
  Let $E$ be the span in $H$ of $\{k_1, \dots, k_n, k'_1, \dots,
  k'_{n'}\}$, and let $m = \dim E$.  Since $\{k_1, \dots, k_n\}$ is
  orthonormal in $E$, we can extend it to an orthonormal basis $\{k_1,
  \dots, k_m\}$ for $E$, and then we have
  \begin{equation*}
    C = \{h \in H : (\inner{h}{k_1}_H, \dots, \inner{h}{k_m}_H) \in
    A \times \R^{m-n}\}.
  \end{equation*}
  Since $\mu_m$ is a product measure, we have $\mu_m(A \times
  \R^{m-n}) = \mu_n(A)$.  So by playing the same game for $\{k'_1,
  \dots, k'_{n'}\}$, there is no
  loss of generality in assuming that in (\ref{Ceq}) we have $n = n' = m$, and that $\{k_1,
  \dots, k_m\}$ and $\{k'_1, \dots, k'_m\}$ are two orthonormal bases
  for the same $E \subset H$.  We then have to show that $\mu_m(A) = \mu_m(A')$.

  We have two orthonormal bases for $E$, so there is a unitary $T : E
  \to E$ such that $T k_i = k'_i$.  Let $P : H \to E$ be orthogonal
  projection, and define $S : E \to \R^m$ by $Sx = ((x, k_1), \dots,
  (x, k_m))$.  Then $S$ is unitary.  If we define $S'$ analogously,
  then $S' = ST^* = ST$, and we have
  \begin{equation*}
    C = P^{-1} S^{-1} A = P^{-1} S'^{-1} A' = P^{-1} T^{-1} S^{-1} A'.
  \end{equation*}
  Since $P : H \to E$ is surjective, we must have $S^{-1} A = T^{-1}
  S^{-1} A'$; since $S,T$ are bijective this says $A' = STS^{-1} A$,
  so $A'$ is the image of $A$ under a unitary map.  But standard
  Gaussian measure on $\R^m$ is invariant under unitary
  transformations, so indeed $\mu_m(A) = \mu_m(A')$, and the
  expression (\ref{mu-cylinder-def}) is well defined.

  It is obvious that $\tilde{\mu}(\emptyset) = 0$ and $\tilde{\mu}(H) = 1$.  For
  finite additivity, suppose $C_1, \dots, C_n \in \mathcal{R}$ are
  disjoint.  By playing the same game as above, we can write $C_i =
  P^{-1} (A_i)$ for some common $P : H \to \R^m$, where the $A_i
  \subset \R^m$ are necessarily disjoint, and then $\tilde{\mu}(C_i) =
  \mu_m(A_i)$.  Since $\bigcup_i C_i = P^{-1}\left(\bigcup_i
  A_i\right)$, the additivity of $\mu_m$ gives us that
  $\tilde{\mu}\left(\bigcup_i C_i\right) = \sum_i \tilde{\mu}(C_i)$.
\end{proof}

We will call $\tilde{\mu}$ the canonical Gaussian measure on $H$.  As we see
in the next proposition, we're using the term ``measure'' loosely.

\begin{proposition}
  If $H$ is infinite dimensional, $\tilde{\mu}$ is not countably additive on
  $\mathcal{R}$.  In particular, it does not extend to a countably
  additive measure on $\sigma(\mathcal{R}) = \mathcal{B}_H$.
\end{proposition}

\begin{proof}
  Fix an orthonormal sequence $\{e_i\}$ in $H$.  Let
  \begin{equation*}
    A_{n,k} = \{ x \in H : |\inner{x}{e_i}| \le k, i=1,\dots, n\}.
  \end{equation*}
  $A_{n,k}$ is a cylinder set, and we have $B(0,k) \subset A_{n,k}$
  for any $n$.  Also, we have $\tilde{\mu}(A_{n,k}) =
  \mu_n([-k,k]^n) = \mu_1([-k,k])^n$ since $\mu_n$ is a product
  measure.  Since $\mu_1([-k,k]) < 1$, for each $k$ we can choose an
  $n_k$ so large that $\tilde{\mu}(A_{n_k,k}) = \mu_1([-k,k])^{n_k} <
  2^{-k}$.  Thus $\sum_{k=1}^\infty \tilde{\mu}(A_{n_k, k}) < 1$, but since
  $B(0,k) \subset A_{n_k, k}$ we have $\bigcup_{k=1}^\infty A_{n_k,k}
  = H$ and $\tilde{\mu}(H)=1$.  So countable additivity does not hold.
\end{proof}

Of course we already knew that this construction cannot produce a
genuine Gaussian measure on $H$, since any Gaussian measure has to
assign measure 0 to its Cameron--Martin space.  The genuine measure
has to live on some larger space $W$, so we have to find a way to
produce $W$.  We'll produce it by producing a new norm
$\norm{\cdot}_W$ on $H$ which is not complete, and set $W$ to be the
completion of $H$ under $\norm{\cdot}_W$.  Then we will be able to
extend $\tilde{\mu}$, in a certain sense, to an honest Borel measure
$\mu$ on $W$.

It's common to make an analogy here with Lebesgue measure.  Suppose we
were trying to construct Lebesgue measure $m$ on $\Q$.  We could
define the measure of an interval $(a,b) \subset \Q$ to be $b-a$, and
this would give a finitely additive measure on the algebra of sets
generated by such intervals.  But it could not be countably additive.
If we want a countably additive measure, it has to live on $\R$, which
we can obtain as the completion of $\Q$ under the Euclidean metric.

\subsubsection{Measurable norms}

\begin{definition}
  By a \textbf{finite rank projection} we mean a map $P : H \to H$
  which is orthogonal projection onto its image $PH$ with $PH$ finite
  dimensional.  We will sometimes abuse notation and identify $P$ with
  the finite-dimensional subspace $PH$, since they are in 1-1
  correspondence.  We will write things like $P_1 \perp P_2$, $P_1
  \subset P_2$, etc.
\end{definition}

We are going to obtain $W$ as the completion of $H$ under some norm
$\norm{\cdot}_W$.  Here is the condition that this norm has to
satisfy.

\begin{definition}
  A norm $\nm_{W}$ on $H$ is said to be \textbf{measurable} if for
  every $\epsilon > 0$ there exists a finite rank projection $P_0$
  such that
  \begin{equation}\label{measurable-condition}
    \tilde{\mu}(\{h : \norm{Ph}_W > \epsilon\}) < \epsilon \text{ for
    all $P \perp P_0$ of finite rank}
  \end{equation}
  where $\tilde{\mu}$ is the canonical Gaussian ``measure'' on $H$.  (Note
  that $\{x : \norm{Ph}_W > \epsilon\}$ is a cylinder set.)
\end{definition}

A quick remark: if $P_0$ satisfies (\ref{measurable-condition}) for
some $\epsilon$, and $P_0 \subset P_0'$, then $P_0'$ also satisfies
(\ref{measurable-condition}) for the same $\epsilon$.  This is because
any $P \perp P_0'$ also has $P \perp P_0$.

In words, this definition requires that $\tilde{\mu}$ puts most of its mass in
``tubular neighborhoods'' of $P_0 H$.  Saying $\norm{Ph}_W >
\epsilon$ means that $x$ is more than distance $\epsilon$ (in
$W$-norm) from $P_0 H$ along  one of the directions from $PH$.

As usual, doing the simplest possible thing doesn't work.

\begin{lemma}
  $\nm_{H}$ is not a measurable norm on $H$.
\end{lemma}

\begin{proof}
  For any finite-rank projection $P$ of some rank $n$, we can find an
  orthonormal basis $\{h_1, \dots, h_n\}$ for $PH$.  Then it's clear
  that $Ph = \sum_{i=1}^n \inner{h}{h_i}_H h_i$, so $\{ h :
  \norm{Ph}_H > \epsilon\} = P^{-1}(\closure{B_{PH}(0,\epsilon)}^C)$,
  where $B_{PH}(0,\epsilon)$ is a ball in $PH$.  By definition of
  $\tilde{\mu}$ we can see that
  \begin{align*}
    \tilde{\mu}(\{ h : \norm{Ph}_H >
    \epsilon\}) &= \mu_n(\closure{B_{\R^n}(0,\epsilon)}^c) \\
    &\ge \mu_n(([-\epsilon, \epsilon]^n)^C) \intertext{(since the ball
    is contained in the cube)}
    &= 1 - \mu_1([-\epsilon,\epsilon])^n.
  \end{align*}
  Thus for any $\epsilon > 0$ and any finite-rank projection $P_0$, if
  we choose $n$ so large that $1 - \mu_1([-\epsilon,\epsilon])^n >
  \epsilon$, then for any projection $P$ of rank $n$ which is
  orthogonal to $P_0$ (of which there are lots), we have
  $\tilde{\mu}(\{h : \norm{Ph}_H > \epsilon\}) > \epsilon$.  So
  $\nm_{H}$ is not measurable.
\end{proof}

As a diversion, let's explicitly verify this for the classical example.

\begin{proposition}
  Let $H$ be the classical Cameron--Martin space of Theorem
  \ref{classical-cameron-martin}.  The supremum norm $\norm{h}_W = \sup_{t \in
  [0,1]} h(t)$ is a measurable norm on $H$.
\end{proposition}

Together with Gross's theorem (Theorem \ref{gross-theorem} below),
this proposition constitutes a construction of Brownian motion: the
completion $W$ of $H$ under $\nm_W$ is precisely $C([0,1])$ (since $H$
is dense in $C([0,1])$), and the measure $\mu$ on $W$ is Wiener
measure (having $H$ as its Cameron--Martin space, we can check that
its covariance function is $a(s,t) = s \wedge t$ as it ought to be).

With the proof we will give, however, it will not be an essentially
new construction.  Indeed, we are going to steal the key ideas from a
construction which is apparently due to L\'evy and can be found in
\cite[Section 2.3]{karatzas-shreve}, which one might benefit from
reading in conjunction with this proof.  In some sense, Gross's
theorem is simply an abstract version of an essential step of that
construction.

\begin{proof}
  Observe up front that by Cauchy--Schwarz
  \begin{equation*}
    \abs{h(t)} = \abs{\int_0^t \dot{h}(t)\,dt} \le t \norm{h}_H
  \end{equation*}
  so taking the supremum over $t \in [0,1]$, we have $\norm{h}_W \le
  \norm{h}_H$.

  We want to choose a good orthonormal basis for $H$.  We use the
  so-called ``Schauder functions'' which correspond to the ``Haar
  functions'' in $L^2([0,1])$.  The Haar functions are given by
  \begin{equation*}
    f^n_k(t) := 
    \begin{cases}
      2^{(n-1)/2}, & \frac{k-1}{2^n} \le t < \frac{k}{2^n} \\
      -2^{(n-1)/2}, & \frac{k}{2^n} \le t \le \frac{k+1}{2^n} \\
      0, & \text{else}
    \end{cases}
  \end{equation*}
  where $f_0^1 = 1$.  Here $k$ should be taken to range over the set
  $I(n)$ consisting of all odd integers between $0$ and $2^n$.  (This
  somewhat peculiar indexing is from Karatzas and Shreve's proof.  It
  may or may not be optimal.)  We note that for $n \ge 1$, we have
  $\int_0^1 f^n_k(t)\,dt = 0$; that for $n > m$, $f^m_j$ is constant
  on the support of $f^n_k$; and that for fixed $n$, $\{f^n_k : k \in
  I(n)\}$ have disjoint supports.  From this it is not hard to check
  that $\{f^n_k : k \in I(n), n \ge 0\}$ are an orthonormal set in
  $L^2(0,1)$.  Indeed, the set forms an orthonormal basis.

  The Schauder functions are defined by $h^n_k(t) := \int_0^t f^n_k(s)\,ds$; since $f \mapsto
  \int_0^\cdot f\,dt$ is an isometric isomorphism from $L^2([0,1])$ to
  $H$, we have that $\{h^n_k : k \in I(n), n \ge 0\}$ is an
  orthonormal basis for $H$.  (We can check easily that it is a basis:
  if $\inner{h}{h^n_k}_H = 0$ then $h(\frac{k-1}{2^n}) =
  h(\frac{k+1}{2^n})$.  If this holds for all $n,k$, then $h(t)=0$ for
  all dyadic rationals $t$, whence by continuity $h=0$.)  Stealing
  Karatzas and Shreve's phrase, $h^n_k$ is a ``little tent'' of height
  $2^{-(n+1)/2}$ supported in $[\frac{k-1}{2^n}, \frac{k+1}{2^n}]$; in
  particular, for each $n$, $\{h^n_k : k \in I(n)\}$ have disjoint
  supports.

  Let $P_m$ be orthogonal projection onto the span of $\{h^n_k : k \in
  I(n), n < m\}$, and suppose $P$ is a projection of finite rank which is orthogonal to $P_m$.  Then for any $h \in H$, we can write
  $Ph$ in terms of the Schauder functions
  \begin{equation*}
    Ph = \sum_{n=m}^\infty \sum_{k \in I(n)} h^n_k \inner{Ph}{h^n_k}_H.
  \end{equation*}
  where the sum converges in $H$ and hence also in $W$-norm,
  i.e. uniformly.  Since for fixed $n$ the $h^n_k$ have disjoint
  supports, we can say
  \begin{equation} \label{W-norm-estimate}
    \begin{split}
    \norm{Ph}_W &\le \sum_{n=m}^\infty  \norm{\sum_{k \in I(n)} h^n_k
    \inner{Ph}{h^n_k}_H}_W \quad \text{(Triangle inequality)} 
\\
      &= \sum_{n=m}^\infty  \max_{k \in I(n)} \norm{h^n_k}_W \abs{
    \inner{Ph}{h^n_k}_H} \quad \text{(since $h^n_k$ have disjoint
    support)}
\\
    &= \sum_{n=m}^\infty  2^{-(n+1)/2} \max_{k \in I(n)} \abs{
    \inner{Ph}{h^n_k}_H}.
    \end{split}
  \end{equation}

  To forestall any nervousness, let us point out that all the following
  appearances of $\tilde{\mu}$ will be to measure sets of the form
  $P^{-1} B$ for our single, fixed $P$, and on such sets $\tilde{\mu}$
  is an honest, countably additive measure (since it is just standard
  Gaussian measure on the finite-dimensional Hilbert space $PH$).  
  Under $\tilde{\mu}$, each $\inner{Ph}{h^n_k}_H$ is a
  centered Gaussian random variable of variance $\norm{P h^n_k}_H^2
  \le 1$ (note that $\inner{Ph}{h^n_k}_H = \inner{Ph}{P h^n_k}_H$, and
  that $P$ is a contraction).  These random variables will be
  correlated in some way, but that will not bother us since we are not
  going to use anything fancier than union bounds.
  
  We recall the standard Gaussian tail estimate: if $N$ is a Gaussian
  random variable with variance $\sigma^2 \le 1$, then $P(\abs{N} \ge
  t) \le C e^{-t^2/2}$ for some universal constant $C$.  (See
  (\ref{gaussian-tail-1d}, or for overkill, Fernique's theorem.)  Thus
  we have for each $n,k$
  \begin{equation*}
    \tilde{\mu}(\{h :  \abs{\inner{Ph}{h^n_k}_H} \ge n\}) \le C e^{-n^2/2}
  \end{equation*}
  and so by union bound
  \begin{equation*}
    \tilde{\mu}\left(\left\{ h : \max_{k \in I(n)} \abs{\inner{Ph}{h^n_k}_H} \ge n\right\}\right) =
    \tilde{\mu}\left(\bigcup_{k \in I(n)} \{h : \abs{\inner{Ph}{h^n_k}_H} \ge n\}\right)
    \le C 2^n e^{-n^2/2}    
  \end{equation*}
  since, being crude, $\abs{I(n)} \le 2^n$.  By another union bound,
  \begin{equation*}
    \tilde{\mu}\left(\bigcup_{n=m}^\infty \left\{ h:  \max_{k \in I(n)}
    \abs{\inner{Ph}{h^n_k}_H} \ge n\right\}\right) \le C \sum_{n=m}^\infty 2^n e^{-n^2/2}.
  \end{equation*}
  On the complement of this event, we have $\max_{k \in I(n)}
    \abs{\inner{Ph}{h^n_k}_H} < n$ for every $n$, and so using
    (\ref{W-norm-estimate}) we have $\norm{Ph}_W < \sum_{n=m}^\infty
    n 2^{-(n+1)/2}$.  Thus we have shown
    \begin{equation}
      \tilde{\mu}\left(\left\{ h: \norm{Ph}_W \ge \sum_{n=m}^\infty n
      2^{-(n+1)/2}\right\}\right) \le C \sum_{n=m}^\infty 2^n e^{-n^2/2}.
    \end{equation}
    Since $\sum n 2^{-(n+1)/2}$ and $\sum
    2^n e^{-n^2/2}$ both converge, for any given $\epsilon > 0$ we may
    choose $m$ so large that $\sum_{n=m}^\infty n 2^{-(n+1)/2} <
    \epsilon$ and $C \sum_{n=m}^\infty
    2^n e^{-n^2/2} < \epsilon$.  Then for any finite-rank projection
    $P$ orthogonal to $P_m$, we have
    \begin{equation}
      \tilde{\mu}(\{ h : \norm{Ph}_W > \epsilon \}) \le \epsilon
    \end{equation}
    which is to say that $\nm_W$ is a measurable norm.
\end{proof}


The name ``measurable'' is perhaps a bit misleading on its face: we
are not talking about whether $h \mapsto \norm{h}_W$ is a measurable function on
$H$.  It just means that $\nm_{W}$ interacts nicely with the
``measure'' $\tilde{\mu}$.  However, $\nm_{W}$ is in fact a measurable
function on $H$, in fact a continuous function, so that it is a weaker
norm than $\nm_{H}$.  

\begin{lemma}\label{cont-embed}
  If $\nm_{W}$ is a measurable norm on $H$, then $\norm{h}_W \le C
  \norm{h}_H$ for some constant $C$.
\end{lemma}

\begin{proof}
  Choose a $P_0$ such that (\ref{measurable-condition}) holds with
  $\epsilon = 1/2$.  Pick some vector $k \in (P_0 H)^\perp$ with
  $\norm{k}_H = 1$.
  Then $Ph = \inner{h}{k} k$ is a (rank-one) projection orthogonal to
  $P_0$, so
  \begin{align*}
    \frac{1}{2} &> \tilde{\mu}(\{h : \norm{Ph}_W > \frac{1}{2}\}) \\ 
    &= \tilde{\mu}(\{h : |\inner{h}{k}| > \frac{1}{2 \norm{k}_W}\}) \\
    &= \mu_1\left(\left[-\frac{1}{2 \norm{k}_W}, \frac{1}{2 \norm{k}_W}\right]^C\right).
  \end{align*}
  Since $\mu_1([-t,t]^C) = 1 - \mu_1([-t,t])$ is a decreasing function
  in $t$, the last line is an increasing function in $\norm{k}_W$, so
  it follows that $\norm{k}_W \le M$ for some $M$.  $k \in (P_0
  H)^\perp$ was arbitrary, and so by scaling we have that $\norm{k}_W
  \le M \norm{k}_H$ for all $k \in (P_0 H)^\perp$.  On the other hand,
  $P_0 H$ is finite-dimensional, so by equivalence of norms we also
  have $\norm{k}_W \le M\norm{k}_H$ for all $k \in P_0 H$, taking $M$
  larger if needed.  Then for
  any $k \in H$, we can decompose $k$ orthogonally as $(k - P_0 k) +
  P_0 k$ and obtain
  \begin{align*}
    \norm{k}_W^2 &= \norm{(k - P_0 k) + P_0 k}_W^2  \\
    &\le (\norm{k - P_0 k}_W + \norm{P_0 k}_W)^2 &&
    \text{Triangle inequality} \\
    &\le 2 (\norm{k - P_0 k}_W^2 + \norm{P_0 k}_W^2) &&
    \text{since $(a+b)^2 \le 2(a^2 + b^2)$, follows from AM-GM}\\
    &\le 2 M^2 (\norm{k - P_0 k}_H^2 + \norm{P_0 k}_H^2) \\
    &= 2 M^2 \norm{k}_H^2 && \text{Pythagorean theorem}
  \end{align*}
  and so the desired statement holds with $C = \sqrt{2} M$.
\end{proof}

\begin{theorem}[Gross \cite{gross-abstract-wiener}]\label{gross-theorem}
  Suppose $H$ is a separable Hilbert space and $\nm_W$ is a measurable
  norm on $H$.  Let $W$ be the completion of $H$ under $\nm_{W}$.
  There exists a Gaussian measure $\mu$ on $(W, \nm_{W})$ whose Cameron-Martin
  space is $(H, \nm_{H})$.
\end{theorem}

\begin{proof}
  We start by constructing a sequence of finite-rank projections $P_n$
  inductively.  First, pick a countable dense sequence $\{v_n\}$ of
  $H$.  Let $P_0 = 0$.  Then suppose that $P_{n-1}$ has been given.
  By the measurability of $\nm_{W}$, for each $n$ we can find a
  finite-rank projection $P_n$ such that for all finite-rank
  projections $P \perp P_n$, we have
  \begin{equation*}
    \tilde{\mu}(\{h \in H : \norm{Ph}_W > 2^{-n}\}) < 2^{-n}.
  \end{equation*}
  As we remarked earlier, we can always choose $P_n$ to be larger, so
  we can also assume that $P_{n-1} \subset P_n$ and also $v_n \in P_n
  H$.  The latter condition ensures that $\bigcup_n P_n H$ is dense
  in $H$, from which it follows that $P_n h \to h$ in
  $H$-norm for all $h \in H$, i.e. $P_n \uparrow I$ strongly.  Let us
  also note that $R_n := P_n - P_{n-1}$ is a finite-rank projection which is
  orthogonal to $P_{n-1}$ (in fact, it is projection onto
  the orthogonal complement of $P_{n-1}$ in $P_n$), and in fact we
  have the orthogonal decomposition $H = \bigoplus_{n=1}^\infty R_n H$.

  Given an orthonormal basis for $P_n H$, we can extend it to an
  orthonormal basis for $P_{n+1} H$.  Repeating this process, we can
  find a sequence $\{h_j\}_{j=1}^\infty$ such that $\{h_1, \dots,
  h_{k_n}\}$ is an orthonormal basis for $P_n H$.  Since $\bigcup P_n
  H$ is dense in $H$, it follows that the entire sequence $\{h_j\}$ is an
  orthonormal basis for $H$.  

  Let $\{X_n\}$ be a sequence of iid standard normal random
  variables defined on some unrelated probability space $(\Omega,
  \mathcal{F}, \mathbb{P})$.  Consider the $W$-valued random variable
  \begin{equation*}
    S_n = \sum_{j=1}^{k_n} X_j h_j.
  \end{equation*}
  Note that $S_n - S_{n-1}$ has a standard normal distribution on the
  finite-dimensional Hilbert space $R_n H$, so by definition of
  $\tilde{\mu}$ we have
  \begin{align*}
    \mathbb{P}(\norm{S_n - S_{n-1}}_W > 2^{-n}) = \tilde{\mu}(\{h \in H
    : \norm{R_n h}_W > 2^{-n}\}) < 2^{-n}.
  \end{align*}
  Thus $\sum_n \mathbb{P}(\norm{S_n - S_{n-1}}_W > 2^{-n}) < \infty$,
  and by Borel--Cantelli, we have that, $\mathbb{P}$-almost surely, $\norm{S_n
  - S_{n-1}}_W \le 2^{-n}$ for all but finitely many $n$.  In
  particular, $\mathbb{P}$-a.s., $S_n$ is Cauchy in $W$-norm, and
  hence convergent to some $W$-valued random variable $S$.

  Let $\mu = \mathbb{P} \circ S^{-1}$ be the law of $S$; $\mu$ is a
  Borel measure on $W$.  If $f \in W^*$, then $f(S) = \lim_{n \to
  \infty} f(S_n) = \lim_{n \to \infty} \sum_{j=1}^{k_n} f(h_j) X_j$ is
  a limit of Gaussian random variables.  Hence by Lemma
  \ref{limit-of-gaussian} $f(S)$ is Gaussian, and moreover we have
  \begin{equation*}
    \infty > \Var(f(S)) = \lim_{n \to \infty} \Var(f(S_n)) = \lim_{n \to
    \infty} \sum_{j=1}^{k_n} |f(h_j)|^2 = \sum_{j=1}^\infty |f(h_j)|^2.
  \end{equation*}
  Pushing forward, we have that $f$ is a Gaussian random variable on
  $(W,\mu)$ with variance $q(f,f) = \sum_{j=1}^\infty
  |f(h_j)|^2 < \infty$.  So $\mu$ is a Gaussian measure and $q$ is its
  covariance form.  Let $H_\mu$ be the Cameron--Martin space
  associated to $(W, \mu)$.  We want to show that $H = H_\mu$
  isometrically.  This is basically just another diagram chase. 

  Let $i$ denote the inclusion map $i : H \hookrightarrow W$.  We know
  by Lemma \ref{cont-embed} that $i$ is 1-1, continuous and has dense
  range, so its adjoint $i^* : W^* \to H$ is also 1-1 and continuous with
  dense range (Exercise \ref{adjoint-exercise}).  Also, we have
  \begin{equation*}
    \norm{i^* f}_H^2 = \sum_{j=1}^\infty |\inner{i^* f}{h_j}_H|^2 =
    \sum_{j=1}^\infty |f(h_j)|^2 = q(f,f) 
  \end{equation*}
  so that $i^* : (W^*,q) \to H$ is an isometry.

  Next, for any $h \in H$ and any $f \in W^*$, Cauchy--Schwarz gives
  \begin{equation*}
    |f(h)|^2 = |\inner{i^* f}{h}_H|^2 \le \norm{i^*f}_H^2 \norm{h}_H^2
     = q(f,f) \norm{h}_H^2
  \end{equation*}
  so that $f \mapsto f(h)$ is a continuous linear functional on
  $(W^*,q)$.  That is, $h \in H_\mu$, and rearranging and taking the
  supremum over $f$ shows $\norm{h}_{H_\mu} \le \norm{h}_H$.  On the
  other hand, if $f_n \in W^*$ with $i^* f_n \to h$ in $H$, we have by
  definition
  \begin{equation*}
    \frac{|f_n(h)|}{\sqrt{q(f_n,f_n)}} \le \norm{h}_{H_\mu}.
  \end{equation*}
  As $n \to \infty$, $f_n(h) = \inner{i^* f_n}{h}_H \to \norm{h}_H^2$,
  and since $i^*$ is an isometry, $q(f_n, f_n) = \norm{i^* f_n}_H^2
  \to \norm{h}_H^2$, so the left side tends to $\norm{h}_H$.  Thus
  $\norm{h}_H = \norm{h}_{h_\mu}$.

  We have shown $H \subset H_\mu$ isometrically; we want equality.
  Note that $H$ is closed in $H_\mu$, since $H$ is complete in
  $H$-norm and hence also in $H_\mu$-norm.  So it
  suffices to show $H$ is dense in $H_\mu$.  Suppose there exists $g
  \in H_\mu$ with $\inner{g}{h}_{H_\mu} = 0$ for all $h \in H$.  If
  $i_\mu : H_\mu \hookrightarrow W$ is the inclusion map, we know that
  $i_\mu^* : (W^*,q) \to H_\mu$ has dense image and is an isometry.
  So choose $f_n \in W^*$ with $i_\mu^* f_n \to g$ in $H_\mu$.   Then
  $f_n$ is $q$-Cauchy, and so $i^* f_n$ converges in $H$-norm to some
  $k \in H$.  But for $h \in H$,
  \begin{equation*}
    \inner{k}{h}_H = \lim \inner{i^* f_n}{h}_H = \lim f_n(h) = \lim
    \inner{i_\mu^* f_n}{h}_{H_\mu} = \inner{g}{h}_{H_\mu} = 0
  \end{equation*}
  so that $k = 0$.  Then
  \begin{equation*}
    \norm{g}_{H_\mu}^2 = \lim \norm{i_\mu^* f_n}_{H_\mu}^2 = \lim
    q(f_n, f_n) = \lim \norm{i^* f_n}_H^2 = 0
  \end{equation*}
  so $g = 0$ and we are done.
\end{proof}

Here is one way to describe what is going on here.  If $h_j$ is an
orthonormal basis for $H$, then $S = \sum X_j h_j$ should be a random
variable with law $\mu$.  However, this sum diverges in $H$ almost
surely (since $\sum |X_j|^2 = \infty$ a.s.).  So if we want it to
converge, we have to choose a weaker norm.

The condition of measurability is not only sufficient but also necessary.

\begin{theorem}
  Let $(W, \mu)$ be an abstract Wiener space with Cameron--Martin
  space $H$.  Then $\nm_{W}$ is a measurable norm on $H$.
\end{theorem}

The first proof of this statement, in this generality, seems to have
appeared in \cite{dudley-feldman-lecam}.  For a nice proof due to Daniel Stroock, see Bruce Driver's notes
\cite{driver-probability}.  

\begin{remark}
  The mere existence of a measurable norm on a given Hilbert space $H$
  is trivial.  Indeed, since all infinite-dimensional separable
  Hilbert spaces are isomorphic, as soon as we have found a measurable
  norm for one Hilbert space, we have found one for any Hilbert
  space.  
\end{remark}

One might wonder if the completion $W$ has any restrictions on its
structure.  Equivalently, which separable Banach spaces $W$ admit
Gaussian measures?  This is a reasonable question, since Banach spaces
can have strange ``geometry.''\footnote{A particularly bizarre example
was given recently in \cite{argyros-haydon}: a separable Banach space
$X$ such that every bounded operator $T$ on $X$ is of the form $T =
\lambda I + K$, where $K$ is a compact operator.  In some sense, $X$
has almost the minimum possible number of bounded operators.}  However, the
answer is that there are no restrictions.

\begin{theorem}[Gross {\cite[Remark 2]{gross-abstract-wiener}}]
  If $W$ is any separable Banach space, there exists a separable
  Hilbert space densely embedded in $W$, on which the $W$-norm is
  measurable.  Equivalently, there exists a non-degenerate Gaussian
  measure on $W$.
\end{theorem}

\begin{proof}
  The finite-dimensional case is trivial, so we suppose $W$ is
  infinite dimensional.    We start with the case of Hilbert spaces.

  First, there exists a separable (infinite-dimensional) Hilbert space
  $W$ with a densely embedded separable Hilbert space $H$ on which the
  $W$-norm is measurable.  Proposition \ref{hs-implies-measurable}
  tells us that, given any separable Hilbert space $H$, we can
  construct a measurable norm $\nm_W$ on $H$ by letting $\norm{h}_W =
  \norm{Ah}_H$, where $A$ is a Hilbert--Schmidt operator on $H$.  Note
  that $\nm_W$ is induced by the inner product $\inner{h}{k}_W =
  \inner{Ah}{Ak}_H$, so if we let $W$ be the completion of $H$ under
  $\nm_W$, then $W$ is a separable Hilbert space with $H$ densely
  embedded.  (We should take $A$ to be injective.  An example of such
  an operator is given by taking an orthonormal basis $\{e_n\}$ and
  letting $A e_n = \frac{1}{n} e_n$.)

  Now, since all infinite-dimensional separable Hilbert spaces are
  isomorphic, this shows that the theorem holds for any separable
  Hilbert space $W$.

  Suppose now that $W$ is a separable Banach space.  By the following
  lemma, there exists a separable Hilbert space $H_1$ densely embedded
  in $W$.  In turn, there is a separable Hilbert space $H$ densely
  embedded in $H_1$, on which the $H_1$-norm is measurable.  The
  $W$-norm on $H$ is weaker than the $H_1$-norm, so it is measurable
  as well.  (Exercise: check the details.)

  Alternatively, there exists a non-degenerate Gaussian measure
  $\mu_1$ on $H_1$.  Push it forward under the inclusion map.  As an
  exercise, verify that this gives a non-degenerate Gaussian measure
  on $W$.
\end{proof}

\begin{lemma}[Gross]
  If $W$ is a separable Banach space, there exists a separable Hilbert
  space $H$ densely embedded in $W$.
\end{lemma}

\begin{proof}
  We repeat a construction of Gross \cite{gross-abstract-wiener}.
  Since $W$ is separable, we may find a countable set $\{z_i\} \subset
  W$ whose linear span is dense in $W$; without loss of generality, we
  can take $\{z_n\}$ to be linearly independent.  We will construct an
  inner product $\inner{\cdot}{\cdot}_K$ on $K = \spanop\{z_i\}$ such
  that $\norm{x}_W \le \norm{x}_K$ for all $x \in K$; thus $K$ will be an
  inner product space densely embedded in $W$.

  We inductively construct a sequence $\{a_i\}$ such that $a_i \ne 0$
  for any real numbers $b_1, \dots, b_n$ with $\sum_{i=1}^n b_i^2 \le
  1$, we have $\norm{\sum_{i=1}^n a_i b_i z_i}_W < 1$.  To begin,
  choose $a_1$ with $0 < |a_1| < \norm{z_1}_W^{-1}$.  Suppose now that $a_1, \dots,
  a_{n-1}$ have been appropriately chosen.  Let $D^n = \{(b_1, \dots,
  b_n) : \sum_{i=1}^n b_n^2 \le 1$ be the closed Euclidean
  unit disk of $\R^n$ and consider the map $f : D^n \times \R \to W$
  defined by
  \begin{equation*}
    f(b_1, \dots, b_n, a) = \sum_{i=1}^{n-1} a_i b_i z_i + a b_n z_n.
  \end{equation*}
  Now $f$ is obviously continuous, and by the induction hypothesis we
  have $f(D^n \times \{0\}) \subset S$, where $S$ is the open unit
  ball of $W$.  So by continuity, $f^{-1}(S)$ is an open set
  containing $D^n \times \{0\}$; hence $f^{-1}(S)$ contains some set
  of the form $D^n \times (-\epsilon, \epsilon)$.  Thus if we choose
  any $a_n$ with $0 \le |a_n| < \epsilon$, we have the desired
  property for $a_1, \dots, a_n$.

  Set $y_i = a_i z_i$; since the $a_i$ are nonzero, the $y_i$ span
  $K$ and are linearly independent.  Let $\inner{\cdot}{\cdot}_K$ be
  the inner product on $K$ which makes the $y_i$ orthonormal; then we
  have $\norm{\sum_{i=1}^n b_i y_i}_K^2 = \sum_{i=1}^n b_i^2$.  By our
  construction, we have that any $x \in K$ with $\norm{x}_K^2 \le 1$
  has $\norm{x}_W < 1$ as well, so $K$ is continuously and densely
  embedded in $W$.  That is to say, the inclusion map $i : (K, \nm_K)
  \to (W, \nm_W)$ is continuous.

  Let $\bar{K}$ be the abstract completion of $K$, so $\bar{K}$ is a
  Hilbert space.  Since $W$ is Banach, the continuous map $i : K \to
  W$ extends to a continuous map $\bar{i} : \bar{K} \to W$ whose image
  is dense (as it contains $K$).  It is possible that $\bar{i}$ is not
  injective, so let $H = (\ker \bar{i})^\perp$ be the orthogonal
  complement in $\bar{K}$ of its kernel.  $H$ is a closed subspace of
  $\bar{K}$, hence a Hilbert space in its own right, and the
  restriction $\bar{i}|_H : H \to W$ is continuous and injective, and
  its range is the same as that of $\bar{i}$, hence still dense in
  $W$.
\end{proof}

\begin{remark}
  The final step of the previous proof (passing to $(\ker
  \bar{i})^\perp$) is missing from Gross's original proof, as was
  noticed by Ambar Sengupta, who asked if it is actually necessary.
  Here is an example to show that it is.

  Let $W$ be a separable Hilbert space with orthonormal basis
  $\{e_n\}_{n=1}^\infty$, and let $S$ be the left shift operator defined by $S e_1
  = 0$, $S e_n = e_{n-1}$ for $n \ge 2$.  Note that the kernel of $S$
  is one-dimensional and spanned by $e_1$.  Let $E$ be the subspace of
  $W$ spanned by the vectors $h_n = e_n - e_{n+1}$, $n = 1, 2, \dots$.
  It is easy to check that $e_1 \notin E$, so the restriction of $S$
  to $E$ is injective.  On the other hand, $E$ is dense in $W$: for
  suppose $x \in E^\perp$.  Since $\inner{x}{h_n}_W = 0$, we have
  $\inner{x}{e_n}_W = \inner{x}{e_{n+1}}_W$, so in fact there is a
  constant $c$ with $\inner{x}{e_n}_W = c$ for all $n$.  But
  Parseval's identity says $\sum_{n=1}^\infty
  |\inner{x}{e_n}_W|^2 = \norm{x}_W^2 < \infty$ so we must have $c=0$
  and thus $x=0$.

  We also remark that $SE$ is also dense in $W$: since $S h_n =
  h_{n-1}$, we actually have $E \subset SE$.

  So we have a separable inner product space $E$, a separable Hilbert
  space $W$, and a continuous injective map $S|_E : E \to W$ with
  dense image, such that the continuous extension of $S|_E$ to the
  completion of $E$ (namely $W$) is not injective (since the extension
  is just $S$ again).

  To make this look more like Gross's construction, we just rename
  things.  Set $K = SE$ and define an inner
  product on $K$ by $\inner{Sx}{Sy}_K = \inner{x}{y}_W$ (this is well
  defined because $S$ is injective on $E$).  Now $K$ is an inner
  product space, continuously and densely embedded in $W$, but the
  completion of $K$ does not embed in $W$ (the continuous extension of
  the inclusion map is not injective, since it is really $S$ in
  disguise).

  The inner product space $(K, \inner{\cdot}{\cdot}_K)$ could actually
  be produced by Gross's construction.  By applying the Gram--Schmidt
  algorithm to $\{h_n\}$, we get an orthonormal set $\{g_n\}$ (with
  respect to $\inner{\cdot}{\cdot}_W$) which
  still spans $E$.  (In fact, $\{g_n\}$ is also an orthonormal basis
  for $W$.)  Take $z_n = S g_n$; the $z_n$s are linearly
  independent and span $K$, which is dense in $W$.  If $\sum_{i=1}^n
  b_i^2 \le 1$, then $\norm{\sum_{i=1}^n b_i z_i}_W = \norm{S
  \sum_{i=1}^n b_i g_i}_W \le 1$ because $S$ is a contraction and the
  $g_i$ are orthonormal.  So we can take $a_i = 1$ in the
  induction.\footnote{Technically, since we were supposed to have
  $\norm{\sum a_i b_i z_i}_W < 1$ with a strict inequality, we should
  take $a_i = c < 1$, and this argument will actually produce $c
  \nm_K$ instead of $\nm_K$, which of course makes no difference.}
  Then of course the inner product which makes the $z_n$ orthonormal
  is just $\inner{\cdot}{\cdot}_K$.

  We can make the issue even more explicit: consider the series
  $\sum_{n=1}^\infty \inner{g_n}{e_1}_W z_n$.  Under $\nm_K$, this
  series is Cauchy, since $z_n$ is orthonormal and $\sum_n
  |\inner{g_n}{e_1}_W|^2 = \norm{e_1}_W^2 = 1$; and its limit is not
  zero, since there must be some $g_k$ with $\inner{g_k}{e_1}_W \ne
  0$, and then we have $\inner{z_k}{\sum_{n=1}^m \inner{g_n}{e_1}_W
  z_n} = \inner{g_k}{e_1}_W$ for all $m \ge k$.  So the series
  corresponds to some nonzero element of the completion $\bar{K}$.
  However, under $\nm_H$, the series converges to zero, since
  $\sum_{n=1}^m \inner{g_n}{e_1}_W z_n = S \sum_{n=1}^m
  \inner{g_n}{e_1}_W g_n \to S e_1 = 0$, using the continuity of $S$
  and the fact that $g_n$ is an orthonormal basis for $W$.
\end{remark}

The following theorem points out that measurable norms are far from
unique.

\begin{theorem}
  Suppose $\nm_W$ is a measurable norm on a Hilbert space $(H,
  \nm_H)$.  Then there exists another measurable norm $\nm_{W'}$ which
  is stronger than $\nm_W$, and if we write $W, W'$ for the
  corresponding completions, the inclusions $H \hookrightarrow W'
  \hookrightarrow W$ are compact.
\end{theorem}

\begin{proof}
  See \cite[Lemma 4.5]{kuo-gaussian-book}.
\end{proof}

\subsection{Gaussian measures on Hilbert spaces}

We have been discussing Gaussian measures on separable Banach spaces
$W$.  This includes the possibility that $W$ is a separable Hilbert
space.  In this case, there is more that can be said about the
relationship between $W$ and its Cameron--Martin space $H$.

Let $H,K$ be separable Hilbert spaces.

\begin{exercise}
  Let $A : H \to K$ be a bounded operator, $A^*$ its adjoint.  Let
  $\{h_n\}$, $\{k_m\}$ be orthonormal bases for $H,K$ respectively.
  Then
  \begin{equation*}
    \sum_{n=1}^\infty \norm{A h_n}_K^2 = \sum_{m=1}^\infty \norm{A^* k_m}_H^2.
  \end{equation*}
\end{exercise}

\begin{definition}\label{hs-def}
  A bounded operator $A : H \to K$ is said to be
  \textbf{Hilbert--Schmidt} if
  \begin{equation*}
    \norm{A}_{HS}^2 = \sum_{i=1}^\infty \norm{A e_n}_K^2 < \infty
  \end{equation*}
  for some orthonormal basis $\{e_n\}$ of $H$.  By the previous
  exercise, this does not depend on the choice of basis, and
  $\norm{A}_{HS} = \norm{A^*}_{HS}$.
\end{definition}

\begin{exercise}
  If $\norm{A}_{L(H,K)}$ denotes the operator norm of $A$, then
  $\norm{A}_{L(H)} \le \norm{A}_{HS}$.
\end{exercise}

\begin{exercise}
  $\nm_{HS}$ is induced by the inner product $\inner{A}{B}_{HS} =
  \sum_{n=1}^\infty \inner{A e_n}{B e_n}_K$ and makes the set of all
  Hilbert--Schmidt operators from  $H$ to $K$ into a Hilbert space.
\end{exercise}

\begin{exercise}
  Every Hilbert--Schmidt operator is compact.  In particular,
  Hilbert--Schmidt operators do not have bounded inverses if $H$ is
  infinite-dimensional.  The identity operator is not
  Hilbert--Schmidt.
\end{exercise}

\begin{exercise}
  If $A$ is Hilbert--Schmidt and $B$ is bounded, then $BA$ and $AB$
  are Hilbert--Schmidt.  So the Hilbert--Schmidt operators form a
  two-sided ideal inside the ring of bounded operators.
\end{exercise}

\begin{exercise}
  If $A$ is a bounded operator on $H$, $H_0$ is a closed subspace of
  $H$, and $A|_{H_0}$ is the restriction of $A$ to $H_0$, then
  $\norm{A|_{H_0}}_{HS} \le \norm{A}_{HS}$. 
\end{exercise}

\begin{lemma}
  If $H$ is a finite-dimensional Hilbert space, $A : H \to K$ is
  linear, and $\mathbf{X}$ has a standard normal distribution on $H$,
  then $\mathbb{E} \norm{A\mathbf{X}}_K^2 = \norm{A}_{HS}^2$.
\end{lemma}

\begin{proof}
  If $\mathbf{Z}$ has a normal distribution on $\R^n$ with covariance
  matrix $\Sigma$, then clearly $E |\mathbf{Z}|^2 = \tr \Sigma$.  The
  covariance matrix of $A \mathbf{X}$ is $A^* A$, and $\tr(A^* A) =
  \norm{A}_{HS}^2$.
\end{proof}

\begin{proposition}\label{hs-implies-measurable}
  Let $A$ be a Hilbert--Schmidt operator on $H$.  Then $\norm{h}_{W} =
  \norm{Ah}_H$ is a measurable norm on $H$.
\end{proposition}

\begin{proof}
  Fix an orthonormal basis $\{e_n\}$ for $H$, and suppose $\epsilon > 0$.  Since
  $\sum_{n=1}^\infty \norm{A e_n}_H^2 < \infty$, we can choose $N$ so
  large that $\sum_{n=N}^\infty \norm{A e_n}_H^2 < \epsilon^3$.  Let
  $P_0$ be orthogonal projection onto the span of $\{e_1, \dots,
  e_{N-1}\}$.  Note in particular that $\norm{A|_{(P_0
  H)^\perp}}_{HS}^2 < \epsilon^3$.  Now suppose $P \perp P_0$ is a
  finite rank projection.  Then
  \begin{equation*}
    \tilde{\mu}(\{h : \norm{Ph}_{W} > \epsilon\}) = 
    \tilde{\mu}(\{h : \norm{APh}_{H} > \epsilon\}) =
    \mathbb{P}(\norm{A \mathbf{X}}_H > \epsilon)
  \end{equation*}
  where $\mathbf{X}$ has a standard normal distribution on $PH$.  By
  the previous lemma,
  \begin{equation*}
    \mathbb{E} \norm{A \mathbf{X}}_H^2 = \norm{A|_{PH}}_{HS}^2 \le
    \norm{A|_{(P_0 H)^\perp}}_{HS}^2 < \epsilon^3
  \end{equation*}
  so Chebyshev's inequality gives $\mathbb{P}(\norm{A \mathbf{X}}_H >
  \epsilon) < \epsilon$ as desired.
\end{proof}

Since the norm $\norm{h}_W = \norm{Ah}_H$ is induced by an inner
product (namely $\inner{h}{k}_W = \inner{Ah}{Ak}_H$), the completion
$W$ is a Hilbert space.  Actually, this is the only way to get $W$ to be
a Hilbert space.  Here is a more general result, due to Kuo.

\begin{theorem}[{\cite[Corollary 4.4]{kuo-gaussian-book}}]
  Let $W$ be a separable Banach space with Gaussian measure $\mu$ and
  Cameron--Martin space $H$, $i : H \to W$ the inclusion map, and let
  $Y$ be some other separable Hilbert space.  Suppose $A : W \to Y$ is a
  bounded operator.  Then $Ai : H \to Y$ (i.e. the restriction of $A$
  to $H \subset W$) is Hilbert--Schmidt, and
  $\norm{Ai}_{HS} \le C \norm{A}_{L(W,Y)}$ for some constant $C$
  depending only on $(W,\mu)$.
\end{theorem}

\begin{proof}
  We consider instead the adjoint $(Ai)^* = i^* A^* : Y \to H$.  Note
  $A^* : Y \to W^*$ is bounded, and $\norm{i^* A^* y}_H^2 = q(A^* y,
  A^* y)$.  So if we fix an orthonormal basis $\{e_n\}$ for $Y$, we have
  \begin{align*}
    \norm{i^* A^*}_{HS}^2 &= \sum_{n=1}^\infty q(A^* e_n, A^* e_n) \\
    &= \sum_{n=1}^\infty \int_W |(A^* e_n)(x)|^2 \mu(dx) \\
    &= \int_W \sum_{n=1}^\infty |(A^* e_n)(x)|^2 \mu(dx) &&
    \text{(Tonelli)} \\
    &= \int_W \sum_{n=1}^\infty |\inner{A x}{e_n}_Y|^2
    \mu(dx) \\
    &= \int_W \norm{Ax}_Y^2 \mu(dx) \\
    &\le \norm{A}_{L(W,Y)}^2 \int_W \norm{x}_W^2 \mu(dx).
  \end{align*}
  By Fernique's theorem we are done.
\end{proof}

\begin{corollary}\label{inclusion-HS}
  If $W$ is a separable Hilbert space with a Gaussian measure $\mu$
  and Cameron--Martin space $H$, then the inclusion $i : H \to W$ is
  Hilbert--Schmidt, as is the inclusion $m : W^* \to K$.
\end{corollary}

\begin{proof}
  Take $Y = W$ and $A = I$ in the above lemma to see that $i$ is
  Hilbert--Schmidt.  To see $m$ is, chase the diagram.
\end{proof}

\begin{corollary}\label{hilbert-in-hilbert}
  Let $\nm_W$ be a norm on a separable
  Hilbert space $H$.  Then the following are equivalent:
  \begin{enumerate}
    \item \label{cor-measurable} $\nm_W$ is measurable and induced by an inner product
    $\inner{\cdot}{\cdot}_W$;
    \item \label{cor-hs} $\norm{h}_W = \norm{Ah}_H$ for some
  Hermitian, positive definite, Hilbert--Schmidt operator $A$ on $H$.
  \end{enumerate}
\end{corollary}

\begin{proof}
  Suppose \ref{cor-measurable} holds.  Then by Gross's theorem
  (Theorem \ref{gross-theorem}) the completion $W$, which is a Hilbert space,
  admits a Gaussian measure with Cameron--Martin space $H$.  Let $i :
  H \to W$ be the inclusion; by Corollary \ref{inclusion-HS} $i$ is
  Hilbert--Schmidt, and so is its adjoint $i^* : W \to H$.  Then $i^*
  i : H \to H$ is continuous, Hermitian, and positive semidefinite.
  It is also positive definite because $i$ and $i^*$ are both
  injective.  Take $A = (i^* i)^{1/2}$.  $A$ is also continuous,
  Hermitian, and positive definite, and we have $\norm{Ah}_H^2 =
  \inner{i^* i h}{h}_H = \inner{ih}{ih}_W = \norm{h}_W^2$.  $A$ is
  also Hilbert--Schmidt since $\sum \norm{A e_n}_H^2 = \sum \norm{i
  e_n}_W^2$ and $i$ is Hilbert--Schmidt.

  The converse is Lemma \ref{hs-implies-measurable}.
\end{proof}

\section{Brownian motion on abstract Wiener space}

Let $(W,H,\mu)$ be an abstract Wiener space.

\begin{notation}
  For $t \ge 0$, let $\mu_t$ be the rescaled measure $\mu_t(A) =
  \mu(t^{-1/2} A)$ (with $\mu_0 = \delta_0$).  It is easy to check
  that $\mu_t$ is a Gaussian measure on $W$ with covariance form
  $q_t(f,g) = t q(f,g)$.  For short, we could call $\mu_t$
  \textbf{Gaussian measure with variance $t$}.
\end{notation}

\begin{exercise}
  If $W$ is finite dimensional, then $\mu_s \sim \mu_t$ for all
  $s,t$.  If $W$ is infinite dimensional, then $\mu_s \perp \mu_t$ for
  $s \ne t$.
\end{exercise}

\begin{lemma}
  $\mu_s * \mu_t = \mu_{s+t}$, where $*$ denotes convolution: $\mu *
  \nu(E) = \iint_{W^2} 1_E(x+y) \mu(dx) \nu(dy)$.  In other words,
  $\{\mu_t : t \ge 0\}$ is a convolution semigroup.
\end{lemma}

\begin{proof}
  Compute Fourier transforms: if $f \in W^*$, then
  \begin{align*}
    \widehat{\mu_s * \mu_t}(f) &= \int_W \int_W e^{i f(x+y)} \mu_s(dx)
    \mu_t(dy) \\
    &= \int_W e^{i f(x)} \mu_s(dx) \int_W e^{i f(y)} \mu_t(dy) \\
    &= e^{- \frac{1}{2} sq(f,f)} e^{- \frac{1}{2} tq(f,f)} \\
    &= e^{- \frac{1}{2} (s+t) q(f,f)} \\
    &= \widehat{\mu_{s+t}}(f).
  \end{align*}
\end{proof}

\begin{theorem}
  There exists a stochastic process $\{B_t, t \ge 0\}$ with values in
  $W$ which is a.s. continuous in $t$ (with respect to the norm
  topology on $W$), has independent increments, and for $t > s$ has
  $B_t - B_s \sim \mu_{t-s}$, with $B_0 = 0$ a.s.  $B_t$ is called
  \textbf{standard Brownian motion on $(W,\mu)$}.
\end{theorem}

\begin{proof}
  Your favorite proof of the existence of one-dimensional Brownian
  motion should work.  For instance, one can use the Kolmogorov
  extension theorem to construct a countable set of $W$-valued random
  variables $\{B_t : t \in E \}$, indexed by the dyadic rationals $E$,
  with independent increments and $B_t - B_s \sim \mu_{t-s}$.  (The
  consistency of the relevant family of measures comes from the
  property $\mu_t * \mu_s = \mu_{t+s}$, just as in the one-dimensional
  case.)  If you are worried that you only know the Kolmogorov
  extension theorem for $\R$-valued random variables, you can use the
  fact that any Polish space can be measurably embedded into $[0,1]$.
  Then the Kolmogorov continuity theorem (replacing absolute values
  with $\nm_W$) can be used to show that, almost surely, $B_t$ is
  H\"older continuous as a function between the metric spaces $E$ and
  $W$.  Use the fact that
  \begin{equation*}
    \mathbb{E} \norm{B_t - B_s}_W^\beta = \int_W \norm{x}_W^\beta
    \mu_{t-s}(dx) = (t-s)^{\beta/2} \int_W \norm{x}_W^\beta \mu(dx) \le
    C (t-s)^{\beta/2}
  \end{equation*}
  by Fernique.  In particular $B_t$ is, almost surely, uniformly
  continuous and so extends to a continuous function on $[0,\infty)$.
\end{proof}

\begin{exercise}
  For any $f \in W^*$, $f(B_t)$ is a one-dimensional Brownian motion
  with variance $q(f,f)$.  If $f_1, f_2, \dots$ are $q$-orthogonal,
  then the Brownian motions $f_1(B_t), f_2(B_t), \dots$ are independent.
\end{exercise}

\begin{question}
  If $h_j$ is an orthonormal basis for $H$, and $B_t^j$ is an iid
  sequence of one-dimensional standard Brownian motions, does
  $\sum_{j=1}^\infty B_t^j h_j$ converge uniformly in $W$, almost
  surely, to a Brownian motion on $W$?  That would be an even easier construction.
\end{question}

\begin{exercise}
  Let $W' = C([0,1], W)$ (which is a separable Banach space) and
  consider the measure $\mu'$ on $W'$ induced by $\{B_t, 0 \le t \le
  1\}$.  Show that $\mu'$ is a Gaussian measure.  For extra credit,
  find a nice way to write its covariance form.
\end{exercise}

\begin{exercise}
  Suppose $W = C([0,1])$ and $\mu$ is the law of a one-dimensional
  continuous Gaussian process $X_s$ with covariance function
  $a(s_1,s_2)$.  Let $Y_{s,t} = B_t(s)$ be the corresponding
  two-parameter process (note $B_t$ is a random element of $C([0,1])$
  so $B_t(s)$ is a random variable).  Show $Y_{s,t}$ is a continuous
  Gaussian process whose covariance function is
  \begin{equation*}
    \mathbb{E}[Y_{s_1,t_1} Y_{s_2, t_2}] = (t_1 \wedge t_2) a(s_1, s_2).
  \end{equation*}
  If $X_s$ is one-dimensional Brownian motion, then $Y_{s,t}$ is
  called the \textbf{Brownian sheet}.
\end{exercise}

$B_t$ has essentially all the properties you would expect a Brownian
motion to have.  You can open your favorite textbook on Brownian
motion and pick most any theorem that applies to $d$-dimensional
Brownian motion, and the proof should go through with minimal
changes.  We note a few important properties here.

\begin{proposition}
  $B_t$ is a Markov process, with transition probabilities $\mathbb{P}^x (B_t \in
  A) = \mu_t^x(A) := \mu(t^{-1/2}(A - x))$.
\end{proposition}

\begin{proof}
  The Markov property is immediate, because $B_t$ has independent
  increments.  Computing the transition probabilities is also very simple.
\end{proof}

\begin{proposition}
  $B_t$ is a martingale.
\end{proposition}

\begin{proof}
  Obvious, because it has independent mean-zero increments.
\end{proof}

\begin{proposition}
  $B_t$ obeys the Blumenthal $0$-$1$ law: let $\mathcal{F}_t =
  \sigma(B_s : 0 \le s \le t)$ and $\mathcal{F}_t^+ = \bigcap_{s > t}
  \mathcal{F}_s$.  Then $\mathcal{F}_0^+$ is $\mathbb{P}^x$-almost
  trivial, i.e. for any $A \in \mathcal{F}_0^+$ and any $x \in W$,
  $\mathbb{P}^x(A) = 0$ or $1$.
\end{proposition}

\begin{proof}
  This holds for any continuous Markov process.
\end{proof}

The transition semigroup of $B_t$ is
\begin{equation*}
  P_t F(x) = E_x F(B_t) = \int F \mu^x_t = \int F(x + t^{1/2}y) \mu(dy)
\end{equation*}
which makes sense for any bounded measurable function.  Clearly $P_t$
is Markovian (positivity-preserving and a contraction with respect to
the uniform norm).

\begin{notation}
  Let $C_b(W)$ denote the space of bounded continuous functions $F : W
  \to \R$.  Let $C_u(W)$ denote the subspace of bounded \emph{uniformly continuous}
  functions.
\end{notation}

\begin{exercise}
  $C_b(W)$ and $C_u(W)$ are Banach spaces.
\end{exercise}

\begin{proposition}
  $P_t$ is a Feller semigroup: if $F$ is continuous, so is $P_t F$.
\end{proposition}

\begin{proof}
  Fix $x \in W$, and $\epsilon > 0$.  Since $\mu_t$ is Radon (see
  Section \ref{radon}), there exists a compact
  $K$ with $\mu_t(K^C) < \epsilon$.

  For any $z \in K$, the function $F(\cdot + z)$ is continuous at $x$,
  so there exists $\delta_z$ such that for any $u$ with $\norm{u}_W <
  \delta_z$, we have $|F(x+z) -F(x+u+z)| < \epsilon$.  The balls $B(z,
  \delta_z/2)$ cover $K$ so we can take a finite subcover $B(z_i,
  \delta_i/2)$.  Now suppose $\norm{x-y}_W < \min \delta_i/2$.  For
  any $z
  \in K$, we may choose $z_i$ with $z \in B(z_i, \delta_i/2)$.  We then have
  \begin{align*}
    |F(x+z) - F(y+z)| &\le |F(x+z) - F(x+z_i)| + |F(x+z_i) - F(y+z)| \\
    &= |F((x + (z - z_i)) + z_i) - F(x + z_i)| \\
    &\quad + |F(x + z_i) - F((x +
     (y-x) + (z-z_i)) + z_i)|
  \end{align*}
  Each term is of the form $|F(x+z_i) - F(x+u+z_i)|$ for some $u$ with
  $\norm{u}_W < \delta_i$, and hence is bounded by $\epsilon$.

  Now we have
  \begin{align*}
    |P_t F(x) - P_t F(y)| &\le \int_W |F(x+z)-F(y+z)| \mu_t(dz) \\
    &= \int_K |F(x+z)-F(y+z)| \mu_t(dz) + \int_{K^C} |F(x+z)-F(y+z)|
    \mu_t(dz) \\
    &\le \epsilon + 2 \epsilon \norm{F}_\infty.
  \end{align*}
\end{proof}

\begin{remark}
  This shows that $P_t$ is a contraction semigroup on $C_b(W)$.  We
  would really like to have a strongly continuous contraction
  semigroup.  However, $P_t$ is not in general strongly continuous on
  $C_b(W)$.  Indeed, take the one-dimensional case $W = \R$, and let
  $f(x) = \cos(x^2)$, so that $f$ is continuous and bounded but not uniformly
  continuous.  One can check that $P_t f$ vanishes at infinity, for
  any $t$.  (For instance, take Fourier transforms, so the convolution
  in $P_t f$ becomes multiplication.  The Fourier transform $\hat{f}$
  is just a scaled and shifted version of $f$; in particular it is
  bounded, so $\widehat{P_t f}$ is integrable.  Then the
  Riemann--Lebesgue lemma implies that $P_t f \in C_0(\R)$.  One can
  also compute directly, perhaps by writing $f$ as the real part of
  $e^{i x^2}$.)  Thus if $P_t f$ were to converge uniformly as $t \to
  0$, the limit would also vanish at infinity, and so could not be
  $f$. (In fact, $P_t f \to f$ pointwise as $t \to 0$, so $P_t f$ does
  not converge uniformly.)
\end{remark}

\begin{remark}
  One should note that the term ``Feller semigroup'' has several
  different and incompatible definitions in the literature, so caution
  is required when invoking results from other sources.  One other common
  definition assumes the state space $X$ is locally compact, and
  requires that $P_t$ be a strongly continuous contraction
  semigroup on $C_0(X)$, the space of continuous functions ``vanishing
  at infinity'', i.e. the uniform closure of the continuous functions
  with compact support.  In our non-locally-compact setting this
  condition is meaningless, since $C_0(W) = 0$.
\end{remark}

\begin{proposition}
  $B_t$ has the strong Markov property.
\end{proposition}

\begin{proof}
  This should hold for any Feller process.  The proof in Durrett looks
  like it would work.
\end{proof}

\begin{theorem}
  $P_t$ is a strongly continuous contraction semigroup on $C_u(W)$.
\end{theorem}

\begin{proof}
  Let $F \in C_u(W)$.  We first check that $P_t F \in C_u(W)$.  It is
  clear that $P_t F$ is bounded; indeed, $\norm{P_t F}_\infty \le \norm{F}_\infty$.  Fix
  $\epsilon > 0$.  There exists $\delta > 0$ such that $|F(x) - F(y)|
  < \epsilon$ whenever $\norm{x-y}_W < \delta$.  For such $x,y$ we
  have
  \begin{equation*}
    |P_t F(x) - P_t F(y)| \le \int |F(x+z)-F(y+z)| \mu_t(dz) \le
     \epsilon.
  \end{equation*}
  Thus $P_t F$ is uniformly continuous.

  Next, we have
  \begin{align*}
    |P_t F(x) - F(x)| &\le \int |F(x + t^{1/2}y) - F(x)| \mu(dy) \\
    &= \int_{\norm{t^{1/2} y}_W < \delta} |F(x + t^{1/2}y) - F(x)|
    \mu(dy) + \int_{\norm{t^{1/2} y}_W \ge \delta} |F(x + t^{1/2}y) -
    F(x)| \mu(dy) \\
    &\le \epsilon + 2 \norm{F}_\infty \mu(\{y : \norm{t^{1/2}y}_W \ge \delta\}).
  \end{align*}
  But $\mu(\{y : \norm{t^{1/2}y}_W \ge \delta\}) =
  \mu(B(0,t^{-1/2})^C) \to 0$ as $t \to 0$.  So for small enough $t$
  we have $|P_t F(x) - F(x)| \le 2 \epsilon$ independent of $x$.
\end{proof}

\begin{remark}
  $C_u(X)$ is not the nicest Banach space to work with here; in
  particular it is not separable, and its dual space is hard to
  describe.  However, the usual nice choices that work in finite
  dimensions don't help us here.  In $\R^n$, $P_t$ is a strongly
  continuous semigroup on $C_0(\R^n)$; but in infinite dimensions
  $C_0(W) = 0$.  

  In $\R^n$, $P_t$ is also a strongly continuous symmetric semigroup
  on $L^2(\R^n, m)$, where $m$ is Lebesgue measure.  In infinite
  dimensions we don't have Lebesgue measure, but we might wonder
  whether $\mu$ could stand in: is $P_t$ a reasonable semigroup on
  $L^2(W,\mu)$?  The answer is emphatically no; it is not even a
  well-defined operator.  First note that $P_t 1 = 1$ for all
  $t$.  Let $e_i \in W^*$ be $q$-orthonormal, so
  that under $\mu$ the $e_i$ are iid $N(0,1)$.  Under $\mu_t$ they are
  iid $N(0,t)$.  Set $s_n(x) = \frac{1}{n} \sum_{i=1}^n |e_i(x)|^2$;
  by the strong law of large numbers, $s_n \to t$, $\mu_t$-a.e.  Let
  $A = \{x : s_n(x) \to 1\}$, so $1_A = 1$ $\mu$-a.e.  On the other
  hand, for any $t > 0$,
  \begin{equation*}
    \int_W P_t 1_A(x)\,\mu(dx) = \int_W \int_W 1_A(x+y) \mu_t(dy) \mu(dx)
    =  (\mu_t * \mu)(A) = \mu_{1+t}(A) = 0.
    \end{equation*}
  Since $P_t 1_A \ge 0$, it must be that $P_t 1_A = 0$, $\mu$-a.e.
  Thus $1$ and $1_A$ are the same element of $L^2(W,\mu)$, but $P_t 1$
  and $P_t 1_A$ are not.
  \end{remark}

We knew that $H$ was ``thin'' in $W$ in the sense that $\mu(H) = 0$.  In
particular, this means that for any $t$, $\mathbb{P}(B_t \in H) = 0$.
Actually more is true.

\begin{proposition}
  Let $\sigma_H = \inf\{t > 0 : B_t \in H\}$.  Then for any $x \in W$,
  $\mathbb{P}^x(\sigma_H = \infty) = 1$.  That is, from any starting
  point, with probability one, $B_t$ \emph{never} hits $H$.  In other
  words, $H$ is polar for $B_t$.
\end{proposition}

\begin{proof}
  Fix $0 < t_1 < t_2 < \infty$.  If $b_t$ is a standard
  one-dimensional Brownian motion started at $0$, let $c = P(\inf\{b_t
  : t_1 \le t \le t_2\} \ge 1)$, i.e. the probability that $b_t$ is
  above $1$ for all times between $t_1$ and $t_2$.  Clearly $c > 0$.
  By the strong Markov property it is clear that for $x_0 > 0$,
  $P(\inf\{b_t + x_0 : t_1 \le t \le t_2\} \ge 1) > c$, and by
  symmetry $P(\sup\{b_t - x_0 : t_1 \le t \le t_2\} \le -1) > c$
  also.  So for any $x_0 \in \R$, we have
  \begin{equation*}
    P(\inf\{|b_t + x_0|^2 : t_1 \le t \le t_2\} \ge 1) > c.
  \end{equation*}

  Fix a $q$-orthonormal basis $\{e_k\} \subset W^*$
  as in Lemma \ref{ek-basis}, so that $B_t \in H$ iff $\norm{B_t}_H^2
  = \sum_k |e_k(B_t)|^2 < \infty$.  Under $\mathbb{P}^x$, $e_k(B_t)$
  are independent one-dimensional Brownian motions with variance 1 and
  starting points $e_k(x)$.  So if we let $A_k$ be the event $A_k =
  \{\inf\{|e_k(B_t)|^2 : t_1 \le t \le t_2\} \ge 1\}$, by the above
  computation we have $\mathbb{P}^x(A_k) > c$.  Since the $A_k$ are independent
  we have $\mathbb{P}^x(A_k \text{ i.o.}) = 1$.  But on the event
  $\{A_k \text{ i.o}\}$ we have $\norm{B_t}_H^2 = \sum_{k=1}^\infty
  |e_k(B_t)|^2 = \infty$ for all $t \in [t_1, t_2]$.  Thus
  $\mathbb{P}^x$-a.s. $B_t$ does not hit $H$ between times $t_1$ and
  $t_2$.  Now let $t_1 \downarrow 0$ and $t_2 \uparrow \infty$ along
  sequences to get the conclusion.
\end{proof}

\begin{remark}
  A priori it is not obvious that $\sigma_H : \Omega \to [0,\infty]$
  is even measurable (its measurability it is defined by an
  uncountable infimum, and there is no apparent way to reduce it to a
  countable infimum), or that $\{\sigma_H = \infty\}$ is a measurable
  subset of $\Omega$.  What we really showed is that there is a
  (measurable) event $A = \{A_k \text{ i.o.}\}$ with $\mathbb{P}^x(A)
  = 1$ and $\sigma_H = \infty$ on $A$.  If we complete the measurable
  space $(\Omega, \mathcal{F})$ by throwing in all the sets which are
  $\mathbb{P}^x$-null for every $x$, then $\{\sigma_h = \infty\}$ will
  be measurable and so will $\sigma_H$.

  In the general theory of Markov processes one shows that under some
  mild assumptions, including the above completion technique, $\sigma_B$ is
  indeed measurable for any Borel (or even analytic) set $B$.
\end{remark}

\section{Calculus on abstract Wiener space}

The strongly continuous semigroup $P_t$ on $C_u(W)$ has a generator
$L$, defined by
\begin{equation*}
  Lf =  \lim_{t \downarrow 0} \frac{1}{t} (P_t f - f).
\end{equation*}
This is an unbounded operator on $C_u(W)$ whose domain $D(L)$ is the set of
all $f$ for which the limit converges in $C_u(W)$.  It is a general
fact that $L$ is densely defined and closed.

In the classical setting where $W = \R^n$ and $\mu$ is standard
Gaussian measure (i.e. $q = \inner{\cdot}{\cdot}_{\R^n}$), so that
$B_t$ is standard Brownian motion, we know that $L = - \frac{1}{2} \Delta$ is the
Laplace operator, which sums the second partial derivatives in all
orthogonal directions.  Note that ``orthogonal'' is with respect to
the Euclidean inner product, which is also the Cameron--Martin inner
product in this case.

We should expect that in the setting of
abstract Wiener space, $L$ should again be a second-order differential
operator that should play the same role as the Laplacian.  So we need
to investigate differentiation on $W$.

\begin{definition}
  Let $W$ be a Banach space, and $F : W \to \R$ a function.  We say
  $F$ is \textbf{Fr\'echet differentiable} at $x \in W$ if there
  exists $g_x \in W^*$ such that, for any sequence $W \ni y_n \to 0$ in
  $W$-norm,
  \begin{equation*}
    \frac{F(x+y_n) - F(x) - g_x(y_n)}{\norm{y_n}_W} \to 0.
  \end{equation*}
  $g_x$ is the Fr\'echet derivative of $F$ at $x$.  One could write
  $F'(x) = g_x$.  It may be helpful to think in terms of directional
  derivatives and write $\partial_y F(x) = F'(x)y = g_x(y)$.
\end{definition}

\begin{example}
  Suppose $F(x) = \phi(f_1(x), \dots, f_n(x))$ is a cylinder
  function.  Then
  \begin{equation*}
    F'(x) = \sum_{i=1}^n \partial_i \phi(f_1(x), \dots, f_n(x)) f_i.
  \end{equation*}
\end{example}

As it turns out, we will be most interested in differentiating in directions $h
\in H$, since in some sense that is really what the usual Laplacian
does.  Also, Fr\'echet differentiability seems to be too much to ask
for; according to references in Kuo \cite{kuo-gaussian-book}, the set of continuously Fr\'echet
differentiable functions is not dense in $C_u(W)$.

\begin{definition}
  $F : W \to \R$ is \textbf{$H$-differentiable} at $x \in W$ if there
  exists $g_x \in H$ such that for any sequence $H \ni h_n \to 0$ in
  $H$-norm,
  \begin{equation*}
    \frac{F(x+h_n) - F(x) - \inner{g_x, h_n}_H}{\norm{h_n}_H} \to 0.
  \end{equation*}
  We will denote the element $g_x$ by $DF(x)$, and we have $\partial_h
  F(x) = \inner{DF(x)}{h}_H$.  $DF : W \to H$ is sometimes called the
  \textbf{Malliavin derivative} or \textbf{gradient} of $F$.
\end{definition}

\begin{example}
  For a cylinder function $F(x) = \phi(f_1(x), \dots, f_n(x))$, we
  have
  \begin{equation*}
    \inner{DF(x)}{h}_H = \sum_{i=1}^n  \partial_i \phi(f_1(x), \dots, f_n(x)) f_i(h)
  \end{equation*}
  or alternatively
  \begin{equation}\label{DF-cylinder}
    DF(x) = \sum_{i=1}^n   \partial_i \phi(f_1(x), \dots, f_n(x)) J f_i
  \end{equation}
\end{example}

We know that for $F \in C_u(W)$, $P_t F$ should belong to the domain
of the generator $L$, for any $t > 0$.  The next proposition shows
that we are on the right track with $H$-differentiability.

\begin{proposition}
  For $F \in C_u(W)$ and $t > 0$, $P_t F$ is $H$-differentiable, and
  \begin{equation*}
    \inner{D P_t F(x)}{h}_H = \frac{1}{t} \int_W F(x+y) \inner{h}{y}_H\, \mu_t(dy).
  \end{equation*}
\end{proposition}

\begin{proof}
  It is sufficient to show that
  \begin{equation*}
    P_t F(x+h) - P_t F(x) = \frac{1}{t} \int_W F(x+y) \inner{h}{y}_H\,
    \mu_t(dy) + o(\norm{h}_H)
  \end{equation*}
  since the first term on the right side is a bounded linear
  functional of $h$ (by Fernique's theorem).  The Cameron--Martin
  theorem gives us
  \begin{equation*}
    P_t F(x+h) = \int_W F(x+y)\, \mu_t^h(dy) = \int_W F(x+y) J_t(h,y)\, \mu_t(dy)
  \end{equation*}
  where
  \begin{equation*}
    J_t(h,y) = \exp\left(-\frac{1}{2t} \norm{h}_H^2 + \frac{1}{t} \inner{h}{y}_H\right)
  \end{equation*}
  is the Radon--Nikodym derivative, or in other words the ``Jacobian
  determinant.''  Then we have
  \begin{align*}
    P_t F(x+h) - P_t F(x) = \int_W F(x+y) (J_t(h,y) - 1) \,\mu_t(dy).
  \end{align*}
  Since $J_t(0,y) = 1$, we can write $J_t(h,y) - 1 = \int_0^1
  \frac{d}{ds} J_t(sh,y)\,ds$ by the fundamental theorem of calculus.
  Now we can easily compute that $\frac{d}{ds} J_t(sh,y) = \frac{1}{t}
  (\inner{h}{y}_H - s \norm{h}_H^2) J_t(sh,y)$, so we have
  \begin{align*}
    P_t F(x+h) - P_t F(x) &= \frac{1}{t} \int_W F(x+y) \int_0^1 
    (\inner{h}{y}_H - s \norm{h}_H^2) J_t(sh,y)\,ds \,\mu_t(dy) \\
    &= \frac{1}{t} \int_W F(x+y) \inner{h}{y}_H \mu_t(dy) \\ 
    &\quad +
    \frac{1}{t} \int_W
    F(x+y) \int_0^1 
    \inner{h}{y}_H (J_t(sh,y)-1)\,ds \,\mu_t(dy) && (\alpha) \\
    &\quad - \frac{1}{t} \int_W F(x+y) \int_0^1 s \norm{h}_H^2
    J_t(sh,y) \,ds\,\mu_t(dy) && (\beta).
  \end{align*}
  So it remains to show that the remainder terms $\alpha, \beta$ are
  $o(\norm{h}_H)$.

  To estimate $\alpha$, we crash through with absolute values and use
  Tonelli's theorem and Cauchy--Schwarz to obtain
  \begin{align*}
    \abs{\alpha} &\le \frac{\norm{F}_\infty}{t} \int_0^1 \int_W
    \abs{\inner{h}{y}_H} \abs{J_t(sh,y) - 1}\,\mu_t(dy)\,ds \\
    &\le \frac{\norm{F}_\infty}{t} \int_0^1 \sqrt{\int_W
    \abs{\inner{h}{y}_H}^2\,\mu_t(dy) \int_W \abs{J_t(sh,y) - 1}^2\,\mu_t(dy)}\,ds.
  \end{align*}
  But $\int_W
    \abs{\inner{h}{y}_H}^2\,mu_t(dy) = t \norm{h}_H^2$ (since under
    $\mu_t$, $\inner{h}{\cdot}_H \sim N(0, t \norm{h}_H^2)$).  Thus
    \begin{equation*}
      \abs{\alpha} \le \frac{\norm{F}_\infty}{\sqrt{t}} \norm{h}_H
      \int_0^1 \sqrt{\int_W \abs{J_t(sh,y) - 1}^2\,\mu_t(dy)}\,ds.
    \end{equation*}

    Now, a quick computation shows
    \begin{align*}
      \abs{J_t(sh,y)-1}^2 = e^{s^2 \norm{h}_H^2/t} J_t(2sh, y) - 2
      J_t(sh,y) + 1.
    \end{align*}
    But $J_t(g,y) \mu_t(dy) = \mu_t^g(dy)$ is a probability measure
    for any $g \in H$, so integrating with respect to $\mu_t(dy)$ gives
    \begin{align*}
      \int_W \abs{J_t(sh,y) - 1}^2\,\mu_t(dy) = e^{s^2 \norm{h}_H^2/t}
      - 1.
    \end{align*}
    So we have
    \begin{equation*}
      \int_0^1 \sqrt{\int_W \abs{J_t(sh,y) - 1}^2\,\mu_t(dy)}\,ds =
      \int_0^1 \sqrt{e^{s^2 \norm{h}_H^2/t}
      - 1}\,ds = o(1)
    \end{equation*}
    as $\norm{h}_H \to 0$, by dominated convergence.  Thus we have
    shown $\alpha = o(\norm{h}_H)$.

    The $\beta$ term is easier: crashing through with absolute values
    and using Tonelli's theorem (and the fact that $J_t \ge 0$), we
    have
    \begin{align*}
      \abs{\beta} &\le \frac{1}{t} \norm{F}_\infty \norm{h}_H^2
      \int_0^1 s \int_W J_t(sh,y) \,\mu_t(dy)\,ds \\
      &= \frac{1}{t} \norm{F}_\infty \norm{h}_H^2
      \int_0^1 s \cancelto{1}{\int_W \,\mu_t^{sh}(dy)}\,ds \\
      &= \frac{1}{2t} \norm{F}_\infty \norm{h}_H^2 = o(\norm{h}_H).
    \end{align*}
\end{proof}

With more work it can be shown that $P_t F$ is in fact infinitely $H$-differentiable.

\begin{question}
  Kuo claims that the second derivative of $P_t F$ is given by
  \begin{equation*}
    \inner{D^2 P_t F h}{k}_H = \frac{1}{t} \int_W F(x+y)
    \left(\frac{\inner{h}{y}_H \inner{k}{y}_H}{t} -
    \inner{h}{k}\right) \mu_t(dy).
  \end{equation*}
  If the generator $L$ is really the Laplacian $\Delta$ defined below,
  then in particular $D^2 P_t F$ should be trace class.  But this
  doesn't seem to be obvious from this formula.  In particular, if we
  let $h = k = h_n$ and sum over an orthonormal basis $h_n$, the
  obvious approach of interchanging the integral and sum doesn't work,
  because we get an integrand of the form $\sum_n (\xi_n^2 - 1)$ where
  $\xi_n$ are iid $N(0,1)$, which diverges almost surely.
\end{question}

\begin{corollary}
  The (infinitely) $H$-differentiable functions are dense in $C_u(W)$.
\end{corollary}

\begin{proof}
  $P_t$ is a strongly continuous semigroup on $C_u(W)$, so for any $F
  \in C_u(W)$ and any sequence $t_n \downarrow 0$, we have $P_{t_n} F \to F$
  uniformly, and we just showed that $P_{t_n} F$ is $H$-differentiable. 
\end{proof}

Now, on the premise that the $H$ inner product should play the same
role as the Euclidean inner product on $\R^n$, we define the Laplacian
as follows.

\begin{definition}
  The \textbf{Laplacian} of a function $F : W \to \R$ is
  \begin{equation*}
    \Delta F(x) = \sum_{k=1}^\infty \partial_{h_k} \partial_{h_k} F(x)
  \end{equation*}
  if it exists, where $\{h_k\}$ is an orthonormal basis for $H$.
  (This assumes that $F$ is $H$-differentiable, as well as each
  $\partial_h F$.)
\end{definition}

\begin{example}
  If $F : W \to \R$ is a cylinder function as above, then
  \begin{equation*}
    \Delta F(x) = \sum_{i,j=1}^n \partial_i \partial_j \phi(f_1(x),
    \dots, f_n(x)) q(f_i, f_j).
  \end{equation*}
\end{example}

\begin{theorem}
  If $F$ is a cylinder function, then $F \in D(L)$, and $LF = -
  \frac{1}{2} \Delta F$.
\end{theorem}

\begin{proof}
  We have to show that
  \begin{equation*}
    \frac{P_t F(x) - F(x)}{t} \to \frac{1}{2}\Delta F(x)
  \end{equation*}
  uniformly in $x \in W$.  As shorthand, write \begin{align*}
    G_i(x) &= \partial_i
  \phi(f_1(x), \dots, f_n(x)) \\
  G_{ij}(x) &= \partial_i
  \partial_j \phi(f_1(x), \dots, f_n(x))
  \end{align*} so that $\Delta F(x) =
  \sum_{i,j=1}^n G_{ij}(x) q(f_i, f_j)$.  Note that $G_{ij}$ is Lipschitz.

  First note the following identity for $\alpha \in C^2([0,1])$, which
  is easily checked by integration by parts:
  \begin{equation}
    \alpha(1) - \alpha(0) = \alpha'(0) + \int_0^1 (1-s) \alpha''(s)\,ds
  \end{equation}
  Using $\alpha(s) = F(x+sy)$, we have
  \begin{align*}
    F(x+y) - F(x) = \sum_{i=1}^n G_i(x) f_i(y) + \int_0^1 (1-s)
    \sum_{i,j=1}^n G_{ij}(x+sy) f_i(y) f_j(y)\,ds.
  \end{align*}
  Integrating with respect to $\mu_t(dy)$, we have
  \begin{align*}
    P_t F(x) - F(x) &= \sum_{i=1}^n G_i(x) \cancelto{0}{\int_W f_i(y)
    \,\mu_t(dy)} + \sum_{i,j=1}^n \int_W f_i(y) f_j(y) \int_0^1 (1-s)
    G_{ij}(x+sy)\,ds \,\mu_t(dy) \\
    &= t \sum_{i,j=1}^n \int_W f_i(y) f_j(y) \int_0^1 (1-s)
    G_{ij}(x+st^{1/2}y)\,ds \,\mu(dy)
  \end{align*}
  rewriting the $\mu_t$ integral in terms of $\mu$ and using the
  linearity of $f_i, f_j$.  Now if we add and subtract $G_{ij}(x)$
  from $G_{ij}(x+st^{1/2}y)$, we have
  \begin{align*}
    \int_W f_i(y) f_j(y) 
    G_{ij}(x)  \,\mu(dy) \int_0^1 (1-s) \,ds = \frac{1}{2} G_{ij}(x)
    \int_W f_i(y) f_j(y)\,\mu(dy) = \frac{1}{2} G_{ij}(x) q(f_i, f_j)
  \end{align*}
  and, if $C$ is the Lipschitz constant of $G_{ij}$,
  \begin{align*}
    & \abs{\int_W f_i(y) f_j(y) \int_0^1 (1-s)
    (G_{ij}(x+st^{1/2}y) - G_{ij}(x))\,ds \,\mu(dy)} \\
 &\le
    \norm{f_i}_{W^*} \norm{f_j}_{W^*} \int_W \norm{y}_W^2 \int_0^1
    (1-s) (Cst^{1/2}\norm{y}_W)\,ds\,\mu(dy) \\
    &\le
    C t^{1/2} \norm{f_i}_{W^*} \norm{f_j}_{W^*} \int_0^1 (s-s^2)\,ds
    \int_0^1 \norm{y}_W^3\,\mu(dy).
  \end{align*}
  The $\mu$ integral is finite by Fernique's theorem, so this goes to
  $0$ as $t \to 0$, independent of $x$.  Thus we have shown
  \begin{equation*}
    P_t F(x) - F(x) = \frac{t}{2}\left(\sum_{i,j=1}^n G_{ij}(x) q(f_i,
    f_j) + o(1)\right) 
  \end{equation*}
  uniformly in $x$, which is what we wanted.
\end{proof}

\begin{question}
  Kuo \cite{kuo-gaussian-book} proves the stronger statement that this holds for $F$ which are
  (more or less) Fr\'echet-$C^2$.
  Even this is not quite satisfactory, because, as claimed by Kuo's
  references, these functions are not dense in $C_u(W)$.  In
  particular, they are not a core for $L$.  Can we produce a
  Laplacian-like formula for $L$ which makes sense and holds on a core
  of $L$?
\end{question}

\subsection{Some $L^p$ theory}

Here we will follow Nualart \cite{nualart} for a while.

It is worth mentioning that Nualart's approach (and notation) are a
bit different from ours.  His setting is a probability space $(\Omega,
\mathcal{F}, P)$ and a ``process'' $W$, i.e. a family $\{W(h) : h \in
H\}$ of jointly Gaussian random variables indexed by a Hilbert space
$H$, with the property that $E[W(h) W(k)] = \inner{h}{k}_H$.  This
includes our setting: take an abstract Wiener space $(B, H, \mu)$, set
$\Omega = B$, $P = \mu$, and $W(h) = \inner{h}{\cdot}_H$.  We will
stick to our notation.

We want to study the properties of the Malliavin derivative $D$ as an
unbounded operator on $L^p(W,\mu)$, $p \ge 1$.  If we take the domain
of $D$ to be the smooth cylinder functions $\mathcal{F}
C_c^\infty(W)$, we have a densely defined unbounded operator from
$L^p(W,\mu)$ into the vector-valued space $L^p(W,\mu; H)$.  (Note that
$\norm{DF(x)}_H$ is bounded as a function of $x$, so there is no
question that $DF \in L^p(W; H)$.)

\begin{lemma}[Integration by parts]
  Let $F \in \mathcal{F} C_c^\infty(W)$ be a cylinder function, and $h
  \in H$.  Then
  \begin{equation}\label{byparts-eq}
    \int_W \inner{DF(x)}{h}_H \mu(dx) = \int_W F(x) \inner{h}{x}_H \mu(dx).
  \end{equation}
\end{lemma}

\begin{proof}
  It is easy to see that both sides of
  (\ref{byparts-eq}) are bounded linear functionals with respect to
  $h$, so it suffices to show that (\ref{byparts-eq}) holds for all
  $h$ in a dense subset: hence suppose $h = i^* f$ for some $f \in
  W^*$.

  Now basically the proof is to reduce to the finite dimensional case.
  By adjusting $\phi$ as needed, there is no loss of generality in
  writing $F(x) = \phi(f_1(x), \dots, f_n(x))$ where $f_1 = f$.  We
  can also apply Gram--Schmidt and assume that all the $f_i$ are
  $q$-orthonormal.  Then $\inner{DF(x)}{h}_H = \partial_1 \phi(f_1(x),
  \dots, f_n(x))$.  The $f_i$ are iid $N(0,1)$ random variables under
  $\mu$, so we have
  \begin{align*}
    \int_W \inner{DF(x)}{h}_H \,\mu(dx) &= \int_{\R^n} \partial_1
    \phi(x_1, \dots, x_n) \frac{1}{(2\pi)^{n/2}} e^{-|x|^2/2} \,dx \\
    &= \int_{\R^n} \phi(x_1, \dots, x_n) x_1 \frac{1}{(2\pi)^{n/2}}
    e^{-|x|^2/2}\,dx \\
    &= \int_{W} \phi(f_1(x), \dots, f_n(x)) f_1(x) \,\mu(dx) \\
    &= \int_W F(x) \inner{h}{x}_H \,\mu(dx)
  \end{align*}
  so we are done.  
\end{proof}

We also observe that $D$ obeys the product rule: $D(F \cdot G) = F
\cdot DG + DF \cdot G$.  (This is easy to check by assuming, without
loss of generality, that we have written $F,G$ in terms of the same
functionals $f_1, \dots, f_n \in W^*$.)  Applying the previous lemma
to this identity gives:

\begin{equation}\label{product-byparts}
  \int_W G(x) \inner{DF(x)}{h}_H \,\mu(dx) = \int_W (-F(x)
  \inner{DG(x)}{h}_H + F(x) G(x) \inner{h}{x}_H) \mu(dx).
\end{equation}

We can use this to prove:

\begin{proposition}
  The operator $D : L^p(W,\mu) \to L^p(W,\mu; H)$ is closable.
\end{proposition}

\begin{proof}
  Suppose that $F_n \in \mathcal{F} C_c^\infty(W)$ are converging to
  0 in $L^p(W,\mu)$, and that $D F_n \to \eta$ in $L^p(W;H)$.  We have
  to show $\eta = 0$.  It is sufficient to show that $\int_W
  \inner{\eta(x)}{h} G(x)\,mu(dx) = 0$ for all $h \in H$ and $G \in
  \mathcal{F} C_c^\infty(W)$.  (Why?)  Now applying
  (\ref{product-byparts}) we have
  \begin{equation*}
    \int_W G(x) \inner{DF_n(x)}{h}_H \mu(dx) = -\int_W F_n(x)
    \inner{DG(x)}{h}_H + \int_W F_n(x) G(x) \inner{h}{x}_H \mu(dx).
  \end{equation*}
  As $n \to \infty$, the left side goes to $\int_W G(x)
  \inner{\eta(x)}{h}_H \mu(dx)$.  The first term on the right side
  goes to 0 (since $DG \in L^q(W;H)$ and hence $\inner{DG}{h} \in
  L^q(W)$) as does the second term (since $G$ is bounded and
  $\inner{h}{\cdot} \in L^q(W)$ because it is a Gaussian random variable).
\end{proof}

Now we can define the Sobolev space $\mathbb{D}^{1,p}$ as the
completion of $\mathcal{F} C_c^\infty(W)$ under the norm
\begin{equation*}
  \norm{F}_{\mathbb{D}^{1,p}}^p = \int_W |F(x)|^p + \norm{DF}_H^p \mu(dx).
\end{equation*}
Since $D$ was closable, $\mathbb{D}^{1,p} \subset L^p(W,\mu)$.  We can
also iterate this process to define higher derivatives $D^k$ and
higher order Sobolev spaces.

\begin{lemma}
  If $\phi \in C^\infty(\R^n)$ and $\phi$ and its first partials have
  polynomial growth, then $F(x) = \phi(f_1(x), \dots, f_n(x)) \in \mathbb{D}^{1,p}$.
\end{lemma}

\begin{proof}
  Cutoff functions.
\end{proof}

\begin{lemma}\label{polynomials-dense}
  The set of functions $F(x) = p(f_1(x), \dots, f_n(x))$ where $p$ is
  a polynomial in $n$ variables and $f_i \in W^*$, is dense in $L^p(W,\mu)$.
\end{lemma}

\begin{proof}
  See \cite[Theorem 39.8]{driver-probability}.
\end{proof}

Let $H_n(s)$ be the $n$'th \textbf{Hermite polynomial} defined by
\begin{equation*}
  H_n(s) = \frac{(-1)^n}{n!} e^{s^2/2} \frac{d^n}{ds^n} e^{-s^2/2}.
\end{equation*}
$H_n$ is a polynomial of degree $n$.  Fact: $H_n' = H_{n-1}$, and
$(n+1) H_{n+1}(s) = s H_n(s) - H_{n-1}(s)$.  In particular, $H_n$ is
an eigenfunction of the one-dimensional Ornstein--Uhlenbeck operator
$Af = f'' - xf'$ with eigenvalue $n$.

Also, we have the property that if $X,Y$ are jointly Gaussian with
variance $1$, then
\begin{equation*}
  E[H_n(X) H_m(Y)] = 
  \begin{cases}
    0, & n \ne m \\
    \frac{1}{n!} E[XY]^n, & n=m.
  \end{cases}
\end{equation*}
(See Nualart for the proof, it's simple.)

This implies:

\begin{proposition}
  If $\{e_i\} \subset W^*$ is a $q$-orthonormal basis, then the functions
  \begin{equation*}
    F_{n_1, \dots, n_k}(x) = \prod_i \sqrt{n_i!} H_{n_i}(e_i(x))
  \end{equation*}
  are an orthonormal basis for $L^2(W,\mu)$.
\end{proposition}

If $\mathcal{H}_n$ is the closed span of all $F_{n_1, \dots, n_k}$ with $n_1
+ \dots + n_k = n$ (i.e. multivariable Hermite polynomials of degree
$n$) then we have an orthogonal decomposition of $L^2$.  Let $J_n$ be
orthogonal projection onto $\mathcal{H}_n$.  This decomposition is
called \textbf{Wiener chaos}.  Note $\mathcal{H}_0$ is the constants,
and $\mathcal{H}_1 = K$.  For $n \ge 2$, the random variables in
$\mathcal{H}_n$ are not normally distributed.

We can decompose $L^2(W;H)$ in a similar way: if $h_j$ is an
orthonormal basis for $H$, then $\{ F_{n_1, \dots, n_k} h_j \}$ is an
orthonormal basis for $L^2(W;H)$, and if $\mathcal{H}_n(H)$ is the
closed span of functions of the form $F h$ where $F \in
\mathcal{H}_n$, $h \in H$, then $L^2(W;H) = \bigoplus
\mathcal{H}_n(H)$ is an orthogonal decomposition, and we again use
$J_n$ to denote the orthogonal projections.  We could do the same for
$L^2(W; H^{\otimes m})$.

Note that
\begin{align*}
  D F_{n_1, \dots, n_k}(x) &= \sum_{j=1}^k \sqrt{n_j!}
  H_{n_j-1}(e_j(x)) \prod_{i \ne j} \sqrt{n_i!} H_{n_i}(e_i(x)) J e_j
  \\
  &= \sum_{j=1}^k \sqrt{n_j} F_{n_1, \dots, n_j-1, \dots, n_k}(x) J e_j.
\end{align*}

We can see from this that $\norm{D F_{n_1, \dots, n_k}}_{L^2(W; H)}^2
= n_1 + \dots + n_k$, or for short, 
$\norm{D F_\alpha}_{L^2(W;H)}^2 =
|\alpha|$.  Also, if $\alpha \ne \beta$, $\inner{D F_\alpha}{D
F_\beta}_{L^2(W;H)} = 0$.  

\begin{lemma}
  For each $n$, $\mathcal{H}_n \subset \mathbb{D}^{1,2}$.  Moreover,
  $\{F_\alpha : |\alpha| = n\}$ are an orthogonal basis for
  $\mathcal{H}_n$ with respect to the $\mathbb{D}^{1,2}$ inner
  product, and $\norm{F_\alpha}^2 = 1+n$.
\end{lemma}

\begin{proof}
  Since $\{F_{\alpha} : |\alpha| = n\}$ are an orthonormal basis for
  $\mathcal{H}_n$, we can write $F = \sum_{i=1}^\infty a_i
  F_{\alpha_i}$ where $\sum a_i^2 < \infty$ and the $\alpha_i$ are
  distinct.  Let $F_m = \sum_{i=1}^m a_i F_{\alpha_i}$, so that $F_m
  \to F$ in $L^2(W)$.  Clearly $F_m \in \mathbb{D}^{1,2}$ and $D F_m =
  \sum_{i=1}^m a_i D F_{\alpha_i}$.  Now we have $\inner{D
  F_{\alpha_i}}{D F_{\alpha_j}}_{L^2(W;H)} = n \delta_{ij}$, so for $k
  \le m$ we have
  \begin{equation*}
    \norm{D F_m - D F_k}_{L^2(W;H)}^2 = \norm{\sum_{i=k}^m a_i D
    F_{\alpha_i}}_{L^2(W;H)} =  n \sum_{i=k}^m a_i^2
  \end{equation*}
  which goes to $0$ as $m,k \to \infty$.  So we have $F_m \to F$ in
  $L^2(W)$ and $D F_m$ Cauchy in $L^2(W;H)$.  Since $D$ is closed, we
  have $F \in \mathbb{D}^{1,2}$.

  In fact, we have shown that $F$ is a $\mathbb{D}^{1,2}$-limit of
  elements of the span of $\{F_\alpha\}$, and we know the $F_\alpha$
  are $\mathbb{D}^{1,2}$-orthogonal.
\end{proof}

Note that $\mathcal{H}_n$ are thus pairwise orthogonal closed subsets
of $\mathbb{D}^{1,2}$.  Furthermore, $D$ maps $\mathcal{H}_n$ into
$\mathcal{H}_{n-1}(H)$.

\begin{proposition}
  The span of all $\mathcal{H}_n$ is dense in $\mathbb{D}^{1,2}$, so
  we can write $\mathbb{D}^{1,2} = \bigoplus \mathcal{H}_n$.
\end{proposition}

\begin{proof}
  This proof is taken from \cite{schmuland92}.

  We begin with the finite dimensional case.  Let $\mu_k$ be standard
  Gaussian measure on $\R^k$, and let $\phi \in C_c^\infty(\R^k)$.  We
  will show that there is a sequence of polynomials $p_m$ such that
  $p_m \to \phi$ and $\partial_i p_m \to \partial_i \phi$ in
  $L^2(\R^k, \mu_k)$ for all $k$.

  For continuous $\psi : \R^k \to \R$, let
  \begin{equation*}
    I_i \psi(x_1, \dots, x_k) = \int_0^{x_i} \psi(x_1, \dots, x_{i-1},
    y, x_{i+1}, \dots, x_k)\,dy.
  \end{equation*}
  By Fubini's theorem, all operators $I_i, 1 \le i \le k$ commute.  If
  $\psi \in L^2(\mu_k)$ is continuous, then $I_i \psi$ is also
  continuous, and $\partial_i I_i \psi = \psi$.  Moreover,
  \begin{align*}
    \int_{0}^\infty |I_i \psi (x_1, \dots, x_k)|^2 e^{-x_i^2/2} dx_i &=
    \int_{0}^\infty \abs{\int_0^{x_i} \psi(\dots, y,\dots)\,dy}^2
    e^{-x_i^2/2} dx_i \\
    &\le \int_0^\infty \int_0^{x_i} |\psi(\dots, y, \dots)|^2 \,dy x_i
    e^{-x_i^2/2}\,dx_i && \text{Cauchy--Schwarz} \\
    &= \int_0^\infty |\psi(\dots, x_i, \dots)|^2 e^{-x_i^2/2} dx_i
  \end{align*}
  where in the last line we integrated by parts.  We can make the same
  argument for the integral from $-\infty$ to $0$, adjusting signs as
  needed, so we have
  \begin{equation*}
    \int_\R |I_i \psi(x)|^2 e^{-x_i^2/2} dx_i \le \int_\R |\psi_i(x)|^2
    e^{-x_i^2/2} dx_i.
  \end{equation*}
  Integrating out the remaining $x_j$ with respect to $e^{-x_j^2/2}$
  shows
  \begin{equation*}
    \norm{I_i \psi}_{L^2(\mu_k)}^2 \le \norm{\psi}_{L^2(\mu_k)}^2,
  \end{equation*}
  i.e. $I_i$ is a contraction on $L^2(\mu_k)$.

  Now for $\phi \in C_c^\infty(\R^k)$, we can approximate $\partial_1
  \dots \partial_k \phi$ in $L^2(\mu_k)$ norm by polynomials $q_n$ (by
  the finite-dimensional case of Proposition \ref{polynomials-dense}).
  If we let $p_n = I_1 \dots I_k q_n$, then $p_n$ is again a
  polynomial, and $p_n \to I_1 \dots I_k \partial_1 \dots \partial_k
  \phi = \phi$ in $L^2(\mu_k)$.  Moreover,
  $\partial_i p_n = I_1 \dots I_{i-1} I_{i+1} \dots I_k q_n \to I_1
  \dots I_{i-1} I_{i+1} \dots I_k \partial_1 \dots \partial_k \phi =
  \partial_i \phi$ in $L^2(\mu_k)$ also.

  Now back to the infinite-dimensional case.  Let $F \in \mathcal{F}
  C^\infty_c(W)$, so we can write $F(x) = \phi(e_1(x), \dots, e_k(x))$
  where $e_i$ are $q$-orthonormal.  Choose polynomials $p_n \to \phi$
  in $L^2(\R^k, \mu_k)$ with $\partial_i p_n \to \partial_i \phi$ in
  $L^2(\mu_k)$ also, and set $P_n(x) = p_n(e_1(x), \dots, e_k(x))$.
  Note $P_n \in \mathcal{H}_m$ for some $m = m_n$.
  Then
  \begin{equation*}
    \int_W |F(x) - P_n(x)|^2 \mu(dx) = \int_{\R^k} |\phi(y) -
    p_n(y)|^2 \mu_k(dy) \to 0.
  \end{equation*}
  Exercise: write out $\norm{DF - DP_n}_{L^2(W;H)}$ and show that it
  goes to 0 also.
\end{proof}

\begin{lemma}\label{DJ-commute}
  For $F \in \mathbb{D}^{1,2}$, $D J_n F = J_{n-1} DF$.
\end{lemma}

\begin{proof}
  If $F = F_\alpha$ where $|\alpha| = n$, then $J_n F_\alpha =
  F_\alpha$ and $DF_\alpha \in \mathcal{H}_{n-1}(H)$ so $J_{n-1} D
  F_\alpha = D F_\alpha$, so this is trivial.  If $|\alpha| \ne n$
  then both sides are zero.

  Now for general $F \in \mathbb{D}^{1,2}$, by the previous
  proposition we can approximate $F$ in $\mathbb{D}^{1,2}$-norm by
  functions $F_m$ which are finite linear combinations of $F_\alpha$.
  In particular, $F_m \to F$ in $L^2(W)$.  Since $J_n$ is continuous
  on $L^2(W)$, $J_n F_m \to J_n F$ in $L^2(W)$.  Also, $D F_m \to DF$
  in $L^2(W;H)$, so $J_{n-1} D F_m \to J_{n-1} DF$.  But $J_{n-1} D
  F_m = D J_n F_m$.

  We have shown $J_n F_m \to J_n F$ and $D J_n F_m \to J_{n-1} DF$.
  By closedness of $D$, we have $D J_n F = J_{n-1} DF$.
\end{proof}

\begin{corollary}
  $J_n$ is a continuous operator on $\mathbb{D}^{1,2}$.
\end{corollary}

\begin{proof}
  $F_m \to F$ in $\mathbb{D}^{1,2}$ means $F_m \to F$ in $L^2(W)$ and
  $D F_m \to DF$ in $L^2(W;H)$.  When this happens, we have $J_n F_m
  \to J_n F$ since $J_n$ is continuous on $L^2(W)$, and $D J_n F_m =
  J_{n-1} D F_m \to J_{n-1} D F = D J_n F$.
\end{proof}

\begin{corollary}
  $J_n$ is orthogonal projection onto $\mathcal{H}_n$ with respect to
  the $\mathbb{D}^{1,2}$ inner product.
\end{corollary}

\begin{proof}
  $J_n$ is the identity on $\mathcal{H}_n$, and vanishes on any
  $\mathcal{H}_m$ for $m \ne n$.  Thus by continuity it vanishes on
  $\bigoplus_{m \ne n} \mathcal{H}_m$ which is the
  $\mathbb{D}^{1,2}$-orthogonal complement of $\mathcal{H}_n$.
\end{proof}

\begin{corollary}
  For $F \in \mathbb{D}^{1,2}$, $F = \sum_{n=0}^\infty J_n F$ where
  the sum converges in $\mathbb{D}^{1,2}$.
\end{corollary}

\begin{corollary}
  For $F \in \mathbb{D}^{1,2}$, $DF = \sum_{n=0}^\infty D J_n F =
  \sum_{n=1}^\infty J_{n-1} DF$ where the sums converge in $L^2(W;H)$.
\end{corollary}

\begin{proof}
  The first equality follows from the previous corollary, since $D :
  \mathbb{D}^{1,2} \to L^2(W;H)$ is continuous.  The second equality
  is Lemma \ref{DJ-commute}.
\end{proof}

\begin{proposition}\label{D-chaos}
  $F \in \mathbb{D}^{1,2}$ if and only if $\sum_n n \norm{J_n
  F}_{L^2(W)}^2 < \infty$, in which case $\norm{DF}_{L^2(W;H)}^2 =
  \sum_n n \norm{J_n F}_{L^2(W)}^2$.
\end{proposition}

\begin{proof}
  If $f \in \mathbb{D}^{1,2}$, we have $DF = \sum_{n=0}^\infty D J_n
  F$.  Since the terms of this sum are orthogonal in $L^2(W;H)$, we
  have
  \begin{equation*}
    \infty > \norm{DF}^2_{L^2(W;H)} = \sum_{n=0}^\infty \norm{D J_n
    F}^2_{L^2(W;H)} = \sum_{n=0}^\infty n \norm{J_n F}^2_{L^2(W;H)}.
  \end{equation*}
  Conversely, if $\sum_n n \norm{J_n F}^2 = \sum_n \norm{D J_n F}^2 < \infty$ then
  $\sum J_n F$ converges to $F$ and $\sum D J_n F$ converges,
  therefore by closedness of $D$, $F \in \mathbb{D}^{1,2}$.
\end{proof}

\begin{corollary}
  If $F \in \mathbb{D}^{1,2}$ and $DF = 0$ then $F$ is constant.
\end{corollary}

\begin{lemma}[Chain rule]\label{chain-rule}
  If $\psi \in C^\infty_c(\R)$, $F \in \mathbb{D}^{1,2}$, then $\psi(F)
  \in \mathbb{D}^{1,2}$ and
  \begin{equation*}
    D \psi(F) = \psi'(F) DF.
  \end{equation*}
\end{lemma}

\begin{proof}
  For $F \in \mathcal{F} C_c^\infty(W)$ this is just the regular chain
  rule.  For general $F \in \mathbb{D}^{1,2}$, choose $F_n \in
  C_c^\infty(W)$ with $F_n \to F$, $D F_n \to DF$.  Then use dominated
  convergence.
\end{proof}

Actually the chain rule also holds for any $\psi \in C^1$ with bounded
first derivative.  Exercise: prove.

\begin{proposition}\label{zero-one-law}
  If $A \subset W$ is Borel and $1_A \in \mathbb{D}^{1,2}$ then
  $\mu(A)$ is 0 or 1.
\end{proposition}

\begin{proof}
  Let $\psi \in C^\infty_c(\R)$ with $\psi(s) = s^2$ on $[0,1]$.  Then
  \begin{equation*}
    D 1_A = D \psi(1_A) = 2 1_A D 1_A
  \end{equation*}
  so by considering whether $x \in A$ or $x \in A^c$ we have $D 1_A =
  0$ a.e.  Then by an above lemma, $1_A$ is (a.e.) equal to a constant.
\end{proof}

As a closed densely defined operator between Hilbert spaces $L^2(W)$
and $L^2(W;H)$, $D$ has an adjoint operator $\delta$, which is a
closed densely defined operator from $L^2(W;H)$ to $L^2(W)$.

To get an idea what $\delta$ does, let's start by evaluating it on
some simple functions.

\begin{proposition}\label{delta-simple}
  If $u(x) = G(x) h$ where $G \in \mathbb{D}^{1,2}$ and $h \in H$,
  then $u \in \dom(\delta)$ and 
  \begin{equation*}
    \delta u(x) = G(x) \inner{h}{x}_H - \inner{DG(x)}{h}_H.
  \end{equation*}
\end{proposition}

\begin{proof}
  If $G \in \mathcal{F} C_c^\infty$, use (\ref{product-byparts}).
  Otherwise, approximate.  (Hmm, maybe we actually need $G \in
  L^{2+\epsilon}(W)$ for this to work completely.)
\end{proof}

Recall in the special case of Brownian motion, where $W = C([0,1])$
and $H = H^1_0([0,1])$, we had found that $\inner{h}{\omega}_H =
\int_0^1 \dot{h}(s) dB_s(\omega)$, i.e. $\inner{h}{\cdot}_H$
produces the Wiener integral, a special case of the It\^o integral
with a deterministic integrand.  So it appears that $\delta$ is also
some sort of integral.  We call it the Skorohod integral.  In fact,
the It\^o integral is a special case of it!
    
\begin{theorem}
  Let $A(t, \omega)$ be an adapted process in $L^2([0,1] \times
  W)$.  Set $u(t, \omega) = \int_0^t A(\tau,\omega)\,d\tau$, so $u(\cdot,
  \omega) \in H$ for each $\omega$.   Then $u \in \dom(\delta)$ and
  $\delta u = \int_0^t A_\tau dB_\tau$.
\end{theorem}
    
\begin{proof}
  First suppose $A$ is of the form $A(t, \omega) = 1_{(r,s]}(t)
  F(\omega)$ where $F(\omega) = \phi(B_{t_1}(\omega), \dots,
  B_{t_n}(\omega))$ for some $0 \le t_1,
  \dots, t_n \le r$ and $\phi \in C_c^\infty(\R^n)$.  (Recall
  $B_t(\omega) = \omega(t)$ is just the evaluation map, a continuous
  linear functional of $\omega$.)   We have $F \in \mathcal{F}_r$ so
  $A$ is adapted.  Then $u(t, \omega) = h(t)
  F(\omega)$ where $h(t) = \int_0^t 1_{(r,s)}(\tau) d\tau$.
  In particular $h(t) = 0$ for $t \le r$, so
  \begin{equation*}
    \inner{DF(\omega)}{h}_H = \sum_i \partial_i \phi(B_{t_1}(\omega),
    \dots, B_{t_n}(\omega)) h(t_i) = 0.
  \end{equation*}
  Thus
  \begin{align*}
    \delta u(\omega) &= F(\omega) \inner{h}{\omega}_H -
    \cancel{\inner{DF(\omega)}{h}_H} \\
    &= F(\omega) (B_r(\omega)-B_s(\omega)) \\
    &= \int_0^1 A_\tau \,dB_\tau (\omega).
  \end{align*}
  Now we just need to do some approximation.  If $F \in L^2(W,
  \mathcal{F}_r)$, we can approximate $F$ in $L^2$ by cylinder
  functions $F_n$ of the above form.  Then defining $u_n$, $u$
  accordingly we have $u_n \to u$ in $L^2(W;H)$ and $\delta u_n \to F
  \cdot (B_r - B_s) = \int_0^1 A_\tau \,dB_\tau$ in $L^2(W)$, so the
  conclusion holds for $A = F(\omega) 1_{(r,s]}(t)$.  By linearity it
  also holds for any linear combination of such processes.  The set of
  such linear combinations is dense in the adapted processes in
  $L^2([0,1] \times W)$, so choose such $A^{(n)} \to A$.  Note that
  $A(t,\omega) \mapsto \int_0^t A(\tau, \omega)\,d\tau$ is an isometry
  of $L^2([0,1] \times W)$ into $L^2(W; H)$ so the corresponding $u_n$
  converge to $u$ in $L^2(W;H)$, and by the It\^o isometry, $\delta
  u_n = \int_0^1 A^{(n)}_\tau dB_\tau \to \int_0^1 A_\tau dB_\tau$.
  Since $\delta$ is closed we are done.
\end{proof}

This is neat because defining an integral in terms of $\delta$ lets us
integrate a lot more processes.  

\begin{example}
  Let's compute $\int_0^1 B_1 dB_s$ (Skorohod).  We have $u(t,\omega)
  = t B_1 = h(t) G(\omega)$ where $h(t)=t$, and $G(\omega) =
  \phi(f(\omega))$ where $\phi(x) = x$, and
  $f(\omega) = \omega(1)$.  So $\delta u (\omega) = G(\omega)
  \inner{h}{\omega}_H - \inner{DG(\omega)}{h}_H$.  But
  $\inner{h}{\omega} = \int_0^1 \dot{h}(t) dB_t = B_1$.  And
  $DG(\omega) = Jf$ so $\inner{DG(\omega)}{h}_H = f(h) = h(1) = 1$.
  So $\int_0^1 B_1 dB_s = B_1^2 - 1$.

  Note that $B_1$, although it doesn't depend on $t$, is not adapted.
  Indeed, a random variable is an adapted process iff it is in
  $\mathcal{F}_0$, which means it has to be constant.
\end{example}

All the time derivatives and integrals here are sort of a red herring;
they just come from the fact that the Cameron--Martin inner product
has a time derivative in it.

Another cool fact is that we can use this machinery to construct
integration with respect to other continuous Gaussian processes.  Again let $W =
C([0,1])$ (or an appropriate subspace), $\mu$ the law of a continuous
centered 
Gaussian process $X_t$ with covariance function $a(s,t) = E[X_s X_t]$.
Since $\delta$ applies to elements of $L^2(W; H)$, and we want to
integrate honest processes (such as elements of $L^2([0,1] \times
W)$), we need some way to map processes to elements of $L^2(W;H)$.  For Brownian
motion it was $A_t \mapsto \int_0^t A_s ds$, an isometry of $L^2([0,1]
\times W) = L^2(W; L^2([0,1])) = L^2(W) \otimes L^2([0,1])$ into $L^2(W; H)$.  To construct such an map $\Phi$ in this
case, we start with the idea that we want $\int_0^t dX_s = X_t$, so we
should have $\Phi 1_{[0,T]} = J \delta_s \in H$.  We can extend this map
linearly to $\mathcal{E}$, the set of all step functions on $[0,1]$.
To make it an isometry, equip $\mathcal{E}$ with the inner product
defined by
\begin{equation*}
\inner{1_{[0,s]}}{1_{[0,t]}}_{\mathcal{E}} = \inner{J \delta_s}{J
  \delta_t}_H = a(s,t)
\end{equation*}
again extended by bilinearity.  It extends isometrically to the completion of
$\mathcal{E}$ under this inner product, whatever that may be.  So the
processes we can integrate are 
we can integrate processes from $L^2(W; \bar{\mathcal{E}}) = L^2(W)
\otimes \bar{\mathcal{E}}$ just by taking $\int_0^1 A\,dX = \delta
u$ where $u(\omega) = \Phi(A(\omega))$.

Exactly what are the elements of $L^2(W; \bar{\mathcal{E}})$ is a
little hard to say.  The elements of $\bar{\mathcal{E}}$ can't
necessarily be identified as functions on $[0,1]$; they might be
distributions, for instance.  But we for sure know it contains step
functions, so $L^2(W; \bar{\mathcal{E}})$ at least contains
``simple processes'' of the form $\sum Y_i 1_{[a_i, b_i]}(t)$.  In the
case of fractional Brownian motion, one can show that
$\bar{\mathcal{E}}$ contains $L^2([0,1])$, so in particular $\Phi$
makes sense for any process in $L^2(W \times [0,1])$.
Of course, there is still the question of whether $\Phi(A) \in \dom
\delta$.  

There's a chapter in Nualart which works out a lot of this in the
context of fractional Brownian motion.  Being able to integrate with
respect to fBM is a big deal, because fBM is not a semimartingale and
so it is not covered by any version of It\^o integration.

A couple of properties of $\delta$ in terms of the Wiener chaos:

\begin{enumerate}
  \item $\mathcal{H}_n(H) \subset \dom \delta$ for each $n$.  (Follows
  from Lemma \ref{delta-simple}.)

  \item For $u \in \dom \delta$, $J_n \delta u = \delta J_{n-1} u$.

    Using Lemma \ref{DJ-commute} and the fact that the $J_n$, being
    orthogonal projections, are self-adjoint, we have for any $F \in
    \mathbb{D}^{1,2}$,
    \begin{align*}
      \inner{J_n \delta u}{F}_{L^2(W)} &= \inner{\delta u}{J_n
      F}_{L^2(W)} \\ &= \inner{u}{D J_n F}_{L^2(W;H)} \\ &= \inner{u}{J_{n-1}
      D F}_{L^2(W;H)} \\ &= \inner{\delta J_{n-1} u}{F}_{L^2(W)}.
    \end{align*}
    $\mathbb{D}^{1,2}$ is dense in $L^2(W)$ so we are done.

  \item $J_0 \delta u = 0$.  For if $F \in L^2(W)$, then
  $\inner{J_0 \delta u}{F} = \inner{\delta u}{J_0 F} = \inner{u}{D J_0
  F}$.  But $J_0 F$ is a constant so $D J_0 F = 0$.
\end{enumerate}

\subsection{The Clark--Ocone formula}

Until further notice, we are working on classical Wiener space, $W =
C_0([0,1])$, with $\mu$ being Wiener measure.

A standard result in stochastic calculus is the \textbf{It\^o
  representation theorem}, which in its classical form says:

\begin{theorem}
  Let $\{B_t\}$ be a Brownian motion on $\R^d$, let
  $\{\mathcal{F}_t\}$ be the filtration it generates, and let $Z$ be an
  $L^2$ random variable which is $\mathcal{F}_1$-measurable (sometimes called a \textbf{Brownian functional}).
  Then there exists an adapted $L^2$ process $Y_t$ such that
  \begin{equation*}
    Z = \mathbb{E}[Z] + \int_0^1 Y_t \,dB_t, \quad \text{a.s.}
  \end{equation*}
\end{theorem}

\begin{proof}[Sketch]
  This is claiming that the range of the It\^o integral contains all
  the $L^2$ random variables with mean zero (which we'll denote
  $\mathcal{H}_0^\perp$).  Since the It\^o integral is an isometry,
  its range is automatically closed, so it suffices to show it is
  dense in $\mathcal{H}_0^\perp$.  One can explicitly produce a dense set.
\end{proof}

\begin{exercise}
  Look up a proof.
\end{exercise}

An important application of this theorem is in finance.  Suppose we
have a stochastic process $\{X_t\}$ which gives the price of a stock
(call it Acme) at time $t$.  (Temporarily you can think $X_t = B_t$ is
Brownian motion, though this is not a good model and we might improve
it later.)  We may want to study an \textbf{option} or
\textbf{contingent claim}, some contract whose ultimate value $Z$ is
determined by the behavior of the stock.  For example:

\begin{itemize}
  \item A \textbf{European call option} is a contract which gives you
  the right, but not the obligation, to buy one share of Acme at time
  1 for a pre-agreed \textbf{strike price} $K$.  So if the price $X_1$
  at time 1 is greater than $K$, you will \textbf{exercise} your
  option, buy a share for $K$ dollars, and then you can immediately
  sell it for $X_1$ dollars, turning a quick profit of $Z = X_1 - K$
  dollars.  If $X_1 < K$, then you should not exercise the option; it
  is  worthless, and $Z=0$.  Thus we can write $Z = (X_1 - K)^+$.

  \item A \textbf{European put option} gives the right to sell one
  share of Acme at a price $K$.  Similarly we have $Z = (K - X_1)^+$.

  \item A \textbf{floating lookback put option} gives one the right,
  at time 1, to sell one share of Acme at the highest price it ever
  attained between times 0 and 1.  So $Z = \sup_{t \in [0,1]} X_t - X_1$. 

  \item There are many more.
\end{itemize}

You can't lose money with these contracts, because you can always just
not exercise it, and you could gain a profit.  Conversely, your
counterparty can only lose money.  So you are going to have to pay
your counterparty some money up front to get them to enter into such a
contract.  How much should you pay?  A ``fair'' price would be
$\mathbb{E}[Z]$.  But it may be that the contract would be worth more or less to you
than that, depending on your appetite for risk.  (Say more about
this.)

Here  the It\^o representation theorem comes to the rescue.  If $X_t =
B_t$ is a Brownian motion, it says that $Z = E[Z] + \int_0^1
Y_t\,dB_t$.  This represents a \textbf{hedging strategy}.  Consider a
trading strategy where at time $t$ we want to own $Y_t$ shares of Acme
(where we can hold or borrow cash as needed to achieve this; negative
shares are also okay because we can sell short).  $Y_t$ is adapted,
meaning the number of shares to own can be determined by what the
stock has already done.  A moment's
thought shows that the net value of your portfolio at time $1$ is
$\int_0^1 Y_t \,dB_t$.  Thus, if we start with $\mathbb{E}[Z]$ dollars in the
bank and then follow the strategy $Y_t$, at the end we will have
exactly $Z$ dollars, almost surely.  We can \textbf{replicate} the
option $Z$ for $\mathbb{E}[Z]$ dollars (not counting transaction costs, which
we assume to be negligible).  So anybody that wants more than
$\mathbb{E}[Z]$ dollars is ripping us off, and we shouldn't pay it even if
we would be willing to.  

So a key question is whether we can explicitly find $Y_t$.  

In Wiener space notation, $Z$ is an element of $L^2(W,\mu)$, which we
had usually called $F$.  Also, now $\mathcal{F}_t$ is the
$\sigma$-algebra generated by the linear functionals $\{\delta_s : s
\le t\}$; since these span a weak-* dense subset of $W^*$ we have
$\mathcal{F}_1 = \sigma(W^*) = \mathcal{B}_W$, the Borel
$\sigma$-algebra of $W$.  

Let $L^2_a([0,1] \times W)$ be the space of adapted processes.

\begin{exercise}
  $L^2_a([0,1] \times W)$ is a closed subspace of $L^2([0,1] \times W)$.
\end{exercise}

\begin{exercise}
  $Y_t \mapsto \mathbb{E}[Y_t | \mathcal{F}_t]$ is orthogonal projection from
  $L^2([0,1] \times W)$ onto $L^2_a([0,1] \times W)$.
\end{exercise}

This section is tangled up a bit by some derivatives coming and
going.  Remember that $H$ is naturally isomorphic to $L^2([0,1])$
via the map $\Phi : L^2([0,1]) \to H$ given by $\Phi f(t) = \int_0^t
f(s) ds$ (its inverse is simply $\frac{d}{dt}$).  Thus $L^2(W;H)$ is
naturally isomorphic to $L^2([0,1] \times W)$.  Under this
identification, we can identify $D : \mathbb{D}^{1,2} \to L^2(W;H)$ with a map
that takes an element $F \in \mathbb{D}^{1,2}$ to a process $D_t F \in
L^2([0,1] \times W)$; namely, $D_t F(\omega) = \frac{d}{dt}
DF(\omega)(t) = \frac{d}{dt} \inner{DF(\omega)}{J \delta_t}_H =
\frac{d}{dt} \inner{DF(\omega)}{\cdot \wedge t}_H$.  So $D_t F =
\Phi^{-1} D F$.

The Clark--Ocone theorem states:

\begin{theorem}
  For $F \in \mathbb{D}^{1,2}$,
  \begin{equation}
    F = \int F\,d\mu + \int_0^1 \mathbb{E}[D_t F | \mathcal{F}_t]\,dB_t.
  \end{equation}
\end{theorem}

To prove this, we want to reduce everything to Skorohod integrals.
Let $E \subset L^2(W;H)$ be the image of $L^2_a([0,1] \times W)$ under
the isomorphism $\Phi$.  Then, since the Skorohod integral extends the
It\^o integral, we know that $E \subset \dom \delta$, and $\delta : E
\to L^2(W)$ is an isometry.  Moreover, by the It\^o representation
theorem, the image $\delta(E)$ is exactly $\mathcal{H}_0^\perp$,
i.e. the orthogonal complement of the constants, i.e. functions with
zero mean.

Let $P$ denote orthogonal projection onto
$E$, so that $\mathbb{E}[\cdot | \mathcal{F}_t] = \Phi^{-1} P \Phi$.

We summarize this discussion by saying that the following diagram
commutes.
\begin{equation}
\xymatrix{
  & L^2([0,1] \times W) \ar@{->}[rr]^{\mathbb{E}[\,\cdot\, |
  \mathcal{F}_t]} 
  \ar@{<=>}[dd]_{\Phi}
  && L^2_a([0,1] \times W) \ar@{<=>}[dd]_{\Phi} \ar@{->}[dr]^{\int_0^1
  \cdot \,dB_t}
\\
  \mathbb{D}^{1,2} \ar@{->}[ur]^{D_t} \ar@{->}[dr]_{D} & & & & L^2(W) \\
  & L^2(W;H) \ar@{->}[rr]^{P} 
  & & E \ar@{->}[ur]_{\delta}
}
\end{equation}


From this diagram, we see that the Clark--Ocone theorem reads:
\begin{equation}
  F = \int F d\mu + \delta P D F.
\end{equation}

Now the proof is basically just a diagram chase.

\begin{proof}
  Suppose without loss of generality that $\int F \,d\mu = 0$, so that
  $F \in \mathcal{H}_0^\perp$.  Let $u \in E$.  Then
  \begin{align*}
    \inner{F}{\delta u}_{L^2(W)} &= \inner{DF}{u}_{L^2(W;H)} \\
    &= \inner{DF}{Pu}_{L^2(W;H)} && \text{since $u \in E$} \\
    &= \inner{PDF}{u}_{L^2(W;H)} 
    \intertext{(since orthogonal projections are self-adjoint)} 
    &= \inner{\delta P D F}{\delta u}_{L^2(W)}
  \end{align*}
  since $PDF \in E$, $u \in E$, and $\delta$ is an isometry on $E$.
  As $u$ ranges over $E$, $\delta u$ ranges over
  $\mathcal{H}_0^\perp$, so we must have $F = \delta P D F$.
\end{proof}

\begin{exercise}
  If the stock price is Brownian motion ($X_t = B_t$), compute the hedging strategy
  $Y_t$ for a European call option $Z = (X_1 - K)^+$.
\end{exercise}

\begin{exercise}
  Again take $X_t = B_t$.  Compute the hedging strategy for a floating lookback call option $Z
  = M - X_1$, where $M = \sup_{t \in [0,1]} X_t$.  (Show that $D_t M =
  1_{\{t \le T\}}$ where $T = \arg\max X_t$, which is a.s. unique, by
  approximating $M$ by the maximum over a finite set.)
\end{exercise}

\begin{exercise}
  Let $X_t$ be a \textbf{geometric Brownian motion} $X_t =
  \exp\left(B_t - \frac{t}{2}\right)$.  Compute the hedging strategy
  for a European call option $Z = (X_1 - K)^+$.  (Note by It\^o's
  formula that $dX_t = X_t dB_t$.)
\end{exercise}

\section{Ornstein--Uhlenbeck process}

We've constructed one canonical process on $W$, namely Brownian motion
$B_t$, defined by having independent increments distributed according
to $\mu$ (appropriately scaled).  In finite dimensions, another
canonical process related to Gaussian measure is the
\textbf{Ornstein--Uhlenbeck process}.  This is a Gaussian process
$X_t$ which can be defined by the SDE $dX_t = \sqrt{2} dB_t - X_t dt$.
Intuitively, $X_t$ tries to move like a Brownian motion, but it
experiences a ``restoring force'' that always pulls it back toward the
origin.  Imagine a Brownian particle on a spring.  A key relationship
between $X_t$ and standard Gaussian measure $\mu$ is that $X_t$ has
$\mu$ as its stationary distribution: if we start $X_t$ in a random
position chosen according to $\mu$, then $X_t$ itself is also
distributed according to $\mu$ at all later times.  This also means
that, from any starting distribution, the distribution of $X_t$
converges to $\mu$ as $t \to \infty$.

One way to get a handle on the Ornstein--Uhlenbeck process, in finite
or infinite dimensions, is via its \textbf{Dirichlet form}.  Here are
some basics on the subject.

\subsection{Crash course on Dirichlet forms}

Suppose $X_t$ is a symmetric Markov process on some topological space
$X$ equipped with a Borel measure $m$.  This means that its transition
semigroup $T_t f(x) = E_x[f(X_t)]$ is a Hermitian operator on
$L^2(X,m)$.  If we add a few extra mild conditions (e.g. c\'adl\'ag,
strong Markov) and make $X_t$ a \textbf{Hunt process}, the semigroup
$T_t$ will be strongly continuous.  It is also \textbf{Markovian},
i.e. if $0 \le f \le 1$, then $0 \le T_t f \le 1$.  For example, if
$X = \R^n$, $m$ is Lebesgue measure, and $X_t$ is Brownian motion,
then $T_t f(x) = \frac{1}{(2 \pi t)^{n/2}} \int_{\R^n} f(y)
e^{-|x-y|^2/2t}\,m(dy)$ is the usual heat semigroup.

A strongly continuous contraction semigroup has an associated
\textbf{generator}, a nonnegative-definite self-adjoint operator $(L, D(L))$
which in general is unbounded, such that $T_t = e^{-tL}$.  For
Brownian motion it is $L = -\Delta/2$ with $D(L) = H^2(\R^n)$.

Associated to a nonnegative self-adjoint operator is an unbounded
bilinear symmetric form $\mathcal{E}$ with domain $\mathbb{D}$, such
that $\mathcal{E}(f,g) = (f,Lg)$ for every $f \in \mathbb{D}$ and $g
\in D(L)$.  We can take $(\mathcal{E},\mathbb{D})$ to be a closed
form, which essentially says that $\mathcal{E}_1(f,g) =
\mathcal{E}(f,g) + (f,g)$ is a Hilbert inner product on $\mathbb{D}$.
Note that $\mathbb{D}$ is generally larger  than $D(L)$.
For Brownian motion, $\mathcal{E}(f,g) = \int_{\R^n} \grad f \cdot
\grad g\,dm$ and $\mathbb{D} = H^1(\R^n)$.  Note $\mathcal{E}_1$ is
the usual Sobolev inner product on $H^1(\R^n)$.  $\mathcal{E}(f,f)$
can be interpreted as the amount of ``energy'' contained in the
distribution $f dm$.  Letting this distribution evolve under the
process will tend to reduce the amount of energy as quickly as possible.

When $T_t$ is Markovian, $(\mathcal{E},\mathbb{D})$ has a
corresponding property, also called Markovian.  Namely, if $f \in
\mathbb{D}$, let $\bar{f} = f \wedge 1 \vee 0$ be a ``truncated''
version of $f$.  The Markovian property asserts that $\bar{f} \in
\mathbb{D}$ and $\mathcal{E}(\bar{f}, \bar{f}) \le \mathcal{E}(f,f)$.
A bilinear, symmetric, closed, Markovian form on $L^2(X,m)$ is called
a \textbf{Dirichlet form}.

So far this is nice but not terribly interesting.  What's neat is that
this game can be played backwards.  Under certain conditions, one can
start with a Dirichlet form and recover a Hunt process with which it is
associated.  This is great, because constructing a process is usually
a lot of work, but one can often just write down a Dirichlet form.
Moreover, one finds that properties of the process often have
corresponding properties for the Dirichlet form.

For example, if the process $X_t$ has continuous sample paths, the
form $(\mathcal{E},\mathbb{D})$ will be \textbf{local}: namely, if
$f=0$ on the support of $g$, then
$\mathcal{E}(f,g) = 0$.  Conversely, if the form is local, then the
associated process will have continuous sample paths.  If additionally
the process is not killed inside $X$, the form is \textbf{strongly
  local}: if $f$ is constant on the support of $g$, then
$\mathcal{E}(f,g) = 0$; and the converse is also true.

So you might ask: under what conditions must a Dirichlet form be
associated with a process?  One sufficient condition is that
$(\mathcal{E},\mathbb{D})$ be \textbf{regular}: that $\mathbb{D} \cap
C_c(X)$ is $\mathcal{E}_1$-dense in $\mathbb{D}$ and uniformly dense
in $C_c(X)$.  We also have to assume that $X$, as a topological space,
is locally compact.  The main purpose of this condition is to exclude the
possibility that $X$ contains ``holes'' that the process would have to
pass through.  Unfortunately, this condition is useless in infinite
dimensions, since if $X = W$ is, say, an infinite-dimensional Banach
space, then $C_c(W) = 0$.

There is a more general condition called \textbf{quasi-regular}, which
is actually necessary and sufficient for the existence of a process.
It is sufficiently complicated that I won't describe it here; see Ma
and R\"ockner's book for the complete treatment.

\subsection{The Ornstein--Uhlenbeck Dirichlet form}

We are going to define the Ornstein--Uhlenbeck process via its
Dirichlet form.  For $F,G \in \mathbb{D}^{1,2}$, let $\mathcal{E}(F,G)
= \inner{DF}{DG}_{L^2(W;H)}$.  This form is obviously bilinear,
symmetric, and positive semidefinite.  With the domain
$\mathbb{D}^{1,2}$, $\mathcal{E}$ is also a closed form (in fact,
$\mathcal{E}_1$ is exactly the Sobolev inner product on
$\mathbb{D}^{1,2}$, which we know is complete).  

\begin{proposition}
  $(\mathcal{E}, \mathbb{D}^{1,2})$ is Markovian.
\end{proposition}

\begin{proof}
  Fix $\epsilon > 0$.  Let $\varphi_n \in C^\infty(\R)$ be a sequence
  of smooth functions with $0 \le \varphi_n \le 1$, $|\varphi_n'| \le
  1+\epsilon$, and $\varphi_n(x) \to x \wedge 1 \vee 0$ pointwise.
  (Draw a picture to convince yourself this is possible.)  Then
  $\varphi_n(F) \to F \wedge 1 \vee 0$ in $L^2(W)$ by dominated
  convergence.  Then, by the chain rule, for $F \in \mathbb{D}^{1,2}$,
  we have $\norm{D \varphi_n(F)}_{L^2(W;H)} = \norm{\varphi_n'(F)
  DF}_{L^2(W;H)} \le (1+\epsilon) \norm{DF}_{L^2(W;H)}$.  It follows
  from Alaoglu's theorem that $F \wedge 1 \vee 0 \in
  \mathbb{D}^{1,2}$, and moreover, $\norm{D[F \wedge 1 \vee
  0]}_{L^2(W;H)} \le (1+\epsilon) \norm{DF}_{L^2(W;H)}$.  Letting
  $\epsilon \to 0$ we are done.
\end{proof}

\begin{exercise}
  Fill in the details in the preceding proof.
\end{exercise}

\begin{theorem}
  $(\mathcal{E}, \mathbb{D}^{1,2})$ is quasi-regular.  Therefore,
  there exists a Hunt process $X_t$ whose transition semigroup is
  $T_t$, the semigroup corresponding to $(\mathcal{E}, \mathbb{D}^{1,2})$.
\end{theorem}

\begin{proof}
  See \cite[IV.4.b]{ma-rockner-book}.
\end{proof}

\begin{lemma}
  The operator $D$ is \textbf{local} in the sense that for any $F \in
  \mathbb{D}^{1,2}$, $DF = 0$
  $\mu$-a.e. on $\{F = 0 \}$.
\end{lemma}

\begin{proof}
  Let $\varphi_n \in C^\infty_c(\R)$ have $\varphi_n(0) = 1$, $0 \le
  \varphi_n \le 1$, and $\varphi_n$ supported inside
  $\left[-\frac{1}{n}, \frac{1}{n}\right]$; note that $\varphi_n \to
  1_{\{0\}}$ pointwise and boundedly.  Then as $n \to \infty$,
  $\varphi_n(F) DF \to 1_{\{F = 0\}} DF$ in $L^2(W;H)$.  Let $\psi_n(t) =
  \int_{-\infty}^t \varphi_n(s)\,ds$, so that $\varphi_n = \psi_n'$;
  then $\psi_n \to 0$ uniformly.  By the chain rule we have
  $D(\psi_n(F)) = \varphi_n(F) DF$.  Now if we fix $u \in \dom
  \delta$, we have
  \begin{align*}
    \inner{1_{\{F=0\}} DF}{u}_{L^2(W;H)} &= \lim_{n \to \infty}
    \inner{\varphi_n(F) DF}{u}_{L^2(W;H)}\\
    &= \lim_{n \to \infty} \inner{D (\psi_n(F))}{u}_{L^2(W;H)} \\
    &= \lim_{n \to \infty} \inner{\psi_n(F)}{\delta u}_{L^2(W)} = 0
  \end{align*}
  since $\psi_n(F) \to 0$ uniformly and hence in $L^2(W)$.  Since
  $\dom \delta$ is dense in $L^2(W;H)$, we have $1_{\{F=0\}} DF = 0$
  $\mu$-a.e., which is the desired statement.
\end{proof}

\begin{corollary}
  The Ornstein--Uhlenbeck Dirichlet form $(\mathcal{E},
  \mathbb{D}^{1,2})$ is strongly local.
\end{corollary}

\begin{proof}
  Let $F,G \in \mathbb{D}^{1,2}$.  Suppose first that $F = 0$ on the support
  of $G$.  By the previous lemma we have (up to $\mu$-null sets) $\{DF
  = 0\} \supset \{F = 0\} \supset \{G \ne 0\} \supset \{DG \ne 0\}$.
  Thus, for a.e. $x$ either $DF(x) = 0$ or $DG(x) = 0$.  So
  $\mathcal{E}(F,G) = \int_X \inner{DF(x)}{DG(x)}_H \,\mu(dx) = 0$.

  If $F = 1$ on the support of $G$, write $\mathcal{E}(F,G) =
  \mathcal{E}(F-1,G) + \mathcal{E}(1,G)$.  The first term vanishes by
  the previous step, while the second term vanishes since $D1 = 0$.
\end{proof}

We now want to investigate the generator $N$ associated to
$(\mathcal{E},\mathbb{D})$.

\begin{lemma}
  For $F \in L^2(W)$, $J_0 F = \int F d\mu$, where $J_0$ is the
  orthogonal projection onto $\mathcal{H}_0$, the constant functions
  in $L^2(W)$.
\end{lemma}

\begin{proof}
  This holds over any probability space.  Write $EF = \int F d\mu$.
  Clearly $E$ is continuous, $E$ is the identity on the constants
  $\mathcal{H}_0$, and if $F \perp \mathcal{H}_0$, then we have $EF =
  \inner{F}{1}_{L^2(W)} = 0$ since $1 \in \mathcal{H}_0$.  So $E$ must
  be orthogonal projection onto $\mathcal{H}_0$.
\end{proof}

\begin{lemma}\label{poincare}[A Poincar\'e inequality]
  For $F \in \mathbb{D}^{1,2}$, we have
  \begin{equation*}
    \norm{F - \int F d\mu}_{L^2(W)} \le \norm{DF}_{L^2(W;H)}.
  \end{equation*}
\end{lemma}

\begin{proof}
  Set $G = F - \int F d\mu$, so that $J_0 G = \int G d\mu = 0$.  Note that $DF
  = DG$ since $D1 = 0$.  Then by Lemma \ref{D-chaos},
  \begin{equation*}
    \norm{DG}_{L^2(W;H)}^2 = \sum_{n=0}^\infty n \norm{J_n
    G}_{L^2(W)}^2 \ge \sum_{n=1}^\infty \norm{J_n
    G}_{L^2(W)}^2 = \sum_{n=0}^\infty \norm{J_n
    G}_{L^2(W)}^2 = \norm{G}_{L^2(W)}^2. 
  \end{equation*}
\end{proof}

Note that by taking $F(x) = f(x)$ for $f \in W^*$, we can see that the
Poincar\'e inequality is sharp.

\begin{theorem}
  $N = \delta D$.  More precisely, if we set
  \begin{equation*}
    \dom N = \dom \delta D = \{ F \in \mathbb{D}^{1,2}  : DF \in \dom \delta \}
  \end{equation*}
  and $NF = \delta D F$ for $F \in \dom N$, then $(N, \dom N)$ is the
  unique self-adjoint operator satisfying $\dom N \subset
  \mathbb{D}^{1,2}$ and 
  \begin{equation}\label{Ndef} \mathcal{E}(F,G) = \inner{F}{NG}_{L^2(W)}
  \text{ for all } F \in \mathbb{D}^{1,2}, G \in \dom N.
  \end{equation}
\end{theorem}

\begin{proof}
  It is clear that $\dom N \subset \mathbb{D}^{1,2}$ and that
  (\ref{Ndef}) holds.  Moreover, it is known there is a unique
  self-adjoint operator with this property (reference?).  We have to
  check that $N$ as defined above is in fact self-adjoint.
  (Should fill this in?)
\end{proof}

\begin{proposition}
  $N F_\alpha = |\alpha| F_\alpha$.  That is, the Hermite polynomials
  $F_\alpha$ are eigenfunctions for $N$, with eigenvalues $|\alpha|$.
  So the $\mathcal{H}_n$ are eigenspaces.
\end{proposition}

\begin{proof}
  Since $F_\alpha$ is a cylinder function, it is easy to see it is in
  the domain of $N$.  Then $\inner{N F_\alpha}{F_\beta}_{L^2(W)} =
  \inner{D F_\alpha}{D F_\beta}_{L^2(W;H)} = |\alpha| \delta_{\alpha
  \beta}$.  Since the $\{F_\beta\}$ are an orthonormal basis for
  $L^2(W)$, we are done.
\end{proof}

There is a natural identification of $\mathcal{H}_n$ with $H^{\otimes
  n}$, which gives an identification of $L^2(W)$ with Fock space
  $\bigoplus_n H^{\otimes n}$.  In quantum mechanics this is the state
  space for a system with an arbitrary number of particles,
  $H^{\otimes n}$ corresponding to those states with exactly $n$
  particles.  $N$ is thus called the number operator because $\inner{N
  F}{F}$ gives the (expected) number of particles in the state $F$.

\begin{proposition}
  $NF = \sum_{n=0}^\infty n J_n F$, where the sum on the right
  converges iff $F \in \dom N$.
\end{proposition}

\begin{proof}
  For each $m$, we have
  \begin{equation*}
    N \sum_{n=0}^m J_n F = \sum_{n=0}^m N J_n F  =
    \sum_{n=0}^m n J_n F.
  \end{equation*}
  Since $\sum_{n=0}^m J_n F \to F$ as $m \to \infty$ and $N$ is
  closed, if the right side converges then $F \in \dom N$ and $NF$
  equals the limit of the right side.

  Conversely, if $F \in \dom N$, we have $\infty >
  \norm{NF}^2_{L^2(W)} = \sum_{n=0}^\infty \norm{J_n N F}^2$.  But,
  repeatedly using the self-adjointness of $J_n$ and $N$ and the
  relationships $J_n = J_n^2$ and $N J_n = n J_n$,
  \begin{align*}
    \norm{J_n N F}^2 = \inner{F}{N J_n N F} = n \inner{F}{J_n N F} = n
    \inner{N J_n F}{F} = n^2 \inner{J_n F}{F} = n^2 \norm{J_n F}^2.
  \end{align*}
  Thus $\sum n^2 \norm{J_n F}^2 < \infty$, so $\sum n J_n F$
  converges.
\end{proof}

Let $T_t = e^{-tN}$ be the semigroup generated by $N$.  Note that each
$T_t$ is a contraction on $L^2(W)$, and $T_t$ is strongly continuous
in $t$.

\begin{proposition}\label{Tt-chaos}
  For any $F \in L^2(W)$, 
  \begin{equation} \label{Tt-chaos-eqn}
    T_t F = \sum_{n=0}^\infty e^{-tn} J_n F.
  \end{equation}
\end{proposition}

\begin{proof}
  Since $J_n F$ is an eigenfunction of $N$, we must have
  \begin{equation*}
    \frac{d}{dt} T_t J_n F = T_t N J_n F = n T_t J_n f.
  \end{equation*}
  Since $T_0 J_n F = J_n F$, the only solution of this ODE is $T_t J_n
  F = e^{-tn} J_n F$.  Now sum over $n$.
\end{proof}

\begin{corollary}\label{Tt-exponential-decay}
  $\norm{T_t F - \int F d\mu}_{L^2(W)} \le e^{-t} \norm{F - \int F d\mu}$.
\end{corollary}

\begin{proof}
  Let $G = F - \int F d\mu$; in particular $J_0 G = 0$.  Then
  \begin{align*}
    \norm{T_t G}^2 = \sum_{n=1}^\infty e^{-2tn} \norm{J_n G}^2 \le
    e^{-2t} \sum_{n=1}^\infty \norm{J_n G}^2 = e^{-2t} \norm{G}^2. 
  \end{align*}
\end{proof}

This is also a consequence of the Poincar\'e inequality (Lemma
\ref{poincare}) via the spectral theorem.

$T_t$ is the transition semigroup of the Ornstein--Uhlenbeck process
$X_t$, i.e. $T_t F(x) = \mathbb{E}_x[F(X_t)]$ for $\mu$-a.e. $x \in
X$.  To get a better understanding of this process, we'll study $T_t$
and $N$ some more.

The finite-dimensional Ornstein--Uhlenbeck operator is given by
\begin{equation*}
  \tilde{N} \phi (x) = \Delta \phi (x) - x \cdot \grad \phi(x). 
\end{equation*}

The same formula essentially works in infinite dimensions.

\begin{lemma}
  For $F \in \mathcal{F} C^\infty_c(W)$ of the form $F(x) =
  \phi(e_1(x), \dots, e_n(x))$ with $e_i$ $q$-orthonormal, we have
  \begin{equation*}
    N F(x) = (\tilde{N} \phi)(e_1(x), \dots, e_n(x)).
  \end{equation*}
\end{lemma}

\begin{proof}
  This follows from the formula $N = \delta D$ and (\ref{DF-cylinder})
  and Proposition \ref{delta-simple}, and the fact that $J : (W^*, q)
  \to H$ is an isometry.  Note for finite dimensions, if we take $e_1,
  \dots, e_n$ to be the coordinate functions on $\R^n$, this shows
  that $\tilde{N}$ really is the Ornstein--Uhlenbeck operator.
\end{proof}

\begin{theorem}
  The Ornstein--Uhlenbeck semigroup $T_t$ is given by
  \begin{equation}\label{Tt-formula}
    T_t F(x) = \int_W F\left(e^{-t} x + \sqrt{1-e^{-2t}} y\right)\,\mu(dy).
  \end{equation}
\end{theorem}

\begin{proof}
  Since this is mostly computation, I'll just sketch it.

  Let $R_t$ denote the right side.  We'll show that $R_t$ is another
  semigroup with the same generator.

  Showing that $R_t$ is a semigroup is easy once you remember that
  $\mu_t$ is a convolution semigroup, or in other words
  \begin{equation*}
    \int_W \int_W G(ax+by) \,\mu(dy)\,\mu(dx) = \int_W
    G\left(\sqrt{a^2+b^2} z\right) \,\mu(dz).
  \end{equation*}

  To check the generator is right, start with the finite dimensional
  case.  If $\phi$ is a nice smooth function on $\R^n$, and $p(y) dy$
  is standard Gaussian measure, then show that
  \begin{equation*}
    \frac{d}{dt}|_{t=0} \int_{\R^n} \phi\left(e^{-t} x + \sqrt{1-e^{-2t}}y\right)
    p(y) dy = \tilde{N} \phi(x).
  \end{equation*}
  (First differentiate under the integral sign.  For the term with the
  $x$, evaluate at $t=0$.  For the term with $y$, integrate by parts,
  remembering that $y p(y) = -\grad p(y)$.  If in doubt, assign it as
  homework to a Math 2220 class.)

  Now if $F$ is a smooth cylinder function on $W$, do the same and use
  the previous lemma, noting that $(e_1, \dots, e_n)$ have a standard
  normal distribution under $\mu$.

  There is probably some annoying density argument as the last step.
  The interested reader can work it out and let me know how it went.
\end{proof}

This shows that at time $t$, $X_t$ started at $x$ has a Gaussian
distribution (derived from $\mu$) with mean $e^{-t} x$ and variance
$1-e^{-2t}$.

Here is a general property of Markovian semigroups that we will use
later:

\begin{lemma}\label{Tt-cauchy-schwarz}
  For bounded nonnegative functions $F,G$, we have
  \begin{equation}\label{Tt-cauchy-schwarz-eqn}
    |T_t (F G)(x)|^2 \le T_t (F^2)(x) T_t (G^2)(x).
  \end{equation}
\end{lemma}

\begin{proof}
  Note the following identity: for $a, b \ge 0$,
  \begin{equation*}
    ab = \frac{1}{2} \inf_{r > 0} \left(r a^2 + \frac{1}{r} b^2\right). 
  \end{equation*}
  (One direction is the AM-GM inequality, and the other comes from
  taking $r = b/a$.)  So
  \begin{align*}
    T_t (FG) &= \frac{1}{2} T_t \left(\inf_{r > 0} \left(r F^2 +
    \frac{1}{r} G^2\right)\right) \\
    &\le \frac{1}{2} \inf_{r > 0} \left( r T_t (F^2) + \frac{1}{r} T_t
    (G^2) \right) \\
    &= \sqrt{T_t(F^2) T_t(G^2)}
  \end{align*}
  where in the second line we used the fact that $T_t$ is linear and
  Markovian (i.e. if $f \le g$ then $T_t f \le T_t g$).
\end{proof}

As a special case, taking $G=1$, we have $|T_t F(x)|^2 \le T_t (F^2)(x)$.

Alternative proof: use (\ref{Tt-formula}), or the fact that $T_t F(x)
= E_x[F(X_t)]$, and Cauchy--Schwarz.

\subsection{Log Sobolev inequality}

Recall that in finite dimensions, the classical Sobolev embedding
theorem says that for $\phi \in C^\infty_c(\R^n)$ (or more generally
$\phi \in W^{1,p}(\R^n)$),
\begin{equation}
  \norm{\phi}_{L^{p^*}(\R^n,m)} \le C_{n,p}(\norm{\phi}_{L^p(\R^n, m)} +
  \norm{\grad \phi}_{L^p(\R^n, m)})
\end{equation}
where $\frac{1}{p^*} = \frac{1}{p} - \frac{1}{n}$.  Note everything is
with respect to Lebesgue measure.
In particular, this says that if $\phi$ and $\grad \phi$ are both in
$L^p$, then the integrability of $\phi$ is actually better: we have
$\phi \in L^{p^*}$.  So
\begin{equation*}
  W^{1,p} \subset L^{p^*}
\end{equation*}
and the inclusion is continuous (actually, if the inclusion holds at
all it has to be continuous, by the closed graph theorem).

This theorem is useless in infinite dimensions in two different ways.
First, it involves Lebesgue measure, which doesn't exist.  Second,
when $n = \infty$ we get $p^* = p$ so the conclusion is a triviality
anyway.

In 1975, Len Gross discovered the logarithmic Sobolev inequality
\cite{gross75} which fixes both of these defects by using Gaussian
measure and being dimension-independent.  Thus it has a chance of
holding in infinite dimensions.  In fact, it does.

The log-Sobolev inequality says that in an abstract Wiener space, for
$F \in \mathbb{D}^{1,2}$ with
\begin{equation}\label{log-sobolev-eqn}
  \int |F|^2 \ln |F| \,d\mu \le \norm{F}_{L^2(W,\mu)}^2 \ln
  \norm{F}_{L^2(W,\mu)} + \norm{DF}_{L^2(W;H)}^2.
\end{equation}

If you are worried what happens for $F$ near 0:
\begin{exercise}
  $g(x) = x^2 \ln x$ is bounded below on $(0, \infty)$, and $g(x) \to
  0$ as $x \downarrow 0$.
\end{exercise}
So if we define ``$0^2 \ln 0 = 0$'', there is no concern about the
existence of the integral on the left side (however, what is not
obvious is that it is finite).  What's really of interest are the
places where $|F|$ is large, since then $|F|^2 \ln |F|$ is bigger than $|F|^2$.

It's worth noting that (\ref{log-sobolev-eqn}) also holds in finite
dimensions, but there are no dimension-dependent constants appearing
in it.

A concise way of stating the log Sobolev inequality is to say that
\begin{equation*}
  \mathbb{D}^{1,2} \subset L^2 \ln L
\end{equation*}
where $L^2 \ln L$, by analogy with $L^p$, represents the set of
measurable functions $F$ with $\int |F|^2 \ln |F| < \infty$.  This is
called  an Orlicz space; one can play this game to define $\phi(L)$
spaces for a variety of reasonable functions $\phi$.

Our proof of the log Sobolev inequality hinges on the following
completely innocuous looking commutation relation.

\begin{lemma}\label{Tt-D-commute}
  For $F \in \mathbb{D}^{1,2}$, $D T_t F = e^{-t} T_t DF$.
\end{lemma}

You may object that on the right side we are applying $T_t$, an
operator on the real-valued function space $L^2(W)$, to the $H$-valued
function $DF$.  Okay then: we can define $T_t$ on $L^2(W;H)$ in any of
the following ways:
\begin{enumerate}
\item Componentwise: $T_t u =
  \sum_i (T_t \inner{u(\cdot)}{h_i}_H)(x) h_i$ where $h_i$ is an
  orthonormal basis for $H$.
\item Via (\ref{Tt-formula}), replacing the Lebesgue integral with
  Bochner.
\item Via (\ref{Tt-chaos-eqn}): set $T_t u = \sum_{n=0}^\infty e^{-tn}
  J_n u$ where $J_n$ is orthogonal projection onto $\mathcal{H}_n(H)
  \subset L^2(W;H)$.
\end{enumerate}
\begin{exercise}
  Verify that these are all the same.  Also verify the inequality
  \begin{equation}\label{Tt-H-markov}
    \norm{T_t u(x)}_H \le T_t \norm{u}_H (x).
  \end{equation}
\end{exercise}

It's worth noting that for any $F \in L^2$, $T_t F \in
\mathbb{D}^{1,2}$.  This follows either from the spectral theorem, or
from the observation that for any $t$, the sequence $\{n e^{-2tn}\}$
is bounded, so $\sum_n n \norm{J_n T_t F}^2 = \sum_n n e^{2tn}
\norm{J_n F}^2 \le C \sum \norm{J_n F}^2 \le C \norm{F}^2$.  In fact,
more is true: we have $T_t F \in \dom N$, and indeed $T_t F \in \dom N^\infty$.

\begin{proof}[Proof of Lemma \ref{Tt-D-commute}]
  \begin{align*}
    D T_t F &= D \sum_{n=0}^\infty e^{-tn} J_n F \\
    &= \sum_{n=1}^\infty e^{-tn} D J_n F && \text{(recall $D J_0 =
    0$)} \\
    &= \sum_{n=1}^\infty e^{-tn} J_{n-1} DF \\
    &= \sum_{k=0}^\infty e^{-t(k+1)} J_k DF = e^{-t} T_t DF
  \end{align*}
  where we re-indexed by letting $k=n-1$.  We've extended to
  $L^2(W;H)$ some Wiener chaos identities that we only really proved
  for $L^2(W)$; as an exercise you can check the details.
\end{proof}

There's also an infinitesimal version of this commutation:

\begin{lemma}
  For $F \in \mathcal{F} C^\infty_c(W)$, $D N F = (N+1) DF$.
\end{lemma}

\begin{proof}
  Differentiate the previous lemma at $t=0$.  Or, use Wiener chaos expansion.
\end{proof}

\begin{exercise}  (Not necessarily very interesting)
  Characterize the set of $F$ for which the foregoing identity makes
  sense and is true.
\end{exercise}



We can now prove the log Sobolev inequality (\ref{log-sobolev-eqn}).
This proof is taken from \cite{ustunel2010} which actually contains
several proofs.

\begin{proof}
  First, let $F$ be a smooth cylinder function which is bounded
  above and bounded below away from 0: $0 < a \le F \le b < \infty$.
  Take $G = F^2$; $G$ has the same properties.  Note in particular
  that $G \in \dom N$.
  We have
  \begin{equation}\label{ls1}
    Q := 2 \left(\int F^2 \ln F d\mu - \norm{F}^2 \ln \norm{F}\right) =
    \int G \ln G d\mu - \int G d\mu \ln \int G d\mu. 
  \end{equation}
  and we want to bound this quantity $Q$ by $2 \norm{DF}_{L^2(W;H)}^2$.
  
  Note that for any $G \in L^2(W)$ we have $\lim_{t \to \infty} T_t G
  = J_0 G = \int G d\mu$.  (Use Lemma \ref{Tt-chaos} and monotone
  convergence.)  So we can think of $T_t G$ as a continuous function
  from $[0, \infty]$ to $L^2(W)$.  It is continuously differentiable on
  $(0,\infty)$ and has derivative $-N T_t G = -T_t N G$.  So define $A : [0,\infty] \to L^2(W)$ by $A(t) = (T_t G)\cdot(\ln T_t G)$ (noting
  that as $T_t$ is Markovian, $T_t G$ is bounded above and below, so
  $(T_t G)\cdot(\ln T_t G)$ is also bounded and hence in $L^2$).  Then 
  $Q = \int_W (A(0) - A(\infty)) d\mu$.  Since we want to use the
  fundamental  theorem of calculus, we use the chain rule to see that
  \begin{equation*}
    A'(t) = -(N T_t G) (1 + \ln T_t G).
  \end{equation*}
So by the fundamental theorem
  of calculus, we have
  \begin{align*}
    Q &= - \int_W \int_0^\infty A'(t)\,dt d\mu \\
    &= \int_W \int_0^\infty (N T_t G)(1 + \ln T_t G)\,dt\,d\mu.
  \end{align*}

  There are two integrals in this expression, so of course we want to
  interchange them.  To justify this, we note that $1 + \ln T_t G$ is
  bounded (since $0 < a^2 \le G \le b^2$ and $T_t$ is Markovian, we
  also have $a^2 \le T_t G \le b^2)$), and so it is enough to bound
  \begin{align*}
    \int_W \int_0^\infty |N T_t G| \,dt\,d\mu &= \int_0^\infty \norm{N
    T_t G}_{L^1(W,\mu)} dt \\
    &\le \int_0^\infty \norm{N T_t G}_{L^2(W,\mu)} dt
  \end{align*}
  since $\norm{\cdot}_{L^1} \le \norm{\cdot}_{L^2}$ over a probability
  measure (Cauchy--Schwarz or Jensen).
  Note that $N T_t G = T_t N G$ is continuous from $[0,\infty]$ to
  $L^2(W,\mu)$, so $\norm{N T_t G}_{L^2(W)}$ is continuous in $t$ and
  hence bounded on compact sets.  So we only have to worry about what
  happens for large $t$.  But Corollary \ref{Tt-exponential-decay}
  says that it decays exponentially, and so is integrable.  (Note that
  $\int NG d\mu = \inner{NG}{1}_{L^2(W)} = \inner{DG}{D1}_{L^2(W)} =
  0$.)

  So after applying Fubini's theorem, we get
  \begin{align*}
    Q &= \int_0^\infty \int_W (N T_t G)(1 + \ln T_t G)\,d\mu \,dt \\
      &= \int_0^\infty \inner{N T_t G}{1 + \ln T_t G}_{L^2(W)}\,dt.
  \end{align*}

  Now since $N = \delta D$ we have, using the chain rule,
  \begin{align*}
    \inner{N T_t G}{1 + \ln T_t G}_{L^2(W)} &= \inner{D T_t G}{\cancel{D1} + D
    \ln T_t G}_{L^2(W;H)} \\ 
    &= \inner{D T_t G}{\frac{D T_t G}{T_t G}}_{L^2(W;H)} \\
    &= \int_W \frac{1}{T_t G} \norm{D T_t G}_H^2 d\mu \\
    &= e^{-2t} \int_W \frac{1}{T_t G} \norm{T_t D G}_H^2 d\mu
  \end{align*}
  where we have just used the commutation $D T_t = e^{-t} T_t D$.

  Let's look at $\norm{T_t DG}_H^2$.  Noting that $DG = 2 F DF$, we have
  \begin{align*}
    \norm{T_t DG}_H^2 &\le (T_t \norm{DG}_H)^2 && \text{by
    (\ref{Tt-H-markov})} \\
      &= 4 (T_t (F \norm{DF}_H))^2 \\
      &\le 4 (T_t(F^2)) (T_t \norm{DF}_H^2) && \text{by (\ref{Tt-cauchy-schwarz-eqn})}.
  \end{align*}

  Thus we have reached
  \begin{align*}
     \int_W \frac{1}{T_t G} \norm{T_t D G}_H^2 d\mu \le 4 \int_W T_t \norm{DF}_H^2\,d\mu.
  \end{align*}
  But since $T_t$ is self-adjoint and $T_t 1 = 1$ (or, if you like,
  the fact that $T_t$ commutes with $J_0$, we have $\int_W T_t f d\mu
  = \int f d\mu$ for any $t$.  Thus $\int_W T_t \norm{DF}_H^2 d\mu =
  \int_W \norm{DF}_H^2 \,d\mu = \norm{DF}_{L^2(W;H)}^2$.  So we have
  \begin{align*}
    Q \le \left(4 \int_0^\infty e^{-2t}\,dt  \right) \norm{DF}_{L^2(W;H)}^2
  \end{align*}
  The parenthesized constant equals 2 (consult a Math 1120 student if
  in doubt).  This is what we wanted.

  To extend this to all $F \in \mathbb{D}^{1,2}$, we need some density
  arguments.  Suppose now that $F$ is a smooth cylinder function which
  is bounded, say $|F| \le M$.  Fix $\epsilon > 0$, and for each $n$
  let $\varphi_n \in C^\infty(\R)$ be a positive smooth function, such that:
  \begin{enumerate}
    \item $\varphi_n$ is bounded away from 0;
    \item $\varphi_n \le M$;
    \item $\varphi_n'| \le 1+\epsilon$;
    \item $\varphi_n(x) \to |x|$ pointwise on $[-M,M]$.
  \end{enumerate}
  Thus $\varphi_n(F)$ is a smooth cylinder function, bounded away from
  0 and bounded above, so it satisfies the log Sobolev inequality.
  Since $\varphi_n(F) \to |F|$ pointwise and boundedly, we have
  $\norm{\varphi_n(F)}_{L^2(W)} \to \norm{F}_{L^2(W)}$ by dominated
  convergence.  We also have, by the chain rule, $\norm{D
  \varphi_n(F)}_{L^2(W;H)} \le (1+\epsilon) \norm{DF}_{L^2(W;H)}$.
  Thus
  \begin{equation*}
    \limsup_{n \to \infty} \int_W \varphi_n(F)^2 \ln \varphi_n(F)\,d\mu \le
    \norm{F}^2 \ln \norm{F} + (1+\epsilon) \norm{DF}^2.
  \end{equation*}
  Now since $x^2 \ln x$ is continuous, we have $\varphi_n(F)^2 \ln
  \varphi_n(F) \to |F|^2 \ln |F|$ pointwise.  Since $x^2 \ln x$ is
  bounded below, Fatou's lemma gives
  \begin{equation*}
    \int_W |F|^2 \ln |F| \,d\mu \le \liminf_{n \to \infty} \int_W \varphi_n(F)^2 \ln \varphi_n(F)\,d\mu
  \end{equation*}
  and so this case is done after we send $\epsilon \to 0$.  (Dominated
  convergence could also have been used, which would give equality in
  the last line.)

  Finally, let $F \in \mathbb{D}^{1,2}$.  We can find a sequence of
  bounded cylinder functions $F_n$ such that $F_n \to F$ in $L^2(W)$
  and $DF_n \to DF$ in $L^2(W;H)$.  Passing to a subsequence, we can
  also assume that $F_n \to F$ $\mu$-a.e., and we use Fatou's lemma as
  before to see that the log Sobolev inequality holds in the limit.
\end{proof}

Note that we mostly just used properties that are true for any
Markovian semigroup $T_t$ that is conservative ($T_t 1 = 1$).  The
only exception was the commutation $D T_t = e^{-t} T_t D$.  In fact,
an inequality like $\norm{D T_t F}_H \le C(t) T_t \norm{DF}_H$ would have
been good enough, provided that $C(t)$ is appropriately integrable.  (One of the main results in my thesis was to prove
an inequality like this for a certain finite-dimensional Lie group, in
order to obtain a log-Sobolev inequality by precisely this method.)

Also, you might wonder: since the statement of the log-Sobolev
inequality only involved $D$ and $\mu$, why did we drag the
Ornstein--Uhlenbeck semigroup into it?  Really the only reason was the
fact that $T_\infty F = \int F d\mu$, which is just saying that $T_t$
is the semigroup of a Markov process whose distribution at a certain
time $t_0$ (we took $t_0=\infty$) is the measure $\mu$ we want to use.  If
we want to prove this theorem in finite dimensions, we could instead
use the heat semigroup $P_t$ (which is symmetric with respect to
Lebesgue measure) and take $t=1$, beak Brownian motion at time 1
also has a standard Gaussian distribution.

\section{Absolute continuity and smoothness of distributions}

This section will just hint at some of the very important applications
of Malliavin calculus to proving absolute continuity results.

When presented with a random variable (or random vector) $X$, a very
basic question is ``What is its distribution?'', i.e. what is
$\nu(A) := P(X \in A)$ for Borel sets $A$?  A more basic question is
``Does $X$ has a continuous distribution?'', i.e. is $\nu$
absolutely continuous to Lebesgue measure?  If so, it has a
Radon--Nikodym derivative $f \in L^1(m)$, which is a density function
for $X$.  It may happen that $f$ is continuous or $C^k$ or $C^\infty$,
in which case so much the better.

Given a Brownian motion $B_t$ or similar process, one can cook up lots
of complicated random variables whose distributions may be very hard
to work out.  For example:

\begin{itemize}
\item $X = f(B_t)$ for some fixed $t$ (this is not so hard)
\item $X = f(B_T)$ for some stopping time $T$
\item $X = \sup_{t \in [0,1]} B_t$
\item $X = \int_0^1 Y_t \,dB_t$
\item $X = Z_t$, where $Z$ is the solution to some SDE $dZ_t = f(Z_t) dB_t$.
\end{itemize}

Malliavin calculus gives us some tools to learn something about the
absolute continuity of such random variables, and the smoothness of
their densities.

Let $(W, H, \mu)$ be an abstract Wiener space.  A measurable function
$F : W \to \R$ is then a random variable, and we can ask about its
distribution.  If we're going to use Malliavin calculus, we'd better
concentrate on $F \in \mathbb{D}^{1,p}$.  An obvious obstruction to
absolute continuity would be if $F$ is constant on some set $A$ of
positive $\mu$-measure; in this case, as we have previously shown,
$DF = 0$ on $A$.  The following theorem says if we ensure that $DF$
doesn't vanish, then $F$ must be absolutely continuous.

\begin{theorem}
  Let $F \in \mathbb{D}^{1,1}$, and suppose that $DF$ is nonzero
  $\mu$-a.e.  Then the law of $F$ is absolutely continuous to Lebesgue
  measure.
\end{theorem}

\begin{proof}
  Let $\nu = \mu \circ F^{-1}$ be the law of $F$; our goal is to show
  $\nu \ll m$.

  By replacing $F$ with something like $\arctan(F)$, we can assume
  that $F$ is bounded; say $0 \le F \le 1$.  So we want to show that
  $\nu$ is absolutely continuous to Lebesgue measure $m$ on $[0,1]$.
  Let $A \subset [0,1]$ be Borel with $m(A) = 0$; we want to show
  $\nu(A) = 0$.

  Choose a sequence $g_n \in C^\infty([0,1])$
  such that $g_n \to 1_A$ $m+\nu$-a.e., and such that the $g_n$ are
  uniformly bounded (say $|g_n| \le 2$).  Set $\psi_n(t) = \int_0^t
  g_n(s) ds$.  Then $\psi_n \in C^\infty$, $|\psi_n| \le 2$, and
  $\psi_n \to 0$ pointwise (everywhere).  

  In particular $\psi_n(F) \to
  0$ $\mu$-a.e. (in fact everywhere), and thus also in $L^1(W, \mu)$
  by bounded convergence.  On the other hand, by the chain
  rule, $D \psi_n (F) = g_n(F) DF$.  Now since $g_n \to 1_A$
  $\nu$-a.e., we have $g_n(F) \to 1_A(F)$ $\mu$-a.e., and boundedly.
  Thus $g_n(F) DF \to 1_A(F) DF$ in $L^1(W; H)$.  Now $D$ is a closed
  operator, so we must have $1_A(F) DF = D0 = 0$.  But by assumption
  $DF \ne 0$ $\mu$-a.e., so we have to have $1_A F = 0$ $\mu$-a.e.,
  that is, $\nu(A) = 0$.
\end{proof}

So knowing that the derivative $DF$ ``never'' vanishes guarantees that
the law of $F$ has a density.  If $DF$ mostly stays away from zero in
the sense that $\norm{DF}_H^{-1} \in L^p(W)$ for some $p$, then this
gives more smoothness (e.g. differentiability) for the density.  See Nualart for precise
statements.

In higher dimensions, if we have a function $F = (F^1, \dots, F^n) : W
\to \R^n$, the object to look at is the ``Jacobian,'' the
matrix-valued function $\gamma_F : W \to \R^{n \times n}$ defined by
$\gamma_F(x)_{ij} = \inner{DF^i(x)}{DF^j(x)}_H$.  If $\gamma_F$ is
almost everywhere nonsingular, then the law of $F$ has a density.   If
we have $(\det \gamma_F)^{-1} \in L^p(W)$ for some $p$, then we get
more smoothness.

Here's another interesting fact.  Recall that the \textbf{support} of
a Borel measure $\nu$ on a topological space $\Omega$ is by definition
the set of all $x \in \Omega$ such that every neighborhood of $x$ has
nonzero $\nu$ measure.  This set is closed.

\begin{proposition}
  If $F \in \mathbb{D}^{1,2}$, then the support of the law of $F$ is
  connected, i.e. is a closed interval in $\R$.
\end{proposition}

\begin{proof}
  Let $\nu = \mu \circ F^{-1}$.  Suppose $\supp \nu$ is not
  connected.  Then there exists $a \in \R$ such that there are points
  of $\supp \nu$ to the left and right of $a$.  Since $\supp \nu$ is
  closed, there is an open interval $(a,b)$ in the complement of
  $\supp \nu$.  That is, we have $\mu(a < F < b) = 0$ but $0 < \mu(F
  \le a) < 1$.  Let $\psi \in C^\infty(\R)$ have $\psi(t) = 1$ for $t
  \le a$ and $\psi(t) = 0$ for $t \ge b$, and moreover take $\psi$ and
  all its derivatives to be bounded.  Then $\psi(F) = 1_{(-\infty,
  a]}(F) = 1_{\{F \le a\}}$.  Since $\psi$ is smooth, $1_{\{F \le a\}}
  = \psi(F) \in
  \mathbb{D}^{1,2}$ by the chain rule (Lemma \ref{chain-rule}).  By
  the zero-one law of Proposition \ref{zero-one-law}, $\mu(F \le a)$
  is either 0 or 1, a contradiction.
\end{proof}

As an example, let's look at the maximum of a continuous process.

Let $(W,H,\mu)$ be an abstract Wiener space.  Suppose we have a
process $\{X_t : t \in [0,1]\}$ defined on $W$, i.e. a measurable map
$X : [0,1] \times W \to \R$, which is a.s. continuous in $t$.  (If we
take $W = C([0,1])$ and $\mu$ the law of some continuous Gaussian
process $Y_t$, then $X_t = Y_t$, in other words
$X_t(\omega)=\omega(t)$, would be an example.  Another natural example
would be to take classical Wiener space and let $X_t$ be the solution
of some SDE.)  Let $M = \sup_{t \in [0,1]} X_t$.  We will show that
under certain conditions, $M$ has an absolutely continuous law.

(Note you can also index $\{X_t\}$ by any other compact metric space
$S$ and the below proofs will go through just fine.  If you take $S$
finite, the results are trivial.  You can take $S = [0,1]^2$ and prove
things about Brownian sheet.  You can even take $S$ to be Cantor space
if you really want (hi Clinton!).)

\begin{lemma}
  Suppose $F_n \in \mathbb{D}^{1,2}$, $F_n \to F$ in $L^2(W)$, and
  $\sup_n \norm{DF_n}_{L^2(W;H)} < \infty$.  Then $F \in
  \mathbb{D}^{1,2}$ and $D F_n \to DF$ weakly in $L^2(W;H)$.
\end{lemma}

\begin{proof}
  This is really a general fact about closed operators on Hilbert
  space.  Since $\{DF_n\}$ is a bounded sequence in $L^2(W;H)$, by
  Alaoglu's theorem we can pass to a subsequence and assume that
  $DF_n$ converges weakly in $L^2(W;H)$, to some element $u$.  Suppose
  $v \in \dom \delta$.  Then $\inner{DF_n}{v}_{L^2(W;H)} =
  \inner{F_n}{\delta v}_{L^2(W)}$.  The left side converges to
  $\inner{u}{v}_{L^2(W;H)}$ and the right side to $\inner{F}{\delta
  v}_{L^2(W)}$.  Since the left side is continuous in $v$, we have $F
  \in \dom \delta^* = \dom D = \mathbb{D}^{1,2}$.  Moreover, since we
  have $\inner{D F_n}{v} \to \inner{DF}{v}$ for all $v$ in a dense
  subset of $L^2(W;H)$, and $\{D F_n\}$ is bounded, it follows from
  the triangle inequality that $D F_n \to DF$ weakly.  Since we get
  the same limit no matter which weakly convergent subsequence we
  passed to, it must be that the original sequence $D F_n$ also
  converges weakly to $DF$.
\end{proof}

Recall, as we've previously argued, that if $F \in \mathbb{D}^{1,2}$,
then $|F| \in \mathbb{D}^{1,2}$ also, and $\norm{D|F|}_H \le
\norm{DF}_H$ a.e.  (Approximate $|t|$ by smooth
functions with uniformly bounded derivatives.)  It follows that if
$F_1, F_2 \in \mathbb{D}^{1,2}$, then $F_1 \wedge F_2$, $F_1 \vee F_2
\in \mathbb{D}^{1,2}$ also.  ($F_1 \wedge F_2 = F_1 + F_2 - |F_1 -
F_2|$, and $F_1 \vee F_2 = F_1 + F_2 + |F_1 - F_2|$.)  Then by
iteration, if $F_1, \dots, F_n \in \mathbb{D}^{1,2}$, then $\min_k
F_k, \max_k F_k \in \mathbb{D}^{1,2}$ as well.

\begin{lemma}\label{Md12}
  Suppose $X, M$ are as above, and:
  \begin{enumerate}
  \item $\int_W \sup_{t \in [0,1]} |X_t(\omega)|^2 \,\mu(d\omega) <
    \infty$;
  \item For any $t \in [0,1]$, $X_t \in \mathbb{D}^{1,2}$;
  \item The $H$-valued process $DX_t$ has an a.s. continuous version
  (which we henceforth fix);
  \item $\int_W \sup_{t \in [0,1]} \norm{DX_t(\omega)}_H^2
  \,\mu(d\omega) < \infty$.
  \end{enumerate}
  Then $M \in \mathbb{D}^{1,2}$.
\end{lemma}

\begin{proof}
  The first property guarantees $M \in L^2(W)$.  Enumerate the
  rationals in $[0,1]$ as $\{q_n\}$.  Set $M_n = \max\{X_{q_1}, \dots,
  X_{q_n}\}$.  Then $M_n \in \mathbb{D}^{1,2}$ (using item 2).
  Clearly $M_n \uparrow M$ so by monotone convergence $M_n \to M$ in
  $L^2(W)$.  It suffices now to show that $\sup_n
  \norm{DM_n}_{L^2(W;H)} < \infty$.  Fix $n$, and for $k = 1, \dots,
  n$ let $A_k$ be the set of all $\omega$ where the maximum in $M_n$
  is achieved by $X_{q_k}$, with ties going to the smaller $k$.  That
  is,
  \begin{align*}
    A_1 &= \{ \omega : X_{q_1}(\omega) = M_n(\omega)\} \\ 
    A_2 &= \{ \omega : X_{q_1}(\omega) \ne M_n(\omega),
    X_{q_2}(\omega) = M_n(\omega)\} \\
    &\vdots \\
    A_n &= \{ \omega : X_{q_1}(\omega) \ne M_n(\omega), \dots,
    X_{q_{n-1}}(\omega) \ne M_n(\omega), X_{q_n}(\omega) = M_n(\omega)\}
  \end{align*}
  Clearly the $A_k$ are Borel and partition $W$, and $M_n = X_{q_k}$
  on $A_k$.  By the local property of $D$, we have $D M_n = D X_{q_k}$
  a.e. on $A_k$.  In particular, $\norm{D M_n}_H \le \sup_{t \in
  [0,1]} \norm{D X_t}_H$  a.e.  Squaring and integrating both sides,
  we are done by the last assumption.
\end{proof}

\begin{exercise}\label{gpex}
  Let $\{X_t, t \in [0,1]\}$ be a continuous centered Gaussian process.  Then
  we can take $W = C([0,1])$ (or a closed subspace thereof) and $\mu$
  to be the law of the process, and define $X_t$ on $W$ by
  $X_t(\omega) = \omega(t)$.  Verify that the hypotheses of
  Proposition \ref{Md12} are satisfied.
\end{exercise}

\begin{proposition}\label{M-ac}
  Suppose $X_t$ satisfies the hypotheses of the previous theorem, and
  moreover
  \begin{equation*}
    \mu(\{\omega : X_t(\omega) = M(\omega) \implies DX_t(\omega) \ne
    0\}) = 1.
  \end{equation*}
  (Note we are fixing continuous versions of $X_t$ and $DX_t$ so the
  above expression makes sense.)  Then $DM \ne 0$ a.e. and $M$ has an
  absolutely continuous law.
\end{proposition}

\begin{proof}
  It is enough to show
  \begin{equation*}
    \mu(\{\omega : X_t(\omega) = M(\omega) \implies DX_t(\omega) =
    DM(\omega)\}) = 1.
  \end{equation*}
  Call the above set $A$.  (Note that for every fixed $\omega$, $M(\omega) = X_t(\omega)$ for some
  $t$.)  

  Let $E$ be a countable dense subset of $H$.  
  For fixed $r,s \in \Q$, $h \in E$, $k > 0$, let
  \begin{equation*}
    G_{r,s,h,k} = \{ \omega : \sup_{t \in (r,s)} X_t(\omega) =
    M(\omega), {\inner{DX_t(\omega) - DM(\omega)}{h}_H} \ge
    \frac{1}{n} \text{ for all $r < t < s$} \}. 
  \end{equation*}
  Enumerate the rationals in $(r,s)$ as $\{q_i\}$.  If we let $M' =
  \sup_{t \in (r,s)} X_t$, $M_n' = \max \{ X_{q_1}, \dots, X_{q_n}\}$,
  then as we argued before, $M_n' \to M'$ in $L^2(W)$, and $D M_n' \to
  DM'$ weakly in $L^2(W;H)$.  On the other hand, by the local property
  used before, for every $\omega$ there is some $t_i$ with $D M_n' = D
  X_{t_i}$.  Thus for $\omega \in G_{r,s,h,k}$ we have
  ${\inner{DM'_n(\omega) - DM'(\omega)}{h}_H} \ge \frac{1}{n}$ for
  all $r < t < s$.  Integrating this inequality, we have $\inner{D
  M'_n - D M'}{h 1_{G_{r,s,h,k}}}_{L^2(W;H)} \ge \frac{1}{n} \mu(G_{r,s,h,k})$ for all
  $n$.  The left side goes to 0 by weak convergence, so it must be
  that $\mu(G_{r,s,h,k}) = 0$.

  However, $A^c = \bigcup G_{r,s,h,k}$ which is a countable union.
  (If $\omega \in A^c$, there exists $t$ such that $X_t(\omega) =
  M(\omega)$ but $DX_t(\omega) \ne DM(\omega)$.  As such, there must
  exist $h \in E$ with $\inner{DX_t(\omega) - DM(\omega)}{h}_H \ne 0$;
  by replacing $h$ by $-h$ or something very close to it, we can
  assume $\inner{DX_t(\omega) - DM(\omega)}{h}_H > 0$.  As $DX_t$ is
  assumed continuous, there exists $(r,s) \in \Q$ and $k > 0$ such
  that $\inner{DX_t(\omega) - DM(\omega)}{h}_H > \frac{1}{k}$ for all
  $t \in (r,s)$.  So we have $\omega \in G_{r,s,h,k}$.)
\end{proof}

\begin{exercise}
  Again let $X_t$ be a centered Gaussian process as in Exercise
  \ref{gpex} above.  Give an example of a process for which $M$ does
  not have an absolutely continuous law.  However, show that if
  $P(M=0)=0$, then the hypothesis of Proposition \ref{M-ac} is
  satisfied.  (Can we show this always holds whenever $X_t$ is strong Markov?)
\end{exercise}

\appendix

\section{Miscellaneous lemmas}

\begin{lemma}\label{ell2}
  Let $y \in \R^\infty$, and suppose that $\sum y(i) g(i)$ converges
  for every $g \in \ell^2$.  Then $y \in \ell^2$.
\end{lemma}

\begin{proof}
  For each $n$, let $H_n \in
  (\ell^2)^*$ be the bounded linear functional $H_n(g) = \sum_{i=1}^n
  y(i) g(i)$.  By assumption, for each $g \in \ell^2$, the sequence
  $\{H_n(g)\}$ converges; in particular $\sup_n |H_n(g)| < \infty$.
  So by the uniform boundedness principle, $\sup_n
  ||H_n||_{(\ell^2)^*} < \infty$.  But $||H_n||_{(\ell^2)^*}^2 =
  \sum_{i=1}^n |y(i)|^2$, so $\sum_{i=1}^\infty |y(i)|^2 = \sup_n
  ||H_n||_{(\ell^2)^*}^2 < \infty$ and $y \in \ell^2$.
\end{proof}

For an elementary, constructive proof, see also \cite{piau-ell2-argument}.

\begin{lemma}
  \label{dense-subspace-basis}
  Let $H$ be a separable Hilbert space and $E \subset H$ a dense
  subspace.  There exists an orthonormal basis $\{e_i\}$ for $H$ with
  $\{e_i\} \subset E$.
\end{lemma}

\begin{proof}
  Choose a sequence $\{x_i\} \subset E$ which is dense in $H$.  (To
  see that this is possible, let $\{y_k\}$ be a countable dense subset
  of $H$, and choose one $x_i$ inside each ball $B(y_k, 1/m)$.)  Then
  apply Gram-Schmidt to $x_i$ to get an orthonormal sequence $\{e_i\}
  \subset E$ with $x_n \in \spanop\{e_1, \dots, e_n\}$.  Then since
  $\{x_i\} \subset \spanop\{e_i\}$, $\spanop\{e_i\}$ is dense in $H$,
  so $\{e_i\}$ is an orthonormal basis for $H$.
\end{proof}

\begin{lemma}\label{limit-of-gaussian}
  Let $X_n \sim N(0, \sigma_n^2)$ be a sequence of mean-zero Gaussian random variables
  converging in distribution to a finite random variable $X$.  Then
  $X$ is also Gaussian, with mean zero and variance $\sigma^2 = \lim
  \sigma_n^2$ (and the limit exists).
\end{lemma}

\begin{proof}
  Suppose $\sigma_{n_k}^2$ is a subsequence of $\sigma_n^2$ converging
  to some $\sigma^2 \in [0,+\infty]$.  (By compactness, such a
  subsequence must exist.)  Now taking
  Fourier transforms, we have $e^{- \lambda^2 \sigma_n^2/2} = E[e^{i
  \lambda X_n}] \to E[e^{i \lambda X}]$ for each $X$, so $E[e^{i
  \lambda X}] = e^{-\lambda^2 \sigma^2/2}$.  Moreover, the Fourier
  transform of $X$ must be continuous and equal 1 at $\lambda = 0$,
  which rules out the case $\sigma^2 = +\infty$.  So $X \sim N(0,
  \sigma^2)$.  Since we get the same $\sigma^2$ no matter which
  convergent subsequence of $\sigma_n^2$ we start with, $\sigma_n^2$
  must converge to $\sigma^2$.
\end{proof}

\begin{lemma}\label{smooth-dense}
  Let $\mu$ be any finite Borel measure on $[0,1]$.  Then
  $C^\infty([0,1])$ is dense in $L^p([0,1], \mu)$.
\end{lemma}

\begin{proof}
  Use Dynkin's multiplicative system theorem.  Let $M$ consist of all
  $\mu$-versions of all bounded measurable functions in the closure of
  $C^\infty$ in $L^p(\mu)$.  Then $M$ is a vector space closed under
  bounded convergence (since bounded convergence implies $L^p(\mu)$
  convergence) and it contains $C^\infty([0,1])$.  By Dynkin's
  theorem, $M$ contains all bounded $\mathcal{F}$-measurable
  functions, where $\mathcal{F}$ is the smallest $\sigma$-algebra that
  makes all functions from $C^\infty([0,1])$ measurable.  But the
  identity function $f(x) = x$ is in $C^\infty$.  So for any Borel set
  $B$, we must have $B = f^{-1}(B) \in \mathcal{F}$.  Thus
  $\mathcal{F}$ is actually the Borel $\sigma$-algebra, and $M$
  contains all bounded measurable functions.  Since the bounded
  functions are certainly dense in $L^p(\mu)$ (by dominated
  convergence), we are done.
\end{proof}

\section{Radon measures}\label{radon}

\begin{definition}
  A finite Borel measure $\mu$ on a topological space $W$ is said to be
  \textbf{Radon} if for every Borel set $B$, we have
  \begin{equation}\label{inner-regular}
    \mu(B) = \sup\{\mu(K) : K \subset B, K \text{ compact}\}
  \end{equation}
  (we say that such a set $B$ is \textbf{inner regular}).
  Equivalently, $\mu$ is Radon if for every Borel set $B$ and every
  $\epsilon > 0$, there exists a compact $K \subset B$ with $\mu(B
  \backslash K) < \epsilon$.
\end{definition}

\begin{proposition}
  If $X$ is a compact metric space, every finite Borel measure on $X$ is
  Radon.
\end{proposition}

\begin{proof}
  Let $(X,d)$ be a compact metric space, and $\mu$ a Borel measure.
  Let $\mathcal{F}$ denote the collection of all Borel sets $B$ such
  that $B$ and $B^C$ are both inner regular.  I claim $\mathcal{F}$ is
  a $\sigma$-algebra.  Clearly $\emptyset \in
  \mathcal{F}$ and $\mathcal{F}$ is also closed under complements.
  If $B_1, B_2, \dots \in \mathcal{F}$ are disjoint, and $B =
  \bigcup_n B_n$ then since $\sum_n
  \mu(B_n) = \mu(B) < \infty$, there exists $n$ so large that
  $\sum_{n=N}^\infty \mu(B_n) < \epsilon$.  For $n = 1, \dots, N$,
  choose a compact $K_n \subset B_n$ with $\mu(B_n \backslash K_n) <
  \epsilon/N$.  Then if $K = \bigcup_{n=1}^N K_n$, $K$ is compact, $K
  \subset B$, and $\mu(B \backslash K) < 2 \epsilon$.  So $B$ is inner
  regular.

  Next, $\mathcal{F}$ contains all open sets $U$.  For any open set
  $U$ may be written as a countable union of compact sets $K_n$.  (For every
  $x \in U$ there is an open ball $B(x,r_x)$ contained in $U$, hence
  $\closure{B(x,r_x/2)} \subset U$ also.  Since $X$ is second
  countable we can find a basic open set $V_x$ with $x \in V_x \subset
  B(x, r_x/2)$, so $\closure{V_x} \subset U$.  Then $U = \bigcup_{x
  \in U} \closure{V_x}$.  But this union actually contains only
  countably many distinct sets.)  Thus by countable additivity, $U$ is
  inner regular.  $U^C$ is compact and so obviously inner regular.
  Thus $U \in \mathcal{F}$.  Since $\mathcal{F}$ is a $\sigma$-algebra
  and contains all open sets, it contains all Borel sets.
\end{proof}

\begin{proposition}
  Every complete separable metric space $(X,d)$ is homeomorphic to a Borel
  subset of the compact metric space $[0,1]^\infty$.
\end{proposition}

\begin{proof}
  Without loss of generality, assume $d \le 1$.  Fix a dense sequence
  $x_1, x_2, \dots$ in $X$ and for each $x \in X$, set $F(x) =
  (d(x,x_1), d(x,x_2), \dots) \in [0,1]^\infty$.  It is easy to check
  that $F$ is continuous.  $F$ is also injective: for any $x \in X$ we
  can choose a subsequence $x_{n_k} \to x$, so that $d(x_{n_k}, x) \to
  0$.  Then if $F(x) = F(y)$, then $d(x_n, x) = d(x_n, y)$ for all $n$,
  so $x_{n_k} \to y$ as well, and $x=y$.  Finally, $F$ has a
  continuous inverse.  Suppose $F(y_m) \to F(y)$.  Choose an $x_n$
  such that $d(x_n, y) < \epsilon$.  We have $F(y_m)_n = d(x_n, y_m)
  \to d(x_n, y) = F(y)_n$, so for sufficiently large $m$, $d(y_m, x_n)
  < \epsilon$, and by the triangle inequality $d(y_m, y) < 2\epsilon$.

  Lastly, we check $F(X)$ is Borel.  Well, this theorem is standard
  and I'm too lazy to write it out.  See, e.g. Srivastava's \emph{A
  course on Borel sets}, section 2.2.
\end{proof}

\begin{corollary}
  Any finite Borel measure $\mu$ on a complete separable metric space $X$ is Radon.
\end{corollary}

\begin{proof}
  Let $F$ be the above embedding of $X$ into $[0,1]^\infty$.  Then
  $\mu \circ F^{-1}$ defines a Borel measure on $F(X)$.  We can extend
  it to a Borel measure on $[0,1]^\infty$ by setting $\tilde{\mu}(B) = \mu(F^{-1}(B
  \cap F(X)))$, i.e. $\tilde{\mu}$ assigns measure zero to all sets
  outside $F(X)$.  Then we know that $\tilde{\mu}$ is Radon and hence
  so is $\mu$.  
\end{proof}

\begin{exercise}
  As a corollary of this, for any Borel probability measure on a
  Polish space, there is
  a sequence of compact sets $K_n$ such that $\mu(\bigcup K_n) = 1$.
  This is perhaps surprising because compact sets in an infinite
  dimensional Banach space are very thin; in particular they are
  nowhere dense.  For classical Wiener space with Wiener measure, find
  explicit sets $K_n$ with this property.  (Hint: Think of some
  well-known sample path properties of Brownian motion.)
\end{exercise}

\section{Miscellaneous Exercises}

\begin{exercise}\label{topologies-first}
  Let $X$ be a set, and let $\tau_s$ and $\tau_w$ be two topologies on
  $X$ such that $\tau_w \subset \tau_s$.  $\tau_w$ is said to be
  ``weaker'' or ``coarser,'' while $\tau_s$ is ``stronger'' or
  ``finer.''  

  Fill in the following chart.  Here $A \subset X$, and $Y,Z$ are some
  other topological spaces.  All terms such as ``more,'' ``less,''
  ``larger,'' ``smaller'' should be understood in the sense of
  implication or containment.  For instance, since every set which is
  open in $\tau_w$ is also open in $\tau_s$, we might say $\tau_s$ has
  ``more'' open sets and $\tau_w$ has ``fewer.''

  \begin{tabular}{|l|c|c|l|}
    \hline
    Property & $\tau_w$ & $\tau_s$ & Choices \\
    \hline \hline
    Open sets & \phantom{Fewer} &\phantom{Fewer} & More / fewer \\
    \hline
    Closed sets & & & More / fewer \\
    \hline
    Dense sets & & & More / fewer \\
    \hline
    Compact sets & & & More / fewer \\
    \hline
    Connected sets & & & More / fewer \\
    \hline
    Closure $\bar{A}$ & & & Larger / smaller \\
    \hline
    Interior $A^\circ$ & & & Larger / smaller \\
    \hline
    Precompact sets & & & More / fewer \\
    \hline
    Separable sets & & & More / fewer \\
    \hline
    Continuous functions $X \to Y$ & & & More / fewer \\
    \hline
    Continuous functions $Z \to X$ & & & More / fewer \\
    \hline
    Identity map continuous & \multicolumn{2}{|c|}{} & $(X, \tau_s)
    \to (X, \tau_w)$ or vice versa \\
    \hline
    Convergent sequences & & & More / fewer \\
    \hline
  \end{tabular}
  
  \end{exercise}

\begin{exercise} Now suppose that $X$ is a vector space, and $\tau_w \subset \tau_s$ are
  generated by two norms $\nm_w, \nm_s$.  Also let $Y,Z$ be other
  normed spaces.

  \begin{tabular}{|l|c|c|l|}
    \hline
    Property & $\nm_w$ & $\nm_s$ & Choices \\
    \hline \hline
    Size of norm & \phantom{Fewer} & \phantom{Fewer} & $\nm_s \le C \nm_w$ or vice versa, or neither \\ 
    \hline
    Closed (unbounded) operators $X \to Y$ & & & More / fewer \\
    \hline
    Closed (unbounded) operators $Z \to X$ & & & More / fewer \\
    \hline
    Cauchy sequences & & & More / fewer \\
    \hline
  \end{tabular}
  \end{exercise}
  
  \begin{exercise}
    Give an example where $X$ is complete in $\nm_s$ but not in
    $\nm_w$.
  \end{exercise}
  
  \begin{exercise}
    Give an example where $X$ is complete in $\nm_w$ but not in
    $\nm_s$.  (This exercise is ``abstract nonsense,'' i.e. it uses
    the axiom of choice.)
  \end{exercise}

\begin{exercise}
  If $X$ is complete in both $\nm_s$ and $\nm_w$, show that the
  two norms are equivalent, i.e. $c \nm_s \le \nm_w \le C \nm_s$ (and in particular $\tau_s = \tau_w$).
\end{exercise}

\begin{exercise}\label{topologies-last}
 In the previous problem, the assumption that $\tau_w \subset
    \tau_s$ was necessary.  Give an example of a vector space $X$ and
    complete norms $\nm_1$, $\nm_2$ which are not equivalent.
    (Abstract nonsense.)
\end{exercise}

\begin{exercise}\label{adjoint-exercise}
 Let $X,Y$ be Banach spaces with $X$ reflexive, $T : X \to Y$ a
  bounded operator, and $T^* : Y^* \to X^*$ its adjoint.
  \begin{enumerate}
  \item If $T$ is injective, then $T^*$ has dense range.
  \item If $T$ has dense range, then $T^*$ is injective.
  \end{enumerate}
\end{exercise}

\begin{exercise}
 For classical Wiener space $(W, \mu)$, find an explicit sequence
  of compact sets $K_n \subset W$ with $\mu\left(\bigcup_n K_n\right) = 1$.
\end{exercise}

\section{Questions for Nate}
\begin{enumerate}
\item Is a Gaussian Borel measure on a separable Banach space always
  Radon?  (Yes, a finite Borel measure on a Polish space is always
  Radon.  See Bogachev Theorem A.3.11.)

\item Compute the Cameron-Martin space $H$ for various continuous
  Gaussian processes (Ornstein--Uhlenbeck, fractional Brownian
  motion).

\item Why should Brownian motion ``live'' in the space $C([0,1])$
  instead of the smaller H\"older space $C^{0,\alpha}([0,1])$ for
  $\alpha < 1/2$?

\item What's the relationship between Brownian motion on classical
  Wiener space and various other 2-parameter Gaussian processes
  (e.g. Brownian sheet)?  (Compute covariances.)
\end{enumerate}

\bibliographystyle{plainnat}
\bibliography{allpapers}

\end{document}